\documentclass[12pt]{amsart}
\usepackage{pifont}
\usepackage{amsfonts}
\usepackage{amsmath}
\usepackage{amssymb}
\usepackage{amsmath,bm}
\usepackage{cite}
\usepackage{color}
\newtheorem{theorem}{Theorem}[section]

\newtheorem{lemma}[theorem]{Lemma}

\theoremstyle{definition}

\newtheorem{remark}{Remark}[section]

\begin{document}

\title{Spreading speeds of KPP-type lattice systems in heterogeneous media}

\author{Xing Liang}
\author {Tao Zhou}
\address{School of Mathematical Sciences and Wu Wen-Tsun Key Laboratory of Mathematics, University of Science and Technology of China, Hefei, Anhui 230026, China}
\date{18th July, 2017 }
\maketitle

\begin{abstract}
In this paper, we investigate spreading properties of the solutions of the  Kolmogorov-Petrovsky-Piskunov-type, (to be simple,KPP-type) lattice system
\begin{equation}\label{firstequation}\overset{.}u_{i}(t) =d^{\prime}_{i}(u_{i+1}(t)-u_{i}(t))+d_{i}(u_{i-1}(t)-u_{i}(t))+f(i,u_{i}).\end{equation}
Motivated by the work in\cite{B2},
we develop some new discrete Harnack-type estimates and homogenization techniques for the lattice system \eqref{firstequation} to construct two speeds $\overline\omega \leq \underline \omega$ such that
$\displaystyle{\lim_{t\rightarrow+\infty}}\sup \limits_{i\geq\omega t}|u_i(t)|=0$ for any $\omega>\overline{\omega}$, and
$\displaystyle{\lim_{t\rightarrow+\infty}}\sup \limits_{0\leq i\leq\omega t}|u_i(t)-1|=0$ for any $\omega<\underline{\omega}$.
These speeds are characterized by two generalized principal eigenvalues of the linearized systems of \eqref{firstequation}.
In particular, we derive the exact spreading speed when the coefficients are random stationary ergodic or almost periodic (where $\underline\omega = \overline\omega$).
Finally, in the case where $f_{s}^{\prime}(i,0)$ is almost periodic in $i$ and the diffusion rate $d_i'=d_i$ is independent of $i$,
we show that the spreading speeds in the positive and negative directions are identical even if  $ f(i,u_{i})$ is not invariant with respect to the reflection. \end{abstract}
\section{Introduction}

In this paper, we focus on the large time behavior of the solution of the following problem:
\begin{equation}\label{1.1}
  \left\{
   \begin{aligned}
   \overset{.}u_{i}(t) =d^{\prime}_{i}(u_{i+1}(t)-u_{i}(t))+d_{i}(u_{i-1}(t)-u_{i}(t))+f(i,u_{i})\ \ t>0,i\in\mathbf{Z},  \\
   0\leq u(0,i)\leq1,\{i:u(0,i)\neq0\}\neq\emptyset\ \text{has finite elements,} \\
   \end{aligned}
   \right.
  \end{equation}
where we assume that the coefficient $d_{i}$ is uniformly bounded in $i$ with $\inf\limits_{i}d_{i}>0$,
and that $f$ satisfies some KPP-type conditions. This will be told in detail later. A simple
example is $f(i,s)=s(1-s)$.
It is known that \eqref{1.1} is a spatial-discrete version of
\begin{equation}
 \label{continuousequation} \left\{
   \begin{aligned}
   \partial_{t}u=a(x)\partial_{xx}u+q(x)\partial_{x}u+f(x,u)\ \ t>0,x\in\mathbf{R}, \\
   0\leq u(0,x)\leq1,\{x:u(0,x)\neq0\}\neq\emptyset\ \ \text{is bounded} .\\
   \end{aligned}
   \right.
  \end{equation}
The pioneer works on the dynamics of the type of equations like \eqref{1.1} and \eqref{continuousequation} were done by Fisher \cite{F1} and Kolmogorov, Petrovsky, Piskunov \cite{K1} in the homogeneous case:
$$\partial_{t}u=\partial_{xx}u+f(u),$$
where $f\in\mathcal{C}^{1}[0,1]$, $f(0)=f(1)=0$. In fact, in \cite{F1,K1}, they proved the existence of the minimal wave speed in the case where $f(s)>0$
and $f^{\prime}(s)\leq f^{\prime}(0)s$ for any $s\in(0,1)$.
 Moreover, in the homogeneous case,
Aronson and Weinberger \cite{A1} proved that if $f^{\prime}(0)>0$ and $f(s)>0$ for any $s\in(0,1)$, then there exists
$\omega^{\ast}>0$ such that
\begin{equation*}
  \left\{
   \begin{aligned}
   &\text{for any}\ \ \ \ \omega>\omega^{\ast},&\displaystyle{\lim_{t\rightarrow\infty}}\sup \limits_{x\geq\omega t}|u(t,x)|=0,\\
   &\text{for all}\ \ \omega\in(0,\omega^{\ast}),\ \ &\displaystyle{\lim_{t\rightarrow\infty}}\sup \limits_{0\leq x<\omega t}|u(t,x)-1|=0.\\
   \end{aligned}
   \right.
  \end{equation*}
A similar result still holds if $x\leq0$. An easy corollary is
$\displaystyle{\lim_{t\rightarrow\infty}}u(t,x+\omega t)=0$ locally uniform in $x\in\mathbf{R}$ if $\omega>\omega^{\ast}$ and
$\displaystyle{\lim_{t\rightarrow\infty}}u(t,x+\omega t)=1$ locally uniform in $x\in\mathbf{R}$ if $0\leq\omega<\omega^{\ast}$.
These results are called spreading properties and $\omega^{\ast}$ is called
the spreading speed.

In the past decades, the spreading properties of \eqref{1.1} and \eqref{continuousequation} in heterogeneous media got increasing attention of mathematicians. The propagation problems in (spatially) periodic
media, the simplest heterogenous case, were widely considered by mathematicians. Applying the approach of
probability,  \cite{GF} first proved the existence of spreading speeds for one-dimensional
KPP-type reaction-diffusion equations in periodic media. \cite{SKT, X} gave the definition of the spatially
periodic traveling waves independently, and then \cite{HZ} proved the existence of
the spatially periodic traveling waves of KPP-type equations in the distributional
sense. In a series of works (e.g.\cite{BH, BHR, BHN}), Berestycki, Hamel and their colleagues investigated
the traveling waves and spreading speeds of KPP-type reaction-diffusion equations
in high-dimensional periodic media.

Besides above works, more general
frameworks are provided by \cite{LZ,W} to study spreading properties for more general diffusion systems in periodic media.

Though  \eqref{firstequation} is just a system of ordinary differential equations, comparing with the reaction-diffusion equation  \eqref{continuousequation}, the study of  \eqref{firstequation} has its own difficulties.
In fact, even in the case where the media is homogeneous, the system \eqref{firstequation} only has the invariance with respect to the action of the spatial translation group $\mathbf Z$,  which is a discrete subgroup of $\mathbf R$.
In this sense, the homogeneous lattice system is essentially a system with spatially periodic heterogeneity (e.g. see \cite{LZ}).
Besides, in the study of reaction-diffusion equations, the Harnack-type estimates and methods of integration by parts are very powerful tools (e.g. \cite{BH, BHR, BHN,BHN2,Na} ).
In the case where the media is discrete, these techniques should be rebuilt and developed.  Related results
on the spreading speeds and traveling waves of lattice systems in homogeneous media can be found in \cite{CLW,CCV,CFG,MWZ,MZ,Mallet, WLW}, and references therein.
There are also some works considering the spreading properties of the lattice system \eqref{firstequation} in periodic media.
\cite{G1,HZ } proved the existence of the traveling waves of \eqref{firstequation} with KPP-type nonlinearity and in periodic media by different methods, and then \cite{LZ} proved the existence of the spreading speeds.

However, there are only a few works on the spreading properties
of KPP-type equations in more complicated media.
Berestycki, Hamel and Nadirashvili \cite{BHN2}  investigated spreading properties in higher dimension for the homogeneous equation in general unbounded domains.
Particularly, in \cite{BHN2} , the concepts of lower and upper spreading speeds were introduced. Then Berestycki and Nadin \cite{B2} also
introduced these two speeds again for \eqref{continuousequation} to study the spreading properties. Precisely, for one-dimensional equation \eqref{continuousequation},
the lower and upper spreading speeds are defined by$$\omega^{\ast}:=\inf\{\omega\geq0,\ \displaystyle{\lim_{t\rightarrow\infty}}\sup \limits_{x\geq\omega t}|u(t,x)|=0\},$$
$$\omega_{\ast}:=\sup\{\omega\geq0,\ \displaystyle{\lim_{t\rightarrow\infty}}\sup \limits_{0\leq x<\omega t}|u(t,x)-1|=0\}.$$
They gave a sharp estimate on $\omega_{\ast},\ \omega^{\ast}$ by constructing
$\underline{\omega},\ \overline{\omega}$, where $\underline{\omega},\ \overline{\omega}$ are represented by two generalized principal eigenvalues (see Definition \ref{def2.1}) of the linearized equation of \eqref{firstequation} such that
$$\underline{\omega}\leq\omega_{\ast}\leq\omega^{\ast}\leq\overline{\omega}.$$
Furthermore, they showed that if the coefficients are (asymptotically) almost periodic or random stationary ergodic, then $\underline{\omega}=\overline{\omega}$,
and hence $\omega_{\ast}=\omega^{\ast}$ is exactly the spreading speed. Most  recently, they also investigated the multidimensional
and space-time heterogeneous case in \cite{B3}.  In fact, Shen (see e.g.\cite{s1,s2,s3}) also introduced the concepts of lower and upper spreading speeds to study the spreading speeds of KPP-type equations in space-time heterogeneous media.

In this paper, we investigate the spreading properties for \eqref{1.1} in general heterogeneous media.
Motivated by \cite{B2}, we establish the theory of generalized principal eigenvalues of linear lattice systems to estimate the lower and upper spreading speeds $\omega_{\ast}, \omega^{\ast}$.
Aiming to estimate the spreading speeds via the generalized principal eigenvalues, we also develop some new discrete Harnack-type inequalities, and homogenization techniques for lattice equations.
Then we  prove that $\omega_{\ast}=\omega^{\ast}$  in the case where the media is almost periodic  or random stationary ergodic.
Finally, in the case where $f_{s}^{\prime}(i,0)$ is almost periodic in $i$ and the diffusion rate $d_i'=d_i$ is independent of $i$,
we show that the spreading speeds in the positive and negative directions are identical even if  $ f(i,u_{i})$ is not invariant with respect to the reflection.
Moreover, such a conclusion still holds for the reaction-diffusion equation \eqref{continuousequation} in corresponding conditions.
Here, we would like to point out that the last conclusion is far from being obvious even in the case where the media is periodic.
\cite{LZ} first noticed such a phenomenon while considering a so-called linear determined reaction-diffusion equation and applying one classical conclusion that  a linear operator and its adjoint operator have identical real spectral sets.
Based on a similar idea, \cite{DL,JZ,s1} proved the same conclusion about the invariance of the spreading speeds with respect to reflection for different systems with linear determined property.
In this paper,  for the systems in almost periodic media,
we show this conclusion by considering the generalized principal eigenvalues of the ``formal" adjoint operators and giving a limit estimate based on the discrete integration by parts.


\section{Preliminary: Definitions, notions, results}
First, let $H\subseteq\mathbf{Z}$. For any function $a:H\rightarrow\mathbf{R}$, we denote $a_{i}:=a(i)$,
$i\in H$. In this paper, we use both $a_{i}$ and $a(i)$ for convenience.
Considering the problem \eqref{1.1}, we assume that
$0<\inf \limits_{i}d_{i}\leq\sup \limits_{i}d_{i}<+\infty$,
$0<\inf \limits_{i}d^{\prime}_{i}\leq\sup \limits_{i}d^{\prime}_{i}<+\infty$,
$f(i,0)\equiv f(i,1)\equiv0$,\ $0<\inf\limits_{i}f(i,s)\leq f(i,s)\leq f_{s}^{\prime}(i,0)s$ for any $s\in (0,1)$ and
$f(i,\cdot)\in\mathcal{C}^{1+\gamma}([0,1])$ uniformly with respect to $i\in\mathbf{Z}$,
that is, $\sup\limits_{i}\|f(i,\cdot)\|_{\mathcal{C}^{1+\gamma}}<+\infty$.
Specially, we also assume that \\
\begin{equation}\label{2.1}
\liminf \limits_{|i|\to \infty}(f_{s}^{\prime}(i,0)-({\sqrt {d^{\prime}_{i}}}-{\sqrt {d_{i}}})^2)>0.
\end{equation}
Denote $X_{n}:=\{a|\ a:[n,+\infty)\cap\mathbf{Z}\rightarrow\mathbf{R}\}$ and
$X_{-\infty}:=\{a|\ a:\mathbf{Z}\rightarrow\mathbf{R}\}$. We define
$A:X_{n}\rightarrow X_{n+1}$ by
$(A\phi)_{i}=d^{\prime}_{i}(\phi_{i+1}-\phi_{i})+d_{i}(\phi_{i-1}-\phi_{i})$,
$\mathcal L:X_{n}\rightarrow X_{n+1}$ by $(\mathcal L\phi)_{i}=(A\phi)_{i}+f_{s}^{\prime}(i,0){\phi}_{i},$
and $L_{p}:X_{n}\rightarrow X_{n+1}$ by
$(L_{p}\phi)_{i}={e}^{-pi}\mathcal (\mathcal{L}e^{p\cdot}\phi)_{i}.$ Moreover, we denote
$D=\max\{\sup\limits_{i} d_{i},\sup\limits_{i} d^{\prime}_{i}\},\ \underline{D}=\min\{\inf\limits_{i} d_{i},\inf\limits_{i} d^{\prime}_{i}\}$.
We can also consider a more general case with heterogenous steady states $p^{-}$ and
$p^{+}$ instead of $0$ and $1$ under some assumptions corresponding to those we have given before.
In fact, under the condition
$0<\inf \limits_{i}(p^{+}_{i}-p^{-}_{i})\leq\sup \limits_{i}(p^{+}_{i}-p^{-}_{i})<+\infty$, we can reduce the equation with heterogeneous steady states into an equation with steady states $0$ and $1$ by setting
$v=\frac{u-p^{-}}{p^{+}-p^{-}}$. Hence we may, without loss of generality, assume that $p^{+}=1$ and $p^{-}=0$
as long as $0<\inf \limits_{i}(p^{+}_{i}-p^{-}_{i})\leq\sup \limits_{i}(p^{+}_{i}-p^{-}_{i})<+\infty$.
\newtheorem{defn}{Definition}[section]
\begin{defn}\label{def2.1}
The generalized principal eigenvalues associated with operator $L_{p}$ on $I_{n}:=(n,+\infty)\cap \mathbf{Z}$, where $n\in\{-\infty\}\cup\mathbf{Z}$, are
$$\underline{{\lambda}_{1}}(p,n):=\sup\{\lambda|\ \exists\ \phi\in{\mathcal{A}}_{n},\ \text{such that}\
(L_{p}\phi)_{i}\geq \lambda\phi_{i}\ \text{for any}\ i\in I_{n}\},$$
$$\overline{{\lambda}_{1}}(p,n):=\inf\{\lambda|\ \exists\ \phi\in{\mathcal{A}}_{n},\ \text{such that}\
(L_{p}\phi)_{i}\leq \lambda\phi_{i}\ \text{for any}\ i\in I_{n}\},$$
where for $n\in\mathbf{Z}$, ${\mathcal{A}}_{n}$ is the set of admissible test functions:\\
${\mathcal{A}}_{n}:=\{\phi\in X_{n}|\ {\phi}_{i}>0\ for\ i\geq n,
\displaystyle{\lim_{i\rightarrow +\infty}}\frac{\ln{{\phi}_{i}}}{i}=0,
{\Big\{\frac{{\phi}_{i+1}-{\phi}_{i}}{{\phi}_{i}}\Big\}}_{i=n}^{\infty}\in{\ell}^{\infty},\\
{\Big\{\frac{{\phi}_{i+1}-{\phi}_{i}}{{\phi}_{i+1}}\Big\}}_{i=n}^{\infty}\in{\ell}^{\infty}\}$\\
and ${\mathcal{A}}_{-\infty}$ is the set of admissible test functions:
$${\mathcal{A}}_{-\infty}:=\{\phi\in X_{-\infty}|\ {\phi}_{i}>0,
\displaystyle{\lim_{|i|\rightarrow +\infty}}\frac{\ln{{\phi}_{i}}}{i}=0,
{\Big\{\frac{{\phi}_{i\pm1}-{\phi}_{i}}{{\phi}_{i}}\Big\}}_{i=-\infty}^{\infty}\in{\ell}^{\infty}\}.$$
\end{defn}
In some cases, we write the generalized principal eigenvalues which are related to $L_{p}$ as
$\underline{{\lambda}_{1}}(p,n,\mathcal L)$ and $\overline{{\lambda}_{1}}(p,n,\mathcal L)$ to emphasize that they depend on $\mathcal{L}$.
It is easy to see that $\underline{{\lambda}_{1}}(p,n)$ is increasing in $n$, and
$\overline{{\lambda}_{1}}(p,n)$ is decreasing in $n$. Furthermore, we have:

\newtheorem{prop}{Proposition}[section]
\begin{prop}\label{prop2.1}
Let $n\in\{-\infty\}\cup\mathbf{Z}$. Then
$$\underline{{\lambda}_{1}}(p,n)\leq\overline{{\lambda}_{1}}(p,n)\ \forall p\in\mathbf{R}.$$

\end{prop}
This proposition and Definition \ref{def2.1} yield the following corollary immediately.
\newtheorem{cor}{Corollary}[section]
\begin{cor}\label{cor2.1}
For any $n\in\{-\infty\}\cup\mathbf{Z}$ and $p\in\mathbf{R}$, if there exist $\lambda\in\mathbf{R}$ and $\phi\in\mathcal{A}_{n}$
such that $(L_{p}\phi)_{i}=\lambda\phi_{i}$, then
$$\lambda=\underline{{\lambda}_{1}}(p,n)=\overline{{\lambda}_{1}}(p,n).$$
\end{cor}
Let us define
\begin{equation}\label{2.2}
\overline{H}(p):=\displaystyle{\lim_{n\rightarrow +\infty}\overline{{\lambda}_{1}}(p,n)},\ \text{and}\ \ \underline{H}(p):=\displaystyle{\lim_{n\rightarrow +\infty}\underline{{\lambda}_{1}}(p,n)},\ \forall \ p\in\mathbf{R}.
\end{equation}

$\overline{H}(p)$ and $\underline{H}(p)$ are well defined by Proposition \ref{prop2.1} and the monotonicity of $\underline{{\lambda}_{1}}(p,n)$ and $\overline{{\lambda}_{1}}(p,n)$. Moreover, we have the following proposition:

\begin{prop}\label{prop2.2}
The functions $\overline{H}$ and $\underline{H}$ are locally Lipschitz continuous. Moreover, there exist constants ${\varepsilon}_{0}, {a}_{0}>0$, and ${a}_{1}<2$ such that
$${\varepsilon}_{0}<\underline{H}(p)\leq\overline{H}(p)\leq{a}_{0}(e^{p}+{e}^{-p}-{a}_{1})\ \forall p\in\mathbf{R},\ \text{and}\
\displaystyle{\lim_{p\rightarrow{0}^{+}}}\frac{\underline{H}(\pm p)}{p}=+\infty.$$
\end{prop}

Now, as in \cite{B2} we can define the speeds $\underline{\omega}$ and $\overline{\omega}$:
\begin{equation}\label{2.3}
\underline{\omega}:=\min \limits_{p>0}\frac{\underline{H}(-p)}{p},\  \text{and}\
\overline{\omega}:=\min \limits_{p>0}\frac{\overline{H}(-p)}{p}.
\end{equation}
The main result of this paper is as follows:
\newtheorem{thm}{Theorem}[section]
\begin{thm}\label{thm2.1}
Let $u(t,i)$ be a solution of \eqref{1.1}. Then\\
1) For any $\omega>\overline{\omega}$, $\displaystyle{\lim_{t\rightarrow+\infty}}\sup \limits_{i\geq\omega t}|u(t,i)|=0$;\\
2) For any $0\leq\omega<\underline{\omega}$, $\displaystyle{\lim_{t\rightarrow+\infty}}\sup \limits_{0\leq i\leq\omega t}|u(t,i)-1|=0.$
\end{thm}

\section{Proof of propositions in section 2}

In this section, the proof of Propositions \ref{prop2.1} and \ref{prop2.2} will be given. We will also
provide another proposition about the generalized principal eigenvalues and then give its proof. 
In fact, we consider the operator $\mathcal{L}$ in the general form:
$$(\mathcal{L}\phi)_{i}:=d^{\prime}_{i}(\phi_{i+1}-\phi_{i})+d_{i}(\phi_{i-1}-\phi_{i})+{c}_{i}{\phi}_{i},$$
i.e., we replace $f_{s}^{\prime}(i,0)$ by
$c={\{c_{i}\}}_{-\infty}^{\infty}\in{\ell}^{\infty}(\mathbf{Z})$. We denote $C:=\sup\limits_{i} c_{i}$.

\begin{proof}[Proof of Proposition \ref{prop2.1}]
We only give the proof in case there $n\in\mathbf{Z}$.
The proof in the case where $n=-\infty$ is similar. The proof of the former case includes two steps.\\
Step 1: $p=0$.
We may assume that $n=0$ without loss of generality by translation. Assume by contradiction that
$\underline{{\lambda}_{1}}(0,0)>\overline{{\lambda}_{1}}(0,0)$. Then there exist $\lambda\in\mathbf{R}$,
$\varepsilon>0$ and $\phi,\psi\in{\mathcal{A}}_{0}$ such that
$\underline{{\lambda}_{1}}(0,0)>\lambda>\lambda-2\varepsilon>\overline{{\lambda}_{1}}(0,0)$
and
$$\left\{
   \begin{aligned}
d^{\prime}_{i}\frac{{\phi}_{i+1}}{{\phi}_{i}}+{d}_{i}\frac{{\phi}_{i-1}}{{\phi}_{i}}
     -(d^{\prime}_{i}+{d}_{i})+{c}_{i}\leq\lambda-2\varepsilon,\ \ \forall i\in{I}_{0},\\
d^{\prime}_{i}\frac{{\psi}_{i+1}}{{\psi}_{i}}+{d}_{i}\frac{{\psi}_{i-1}}{{\psi}_{i}}
     -(d^{\prime}_{i}+{d}_{i})+{c}_{i}\geq\lambda,\ \ \forall i\in{I}_{0}.\\
   \end{aligned}\
   \right.$$
This yields
\begin{equation}\label{3.1}
d^{\prime}_{i}(\frac{{\psi}_{i+1}}{{\psi}_{i}}-\frac{{\phi}_{i+1}}{{\phi}_{i}})+
{d}_{i}(\frac{{\psi}_{i-1}}{{\psi}_{i}}-\frac{{\phi}_{i-1}}{{\phi}_{i}})\geq 2\varepsilon,\ \forall i\in{I}_{0}.
\end{equation}
We have the following two claims:\\
Claim 1. If there exists ${i}_{0}\in{I}_{0}$ such that
$d^{\prime}_{{i}_{0}}(\frac{{\psi}_{{i}_{0}+1}}{{\psi}_{{i}_{0}}}-\frac{{\phi}_{{i}_{0}+1}}{{\phi}_{{i}_{0}}})
\geq\varepsilon$,
then $d^{\prime}_{i}(\frac{{\psi}_{i+1}}{{\psi}_{i}}-\frac{{\phi}_{i+1}}{{\phi}_{i}})>\varepsilon$ for any
$i>{i}_{0}$.\\
Proof of Claim 1: We only need to show
$d^{\prime}_{{i}_{0}+1}(\frac{{\psi}_{{i}_{0}+2}}{{\psi}_{{i}_{0}+1}}-\frac{{\phi}_{{i}_{0}+2}}{{\phi}_{{i}_{0}+1}})
>\varepsilon$. If not, by \eqref{3.1}, we have
${d}_{{i}_{0}+1}(\frac{{\psi}_{{i}_{0}}}{{\psi}_{{i}_{0}+1}}-\frac{{\phi}_{{i}_{0}}}{{\phi}_{{i}_{0}+1}})
\geq\varepsilon$. Hence
$$1=\frac{{\psi}_{{i}_{0}+1}}{{\psi}_{{i}_{0}}}\cdot\frac{{\psi}_{{i}_{0}}}{{\psi}_{{i}_{0}+1}}
\geq(\frac{{\phi}_{{i}_{0}+1}}{{\phi}_{{i}_{0}}}+\frac{\varepsilon}{d^{\prime}_{{i}_{0}}})\cdot
(\frac{{\phi}_{{i}_{0}}}{{\phi}_{{i}_{0}+1}}+\frac{\varepsilon}{d_{{i}_{0}+1}})
>1+\frac{{\varepsilon}^{2}}{d^{\prime}_{{i}_{0}}d_{{i}_{0}+1}}.$$
This is impossible so that the claim is valid.\\
Claim 2. There exists $\delta>0$ such that $\delta<\frac{{\phi}_{i+1}}{{\phi}_{i}}<\frac{1}{\delta}$
for any $i\in{I}_{0}$.\\
Proof of Claim 2: We only need to show that there exists $\delta>0$ such that
$\liminf \limits_{i\to \infty}\frac{{\phi}_{i\pm 1}}{{\phi}_{i}}\geq\delta$.
If not, then for any $k$, there exists ${i}_{k}$ such that
$\frac{{\phi}_{i_{k}\pm 1}}{{\phi}_{i_{k}}}<\frac{1}{k+1}$, which means that
$\frac{{\phi}_{{i}_{k}}-{\phi}_{{i}_{k}\pm 1}}{{\phi}_{{i}_{k}\pm 1}}\geq k$, and this contradicts
$\phi\in{\mathcal{A}}_{0}$. Hence the claim holds.

Now we turn to the proof of the proposition. From Claim 1 there are two cases we need to consider.\\
Case 1: $d^{\prime}_{i}(\frac{{\psi}_{i+1}}{{\psi}_{i}}-\frac{{\phi}_{i+1}}{{\phi}_{i}})<\varepsilon$
for any $i\in{I}_{0}$. Then
${d}_{i}(\frac{{\psi}_{i-1}}{{\psi}_{i}}-\frac{{\phi}_{i-1}}{{\phi}_{i}})>\varepsilon$ for any
$i\in{I}_{0}$ since \eqref{3.1},
i.e., $\frac{{\psi}_{i-1}}{{\psi}_{i}}>\frac{{\phi}_{i-1}}{{\phi}_{i}}+\frac{\varepsilon}{{d}_{i}}$.
Then by Claim 2, we get
$$\frac{{\psi}_{0}}{{\psi}_{k}}>\prod_{i=1}^{k}(\frac{{\phi}_{i-1}}{{\phi}_{i}}+\frac{\varepsilon}{{d}_{i}})
=\frac{{\phi}_{0}}{{\phi}_{k}}\prod_{i=1}^{k}(1+\frac{{\phi}_{i}}{{\phi}_{i-1}}\cdot\frac{\varepsilon}{{d}_{i}})
\geq\frac{{\phi}_{0}}{{\phi}_{k}}{(1+\frac{\varepsilon\delta}{D})}^{k},$$
which yields
\begin{equation}\label{3.2}
\frac{\ln{\psi}_{0}}{k}-\frac{\ln{\psi}_{k}}{k}>
\frac{\ln{\phi}_{0}}{k}-\frac{\ln{\phi}_{k}}{k}+\ln(1+\frac{\varepsilon\delta}{D}).
\end{equation}
Noting that $\phi,\psi\in{\mathcal{A}}_{0}$, and setting $k\rightarrow\infty$ in \eqref{3.2},
we have $0\geq\ln(1+\frac{\varepsilon\delta}{D})>0$, which is a contradiction!\\
Case 2: There exists ${i}_{0}\in{I}_{0}$ such that
$d^{\prime}_{i_{0}}(\frac{{\psi}_{i_{0}+1}}{{\psi}_{i_{0}}}-\frac{{\phi}_{i_{0}+1}}{{\phi}_{i_{0}}})\geq\varepsilon$.
Then
$\frac{{\psi}_{i+1}}{{\psi}_{i}}\geq\frac{{\phi}_{i+1}}{{\phi}_{i}}+\frac{\varepsilon}{d^{\prime}_{i}}$
for any $i>{i}_{0}$ by Claim 1. Combining this with Claim 2, we have
$$\frac{{\psi}_{k+1}}{{\psi}_{{i}_{0}}}\geq\prod_{i={i}_{0}}^{k}
  (\frac{{\phi}_{i+1}}{{\phi}_{i}}+\frac{\varepsilon}{d^{\prime}_{i}})
  =\frac{{\phi}_{k+1}}{{\phi}_{{i}_{0}}}\prod_{i={i}_{0}}^{k}(1+\frac{{\phi}_{i+1}}{{\phi}_{i}}
  \cdot\frac{\varepsilon}{d^{\prime}_{i}})\geq
  \frac{{\phi}_{k+1}}{{\phi}_{{i}_{0}}}{(1+\frac{\varepsilon\delta}{D})}^{k-{i}_{0}+1}.$$
Then
\begin{equation}\label{3.3}
\frac{\ln{\psi}_{k+1}}{k+1}-\frac{\ln{\psi}_{{i}_{0}}}{k+1}>
 \frac{\ln{\phi}_{k+1}}{k+1}-\frac{\ln{\phi}_{{i}_{0}}}{k+1}+
 \frac{k-{i}_{0}+1}{k+1}\ln(1+\frac{\varepsilon\delta}{D}).
\end{equation}
Taking $k\rightarrow\infty$ in \eqref{3.3}, we have $0\geq\ln(1+\frac{\varepsilon\delta}{D})>0$,
which is a contradiction.
Thus it must be $\underline{{\lambda}_{1}}(0,0)\leq\overline{{\lambda}_{1}}(0,0).$\\
Step 2: $p\neq0$.
Setting ${\phi}_{i}^{(p)}\in X_{n}$ with ${\phi}_{i}^{(p)}={e}^{pi}{\phi}_{i}$,
we note that $\phi\in{\mathcal{A}}_{0}$. Then
$\displaystyle{\lim_{i\rightarrow \infty}}\frac{\ln{{\phi}_{i}^{(p)}}}{i}=p,
{\Big\{\frac{{\phi}_{i+1}^{(p)}-{\phi}_{i}^{(p)}}{{\phi}_{i}^{(p)}}\Big\}}_{i=0}^{\infty}\in{\ell}^{\infty},
{\Big\{\frac{{\phi}_{i+1}^{(p)}-{\phi}_{i}^{(p)}}{{\phi}_{i+1}^{(p)}}\Big\}}_{i=0}^{\infty}\in{\ell}^{\infty}$.
For $\psi$, we have the same conclusion. By the same argument as in Step 1, we still have \eqref{3.2} and
\eqref{3.3} by replacing $\phi,\psi$ with ${\phi}^{(p)},{\psi}^{(p)}$.
Then taking $k\rightarrow\infty$, we still obtain contradictions.
Thus the conclusion that $\underline{{\lambda}_{1}}(p,0)\leq\overline{{\lambda}_{1}}(p,0)$ is proved.
\end{proof}


Next we will prove that $\overline{H}(p)$ and $\underline{H}(p)$ are locally Lipschitz continuous by showing
that $\underline{{\lambda}_{1}}(p,n)$ and $\overline{{\lambda}_{1}}(p,n)$ are locally Lipschitz continuous with
respect to $p$ uniformly in $n$.

\begin{lemma}\label{lem3.1}
Let $\alpha,\beta\in(0,1)$ with $\alpha+\beta=1$ and $\Phi,\Psi\in X_{n}, n\in\{-\infty\}\cup\mathbf{Z}$, with ${\Phi}_{i},{\Psi}_{i}>0$ for any $i\in I_{n-1}$. Then
 \begin{equation}\label{l3.1.1}
  \alpha\frac{(A\Phi)_{i}}{{\Phi}_{i}}+\beta\frac{(A\Psi)_{i}}{{\Psi}_{i}}
  \geq\frac{(A{\Phi}^{\alpha}{\Psi}^{\beta})_{i}}{{\Phi}_{i}^{\alpha}{\Psi}_{i}^{\beta}},\ \forall i\in{I}_{0}.
 \end{equation}
\end{lemma}

\begin{proof}
It is sufficient to show
$$\alpha\frac{d^{\prime}_{i}{\Phi}_{i+1}+d_{i}{\Phi}_{i-1}}{{\Phi}_{i}}+\beta\frac{d^{\prime}_{i}{\Psi}_{i+1}+d_{i}{\Psi}_{i-1}}{{\Psi}_{i}}
\geq\frac{d^{\prime}_{i}{\Phi}_{i+1}^{\alpha}{\Psi}_{i+1}^{\beta}+d_{i}{\Phi}_{i-1}^{\alpha}{\Psi}_{i-1}^{\beta}}
{{\Phi}_{i}^{\alpha}{\Psi}_{i}^{\beta}}.$$
That is,
$$\alpha(d^{\prime}_{i}{\Phi}_{i+1}+d_{i}{\Phi}_{i-1}){\Psi}_{i}+\beta(d^{\prime}_{i}{\Psi}_{i+1}+d_{i}{\Psi}_{i-1}){\Phi}_{i}$$
$$\geq d^{\prime}_{i}{\Phi}_{i+1}^{\alpha}{\Psi}_{i+1}^{\beta}{\Phi}_{i}^{1-\alpha}{\Psi}_{i}^{1-\beta}+
{d}_{i}{\Phi}_{i-1}^{\alpha}{\Psi}_{i-1}^{\beta}{\Phi}_{i}^{1-\alpha}{\Psi}_{i}^{1-\beta},$$
i.e.,\\
$$\alpha d^{\prime}_{i}{\Phi}_{i+1}{\Psi}_{i}+\alpha{d}_{i}{\Phi}_{i-1}{\Psi}_{i}+
\beta d^{\prime}_{i}{\Psi}_{i+1}{\Phi}_{i}+\beta{d}_{i}{\Psi}_{i-1}{\Phi}_{i}$$
$$\geq{(d^{\prime}_{i}{\Phi}_{i+1}{\Psi}_{i})}^{\alpha}{(d^{\prime}_{i}{\Psi}_{i+1}{\Phi}_{i})}^{\beta}+
{({d}_{i}{\Phi}_{i-1}{\Psi}_{i})}^{\alpha}{({d}_{i}{\Psi}_{i-1}{\Phi}_{i})}^{\beta}.$$
This is true by Young's inequality that\\
${(d^{\prime}_{i}{\Phi}_{i+1}{\Psi}_{i})}^{\alpha}{(d^{\prime}_{i}{\Phi}_{i+1}{\Psi}_{i})}^{\beta}\leq
\alpha{[{(d^{\prime}_{i}{\Phi}_{i+1}{\Psi}_{i})}^{\alpha}]}^{\frac{1}{{\alpha}}}+
\beta{[{(d^{\prime}_{i}{\Phi}_{i+1}{\Psi}_{i})}^{\beta}]}^{\frac{1}{{\beta}}}$
$$=\alpha d^{\prime}_{i}{\Phi}_{i+1}{\Psi}_{i}+\beta d^{\prime}_{i}{\Psi}_{i+1}{\Phi}_{i}$$
and\\
${({d}_{i}{\Phi}_{i-1}{\Psi}_{i})}^{\alpha}{({d}_{i}{\Phi}_{i-1}{\Psi}_{i})}^{\beta}\leq
\alpha{[{({d}_{i}{\Phi}_{i-1}{\Psi}_{i})}^{\alpha}]}^{\frac{1}{{\alpha}}}+
\beta{[{({d}_{i}{\Phi}_{i-1}{\Psi}_{i})}^{\beta}]}^{\frac{1}{{\beta}}}$
$$=\alpha{d}_{i}{\Phi}_{i-1}{\Psi}_{i}+\beta{d}_{i}{\Psi}_{i-1}{\Phi}_{i}.$$
\end{proof}
\begin{lemma}\label{lem3.2}
For any $n\in\{-\infty\}\cup\mathbf{Z}$, $\overline{{\lambda}_{1}}(p,n)$ is convex with respect to $p$.
\end{lemma}
\begin{proof}
Let $\Phi\in\mathcal{A}_{n}$ with
${\Phi}_{i}={e}^{{p}_{1}i}{\phi}_{i},$ and $\Psi\in \mathcal{A}_{n}$ with
${\Psi}_{i}={e}^{{p}_{2}i}{\psi}_{i}$. Then by Lemma \ref{lem3.1} we get
$$\alpha\frac{(\mathcal{L}\Phi)_{i}}{{e}^{{p}_{1}i}{\phi}_{i}}+
  \beta\frac{(\mathcal{L}\Psi)_{i}}{{e}^{{p}_{2}i}{\psi}_{i}}
  \geq\frac{(\mathcal{L}\Phi^{\alpha}\Psi^{\beta})_{i}}
  {{e}^{(\alpha{p_{1}}+\beta{p}_{2})i}{\phi}_{i}^{\alpha}{\psi}_{i}^{\beta}},$$
i.e.,
$$\alpha\frac{(L_{{p}_{1}}{\phi})_{i}}{{\phi}_{i}}+\beta\frac{(L_{{p}_{2}}{\psi})_{i}}{{\psi}_{i}}
\geq\frac{(L_{\alpha{p}_{1}+\beta{p}_{2}}{\phi}^{\alpha}{\psi}^{\beta})_{i}}{{\phi}_{i}^{\alpha}{\psi}_{i}^{\beta}}.$$
The definition of $\overline{{\lambda}_{1}}(p,n)$ yields
$$\alpha\overline{{\lambda}_{1}}({p}_{1},n)+\beta\overline{{\lambda}_{1}}({p}_{2},n)
\geq\overline{{\lambda}_{1}}(\alpha{p}_{1}+\beta{p}_{2},n).$$
\end{proof}
\begin{lemma}
$\underline{{\lambda}_{1}}(p,n)$ and $\overline{{\lambda}_{1}}(p,n)$ are locally Lipschitz continuous in $p$
uniformly with respect to $n\in\{-\infty\}\cup\mathbf{Z}$.
\end{lemma}
\begin{proof}
We may, without loss of generality, assume that ${p}_{1}<{p}_{2}$ and that
$\underline{{\lambda}_{1}}({p}_{j},n)$ and $\overline{{\lambda}_{1}}({p}_{j},n),(j=1,2),$ are positive
since $\underline{{\lambda}_{1}}(p,n,\mathcal{L}+M)=\underline{{\lambda}_{1}}(p,n,\mathcal L)+M$ for any constant $M$.
Let $\alpha=\frac{{p}_{2}-{p}_{1}}{1+{p}_{2}-{p}_{1}},$ and
$\beta=\frac{1}{1+{p}_{2}-{p}_{1}}$.

First we prove that $\underline{{\lambda}_{1}}(p,n)$ is locally Lipschitz continuous in $p$ uniformly
with respect to $n\in\{-\infty\}\cup\mathbf{Z}$. Taking $\varepsilon>0$, there exists
$\phi\in{\mathcal{A}_{n}}$ such that
$$L_{{p}_{1}}\phi\geq(\underline{{\lambda}_{1}}({p}_{1},n)-\varepsilon)\phi\ \text{on}\ I_{n},$$
i.e.,
$$(\frac{(A{e}^{{p}_{1}\cdot}\phi)_{i}}{{e}^{{p}_{1}i}\phi_{i}}+{c}_{i})
\geq\underline{{\lambda}_{1}}({p}_{1},n)-\varepsilon\ \forall i\in I_{n}.$$
Let ${\Phi}_{i}={e}^{({p}_{1}-1)i},{\Psi}_{i}={e}^{{p}_{2}i}{\phi}_{i}^{\eta}$, where $\eta=1+{p_{2}}-{p}_{1}$.
Then ${\Phi}_{i}^{\alpha}{\Psi}_{i}^{\beta}={e}^{{p}_{1}i}{\phi}_{i}$. By lemma \ref{lem3.1}
\begin{equation}\label{l3.3.1}
\begin{split}
\beta(\frac{(A{e}^{{p}_{2}\cdot}{\phi}^{\eta})_{i}}{{e}^{{p}_{2}i}{\phi}_{i}^{\eta}}+{c}_{i})
&\geq(\frac{(A{e}^{{p}_{1}\cdot}{\phi})_{i}}{{e}^{{p}_{1}i}{\phi}_{i}}+{c}_{i})-
\alpha(\frac{(A{e}^{({p}_{1}-1)\cdot})_{i}}{{e}^{({p}_{1}-1)i}}+{c}_{i}) \\
&\geq\underline{{\lambda}_{1}}({p}_{1},n)-\varepsilon-
 \alpha(d^{\prime}_{i}{e}^{({p}_{1}-1)}+d_{i}{e}^{(1-{p}_{1})}-d^{\prime}_{i}-d_{i}+c_{i}) \\
 &\geq\underline{{\lambda}_{1}}({p}_{1},n)-\varepsilon-
 \alpha(D{e}^{({p}_{1}-1)}+D{e}^{(1-{p}_{1})}-2d+C).
\end{split}
\end{equation}
\\
Noting that ${\phi}^{\eta}\in{\mathcal{A}_{n}}$ since $\phi\in{\mathcal{A}_{n}}$,
the Definition \ref{def2.1} yields that
$$\underline{{\lambda}_{1}}({p}_{2},n)\geq\beta\underline{{\lambda}_{1}}({p}_{2},n)
\geq\underline{{\lambda}_{1}}({p}_{1},n)-\varepsilon-
 \alpha(D{e}^{({p}_{1}-1)}+D{e}^{(1-{p}_{1})}-2d+C).$$
Taking $\varepsilon\to 0$, we have
\begin{equation}\label{l3.3.2}
\underline{{\lambda}_{1}}({p}_{2},n)\geq
\underline{{\lambda}_{1}}({p}_{1},n)-
 ({p}_{2}-{p}_{1})(D{e}^{({p}_{1}-1)}+D{e}^{(1-{p}_{1})}-2d+C).
\end{equation}
On the other hand, there exists $\psi\in{\mathcal{A}_{n}}$ such that $$L_{{p}_{2}}\psi\geq(\underline{{\lambda}_{1}}({p}_{2},n)-\varepsilon)\psi.$$
Let
${\Phi}_{i}={e}^{({p}_{2}+1)i},{\Psi}_{i}={e}^{{p}_{1}i}{\psi}_{i}^{\eta}$, where
$\eta=1+{p_{2}}-{p}_{1}$. Then ${\Phi}_{i}^{\alpha}{\Psi}_{i}^{\beta}={e}^{{p}_{2}i}\psi$. Therefore,
as what we just did before,
\begin{equation}\label{l3.3.3}
\underline{{\lambda}_{1}}({p}_{1},n)\geq
\underline{{\lambda}_{1}}({p}_{2},n)-
 ({p}_{2}-{p}_{1})(D{e}^{({p}_{2}+1)}+D{e}^{(-1-{p}_{2})}-2d+C).
\end{equation}
From \eqref{l3.3.2} and \eqref{l3.3.3}, we have
\begin{equation}
|\underline{{\lambda}_{1}}({p}_{1},n)-\underline{{\lambda}_{1}}({p}_{2},n)|\leq
 M|{p}_{2}-{p}_{1}|,
\end{equation}
where $M$ is a constant depending on ${p}_{1},{p}_{2}$ but not depending on $n$.

Next we will prove that $\overline{{\lambda}_{1}}(p,n)$ is locally Lipschitz continuous in $p$ uniformly
with respect to $n\in\mathbf{Z}$. Taking $\varepsilon>0$, there exist $\phi,\psi\in\mathcal{A}_{n}$ such that
$$L_{{p}_{1}}\phi\leq(\overline{{\lambda}_{1}}({p}_{1},n)+\varepsilon)\phi,$$
$$L_{{p}_{2}}\psi\leq(\overline{{\lambda}_{1}}({p}_{2},n)+\varepsilon)\psi.$$
Set ${\Phi}_{i}={e}^{({p}_{2}+1)i},{\Psi}_{i}={e}^{{p}_{1}i}{\phi}_{i}$. Then ${\Phi}_{i}^{\alpha}{\Psi}_{i}^{\beta}={e}^{{p}_{2}i}{\phi}_{i}^{\beta}$.
Similarly, by Lemma \ref{lem3.1}, we have
$$\overline{{\lambda}_{1}}({p}_{1},n)+ ({p}_{2}-{p}_{1})(D{e}^{({p}_{2}+1)}+D{e}^{(-1-{p}_{2})}-2d+C)
\geq\overline{{\lambda}_{1}}({p}_{2},n).$$
While setting ${\Phi}_{i}={e}^{({p}_{1}-1)i},{\Psi}_{i}={e}^{{p}_{2}i}{\psi}_{i}$, we have
${\Phi}_{i}^{\alpha}{\Psi}_{i}^{\beta}={e}^{{p}_{1}i}{\psi}_{i}^{\beta}$. Therefore,
$$\overline{{\lambda}_{1}}({p}_{2},n)+({p}_{2}-{p}_{1})(D{e}^{({p}_{1}-1)}+D{e}^{(1-{p}_{1})}-2d+C)
\geq\overline{{\lambda}_{1}}({p}_{1},n).$$
Thus we have
\begin{equation}
|\overline{{\lambda}_{1}}({p}_{1},n)-\overline{{\lambda}_{1}}({p}_{2},n)|\leq
M|{p}_{2}-{p}_{1}|,
\end{equation}
where $M$ is a constant as before. The proof is complete.
\end{proof}

Let us now turn to the proof of Proposition \ref{prop2.2}. We need the assumption
\begin{equation}\label{3.4}
\liminf \limits_{|i|\to \infty}({c}_{i}-({\sqrt {d^{\prime}_{i}}}-{\sqrt {d_{i}}})^2)>0.
\end{equation}

\begin{proof}[Proof of Proposition \ref{prop2.2}]
It is easy to see that $\overline{H}(p)$ and $\underline{H}(p)$ are locally Lipschitz continuous by Lemma \ref{lem3.1}. Take $\phi\equiv 1$ as a test function. Then we have
\begin{equation}\label{3.5}
\begin{split}
\inf\limits_{i\in{I}_{n}}(d^{\prime}_{i}{e}^{p}+d_{i}{e}^{-p}-d^{\prime}_{i}-d_{i}+c_{i})
&\leq\underline{{\lambda}_{1}}(p,n)\\
&\leq\overline{{\lambda}_{1}}(p,n)\\
&\leq\sup\limits_{i\in{I}_{n}}(d^{\prime}_{i}{e}^{p}+d_{i}{e}^{-p}-d^{\prime}_{i}-d_{i}+c_{i}).
\end{split}
\end{equation}
Let $h_{i}(p):=d^{\prime}_{i}{e}^{p}+d_{i}{e}^{-p}-d^{\prime}_{i}-d_{i}+c_{i}$. Then\\
$$h_{i}^{\prime}(p):=\frac{d}{dp}h_{i}(p)=d^{\prime}_{i}{e}^{p}-d_{i}{e}^{-p},\ \ h_{i}^{\prime\prime}:=\frac{d}{dp}h_{i}^{\prime}(p)=d^{\prime}_{i}{e}^{p}+d_{i}{e}^{-p}>0.$$
Hence ${h}_{i}^{\prime}$ is strictly increasing, and
${h}_{i}^{\prime}{(\frac{\ln{d}_{i}-\ln{{q}_{i+1}}}{2})}=0.$
That is to say, $h_{i}(p)$ reaches its minimum ${c}_{i}-({\sqrt {d^{\prime}_{i}}}-{\sqrt {d_{i}}})^2$
at $\frac{1}{2}(\ln{d}_{i}-\ln{d^{\prime}_{i}})$ for any $p\in\mathbf{R}$. Noting \eqref{3.4}, there exist
${N}_{0}\in\mathbf{Z}$ and ${\varepsilon}_{0}>0$
such that
$\inf\limits_{i\in{I}_{n}}(d^{\prime}_{i}{e}^{p}+d_{i}{e}^{-p}-d^{\prime}_{i}-d_{i}+c_{i})
 >{\varepsilon}_{0}$ for all $n\geq{N}_{0}.$
Letting $n\to +\infty$ in \eqref{3.5}, we have
$$\varepsilon_{0}<\inf\limits_{i\in{I}_{{N}_{0}}}
  (d^{\prime}_{i}{e}^{p}+d_{i}{e}^{-p}-d^{\prime}_{i}-d_{i}+c_{i})
\leq\underline{H}(p)\leq\overline{H}(p)
\leq D({e}^{p}+{e}^{-p})-2\underline{D}+C.$$
Hence
$\displaystyle{\lim_{p\rightarrow{0}^{+}}}\frac{\underline{H}(\pm p)}{p}
\geq\displaystyle{\lim_{p\rightarrow{0}^{+}}}\frac{{\varepsilon}_{0}}{p}=+\infty$ and
\begin{equation*}
\begin{split}
\displaystyle{\lim_{p\rightarrow+\infty}}\frac{\underline{H}(\pm p)}{p}
&\geq\displaystyle{\lim_{p\rightarrow+\infty}}
 \inf\limits_{i\in{I}_{{N}_{0}}}
 \frac{(d^{\prime}_{i}{e}^{\pm p}+d_{i}{e}^{\mp p}-d^{\prime}_{i}-d_{i}+c_{i})}{p}\\
&\geq\displaystyle{\lim_{p\rightarrow+\infty}}(\underline{D}\frac{{e}^{p}}{p}-\frac{2D}{p}-\frac{C}{p})\\
&=+\infty.
\end{split}
\end{equation*}
This concludes the proof by taking $a_{0}=D, a_{1}=\frac{2\underline{D}-C}{D}<2$.
\end{proof}

Next we will prove that the generalized principal eigenvalues are continuous with respect to the coefficient $c\in{\ell}^{\infty}$.
Define $\mathcal L^{\prime}:X_{n}\rightarrow X_{n+1}$ by
$(\mathcal L^{\prime}\phi)_{i}:=d^{\prime}_{i}(\phi_{i+1}-\phi_{i})+d_{i}(\phi_{i-1}-\phi_{i})+c^{\prime}_{i}{\phi}_{i}$,
and $L^{\prime}_{p}\phi:={e}^{-p\cdot}\mathcal L^{\prime}(e^{p\cdot}\phi).$
\begin{prop}
\begin{equation}\label{p3.1.1}
|\underline{{\lambda}_{1}}(p,n,\mathcal L^{\prime})-\underline{{\lambda}_{1}}(p,n,\mathcal L)|
\leq\|c^{\prime}-c\|_{l^{\infty}}\ \forall n\in\{-\infty\}\cup\mathbf{Z},\ p\in\mathbf{R},
\end{equation}
\begin{equation}\label{p3.1.2}
|\overline{{\lambda}_{1}}(p,n,\mathcal L^{\prime})-\overline{{\lambda}_{1}}(p,n,\mathcal L)|
\leq\|c^{\prime}-c\|_{l^{\infty}}\ \forall n\in\{-\infty\}\cup\mathbf{Z},\ p\in\mathbf{R}.
\end{equation}
\end{prop}
\begin{proof}
For any $\varepsilon>0$, there exists $\phi\in\mathcal{A}_{n}$ such that
$$(L_{p}\phi)_{i}\geq(\underline{{\lambda}_{1}}(p,n,\mathcal L)-\varepsilon)\phi_{i}\ \forall i\in I_{n}.$$
Hence
\begin{equation*}
\begin{split}
(L^{\prime}_{p}\phi)_{i}
&=(L_{p}\phi)_{i}+(c^{\prime}_{i}-c_{i})\phi_{i}\\
&\geq(\underline{{\lambda}_{1}}(p,n,\mathcal L)-\varepsilon)\phi_{i}+(c^{\prime}_{i}-c_{i})\phi_{i}\\
&\geq(\underline{{\lambda}_{1}}(p,n,\mathcal L)-\varepsilon-\|c^{\prime}-c\|_{l^{\infty}})\phi_{i}.
\end{split}
\end{equation*}
Letting $\varepsilon\to 0$, the above inequality yields
$$\underline{{\lambda}_{1}}(p,n,\mathcal L^{\prime})\geq\underline{{\lambda}_{1}}(p,n,\mathcal L)-\|c^{\prime}-c\|_{l^{\infty}}.$$

By the symmetry, one has
$$|\underline{{\lambda}_{1}}(p,n,\mathcal L^{\prime})-\underline{{\lambda}_{1}}(p,n,\mathcal L)|
\leq\|c^{\prime}-c\|_{l^{\infty}}.$$
A similar argument gives \eqref{p3.1.2}.
\end{proof}

\section{Proof of the spreading property}

\subsection{Maximum principles}
Before going any further, we first give some maximum principles which we will use later. Let $\ell_{1}(t)$,
$\ell_{2}(t)$ be two functions defined on $[t_{0},+\infty)$.
Assume that $\ell_{1}(t)$ is decreasing and
$\ell_{2}(t)$ is increasing, and that $[\ell_{2}(t)]-\ell_{1}(t)\geq 2$. Denote
$$S_{t}:=\{i\in\mathbf{Z}|\ \ell_{1}(t)\leq i\leq\ell_{2}(t)\},\ \Omega_{{t}_{0}}:=\{(t,i)|\ t\in[{t}_{0},+\infty),i\in S_{t}\},$$
$$\Gamma_{t_{0}}:=\{(t_{0},i)\in\Omega_{t_{0}}\}
\cup\{(t,i)\in\Omega_{t_{0}}|\ (t,i+1)\notin\Omega_{t_{0}}\ \text{or}\ (t,i-1)\notin\Omega_{t_{0}}\},$$
$$\Omega_{{t}_{0},i}:=\{t>0|\ (t,i)\in(\Omega_{{t}_{0}}\setminus\Gamma_{t_{0}})\}.$$
\begin{lemma}\label{lem4.1}
Assume that $z_{i}(t)$ is differentiable in $t\in\Omega_{{t}_{0},i}$ for any $i\in\mathbf{Z}$,
and that $z_{i}(t)$ satisfies
\begin{equation}\label{l4.1.1}
\left\{
   \begin{aligned}
 \overset{.}z_{i}(t)-(Az)_{i}-c_{i}z_{i}\geq0,\ & \ \ \Omega_{{t}_{0}}\setminus\Gamma_{t_{0}},\\
                                z\geq0,\ & \ \ \Gamma_{t_{0}}.\\
   \end{aligned}
   \right.
\end{equation}
Then $z\geq0$ in $\Omega_{{t}_{0}}$.
\end{lemma}
\begin{proof}
We first prove the the result by assuming $c_{i}<0$. Assume by contradiction that there exists
$(T,j)\in\Omega_{{t}_{0}}$ such that $z(T,j)<0$. Then $z$ reaches its minimum at some point, say $(s,k)$, over
$\Omega_{{t}_{0}}^{T}:=([t_{0},T]\times\mathbf{Z})\cap\Omega_{{t}_{0}}$. Obviously, $(s,k)\notin\Gamma_{t_{0}}$.
Hence at $(s,k)$, we have
$$\overset{.}z_{i}(t)\leq0,\ (Az)_{i}\leq0,\ c_{i}z_{i}>0,$$
i.e., $\overset{.}z_{i}(t)-(Az)_{i}-c_{i}z_{i}<0$, which contradicts \eqref{l4.1.1}!

For general $c_{i}$, set $\zeta_{i}(t)=z_{i}(t){e}^{-(C+1)t}$. Then $\zeta$ satisfies
$$\left\{
   \begin{aligned}
 \overset{.}\zeta_{i}(t)-(A\zeta)_{i}-(c_{i}-C-1)\zeta_{i}\geq0,\ & \ \ \Omega_{{t}_{0}}\setminus\Gamma_{t_{0}},\\
                                             \zeta\geq0,\ & \ \ \Gamma_{t_{0}}.\\
   \end{aligned}
   \right.$$
Therefore, $\xi\geq0$ in $\Omega_{t_{0}}$, which yields that $z\geq0$ in $\Omega_{t_{0}}$.
\end{proof}
From the proof of Lemma \ref{lem4.1}, one has
\begin{cor}\label{cor4.1}
Assume that $z_{i}(t)$ is differentiable in $t\in\Omega_{{t}_{0},i}$ for all $i\in\mathbf{Z}$,
and that $z_{i}(t)$ satisfies $\overset{.}z_{i}(t)-(Az)_{i}-c_{i}z_{i}\geq0$ in
$\Omega_{{t}_{0}}^{T}\setminus\Gamma_{t_{0}}$ for $T>t_{0}$. Then $z$ can't reach its negative minimum in $\Omega_{{t}_{0}}^{T}\setminus\Gamma_{t_{0}}$.
\end{cor}

Moreover, we have the following maximum principle in the whole space.
\begin{lemma}\label{lem4.2}
Assume that for any bounded interval $I\subset[0,+\infty)$, $z$ is bounded on $I\times\mathbf{Z}$, and that
$z_{i}(t)$ is differentiable in $t\in(0,+\infty)$ for any $i\in\mathbf{Z}$. Suppose that $z$ satisfies
$$\left\{
   \begin{aligned}
 \overset{.}z_{i}(t)-(Az)_{i}-c_{i}z_{i}\geq0,\ & \ \ \ (0,+\infty)\times\mathbf{Z},\\
                                     z\geq0,\ & \ \ \ t=0.\\
   \end{aligned}
   \right.$$
Then $z\geq0$ on $(0,+\infty)\times\mathbf{Z}$.
\end{lemma}
\begin{proof}
We may assume, without loss of generality, that $c_{i}\leq-1$. If the conclusion is not true, then there exists
$(\tau_{0},i_{0})$ such that $z(\tau_{0},i_{0})<0$ and exists $0<t_{0}\leq\tau_{0}$ such that
$z(t_{0},i_{0})=\min\limits_{t\in[0,\tau_{0}]}z(t,i_{0})$. Hence $\overset{.}z_{i_{0}}(t_{0})\leq0$.
Consider $\Omega_{0}=[0,t_{0}]\times\{i_{0},i_{0}\pm1\}$. By Corollary \ref{cor4.1}, there must be
$t\in(0,t_{0}],i\in\{i_{0}\pm1\}$ such that $z(t,i)<z(t_{0},i_{0})$. We divide the proof into
the following two cases.\\
Case 1: If $z(t_{1},i_{0}-1)=\min\limits_{\Omega_{0}}z(t,i)<z(t_{0},i_{0})<0$ for some
$t_{1}\in(0,t_{0}]$, then $\overset{.}z(t_{1},i_{0}-1)\leq0$. Consider $\Omega_{1}=[0,t_{1}]\times\{i_{0},i_{0}\pm1,i_{0}-2\}$. Still by Corollary \ref{cor4.1}, there must be $t_{2}\in(0,t_{1}]$ such that $z(t_{2},i_{0}-2)=\min\limits_{\Omega_{1}}z(t,i)<z(t_{1},i_{0}-1)<0$
and $\overset{.}z(t_{2},i_{0}-2)\leq0$. Then we obtain two sequences
$$0<\cdots\leq t_{n}\leq t_{n-1}\leq\cdots\leq t_{1}\leq t_{0}$$ and
$$0>z(t_{1},i_{0}-1)>\cdots>z(t_{n-1},i_{0}-(n-1))>z(t_{n},i_{0}-n)>\cdots$$
with $\overset{.}z(t_{n},i_{0}-n)\leq0$. By the boundedness of $z$,
$z_{i_{0}-n}(t_{n})$ converges as $n\to \infty$. Moreover,
\begin{equation*}
\begin{split}
z_{i_{0}-n+1}(t_{n})-z_{i_{0}-n}(t_{n})
&=z_{i_{0}-(n-1)}(t_{n})-z_{i_{0}-(n-1)}(t_{n-1})\\
 &\quad+z_{i_{0}-(n-1)}(t_{n-1})-z_{i_{0}-n}(t_{n})\\
&=\int_{t_{n-1}}^{t_{n}}\overset{.}z_{i_{0}-(n-1)}(t)dt+z_{i_{0}-(n-1)}(t_{n-1})-z_{i_{0}-n}(t_{n})\\
&\geq\int_{t_{n-1}}^{t_{n}}\big((Az)_{i_{0}-(n-1)}+c_{i_{0}-(n-1)}z_{i_{0}-(n-1)}\big)dt+z_{i_{0}-(n-1)}(t_{n-1})-z_{i_{0}-n}(t_{n})\\
&\geq-\int_{t_{n-1}}^{t_{n}}|(Az)_{i_{0}-(n-1)}+c_{i_{0}-(n-1)}z_{i_{0}-(n-1)}|dt-|z_{i_{0}-(n-1)}(t_{n-1})-z_{i_{0}-n}(t_{n})|\\
&\geq-(4D+\sup\limits_{i}|c_{i}|)\sup\limits_{t\in[0,t_{0}]}{\|z_{i}\|}_{{\ell}^{\infty}}(t_{n-1}-t_{n})\\
 &\quad-|z_{i_{0}-(n-1)}(t_{n-1})-z_{i_{0}-n}(t_{n})|.
\end{split}
\end{equation*}
Similarly,
\begin{equation}\label{l4.2.1}
\begin{split}
z_{i_{0}-n-1}(t_{n})-z_{i_{0}-n}(t_{n})
&\geq-(4D+\sup\limits_{i}|c_{i}|)\sup\limits_{t\in[0,t_{0}]}{\|z_{i}\|}_{{\ell}^{\infty}}(t_{n}-t_{n+1})\\
 &\quad-|z_{i_{0}-(n+1)}(t_{n+1})-z_{i_{0}-n}(t_{n})|.
\end{split}
\end{equation}
Then for any $n$, we have
$$0>z_{i_{0}}(t_{0})>-c_{i_{0}-n}z_{i_{0}-n}(t_{n})\geq (Az)_{i_{0}-n}-\overset{.}z(t_{n},i_{0}-n)
\geq (Az)_{i_{0}-n}$$
since $\overset{.}z(t_{n},i_{0}-n)\leq0$. Moreover,
\begin{equation*}
\begin{split}
(Az)_{i_{0}-n}&=d^{\prime}_{i_{0}-n}(z_{i_{0}-n+1}(t_{n})-z_{i_{0}-n}(t_{n}))
               +d_{i_{0}-n}(z_{i_{0}-n-1}(t_{n})-z_{i_{0}-n}(t_{n}))\\
&\geq-D((4D+\sup\limits_{i}|c_{i}|){\|z_{i}\|}_{{\ell}^{\infty}}(t_{n-1}-t_{n})
      +|z_{i_{0}-(n-1)}(t_{n-1})-z_{i_{0}-n}(t_{n})|)\\
&\quad-D((4D+\sup\limits_{i}|c_{i}|){\|z_{i}\|}_{{\ell}^{\infty}}(t_{n}-t_{n+1})+
      |z_{i_{0}-(n+1)}(t_{n+1})-z_{i_{0}-n}(t_{n})|).\\
\end{split}
\end{equation*}
The right hand side of the above inequality converges to 0 as $n\to \infty$, which is a contradiction!\\
Case 2: If $z(t_{1},i_{0}+1)=\min\limits_{\Omega_{0}}z(t,i)<z(t_{0},i_{0})<0$, where
$t_{1}\in(0,t_{0}]$, then one can still obtain a contradiction as before by a similar argument.
The proof is complete.
\end{proof}
\begin{remark}\label{re4.1}
Assume that $z$ is bounded on $I\times\mathbf{Z}$ for any bounded interval $I\subset[0,+\infty)$, and that
$z_{i}(t)$ is locally Lipschitz continuous on $[0,\infty)$ for any $i\in\mathbf{Z}$. We denote the left derivative by $\frac{d^{-}}{dt}$,
and assume that $\frac{d^{-}}{dt}z_{i}(t)$ exists for any $t>0$, $i\in\mathbf{Z}$. Moreover, we assume that $z$ satisfies
$$\left\{
   \begin{aligned}
 \frac{d^{-}}{dt}z_{i}(t)-(A(t)z)_{i}-c_{i}(t)z_{i}\geq0,\ & \ \ \ (0,+\infty)\times\mathbf{Z},\\
                                     z\geq0,\ & \ \ \ t=0,\\
   \end{aligned}
   \right.$$
where
$A(t):X_{-\infty}\rightarrow X_{-\infty}$ is defined by
$(A(t)\phi)_{i}=d^{\prime}_{i}(t)(\phi_{i+1}-\phi_{i})+d_{i}(t)(\phi_{i-1}-\phi_{i})$ with
 the coefficients $d^{\prime}_{i}=d^{\prime}_{i}(t),\ d_{i}=d_{i}(t),\ c_{i}=c_{i}(t)$
being continuous functions in $t$ and satisfying
$$0<\inf \limits_{(t,i)\in[0,+\infty)\times\mathbf{Z}}d_{i}(t)\leq
\sup \limits_{(t,i)\in[0,+\infty)\times\mathbf{Z}}d_{i}(t)<+\infty,$$
$$0<\inf \limits_{(t,i)\in[0,+\infty)\times\mathbf{Z}}d_{i}^{\prime}(t)\leq
\sup \limits_{(t,i)\in[0,+\infty)\times\mathbf{Z}}d_{i}^{\prime}(t)<+\infty,$$
and $\sup \limits_{(t,i)\in[0,+\infty)\times\mathbf{Z}}|c_{i}(t)|<+\infty$.
Replacing $\overset{.}z_{i}(t)$ by $\frac{d^{-}}{dt}z_{i}(t)$ in the proof of Lemma \ref{lem4.2},
one can still prove $z\geq0$ in $(0,+\infty)\times\mathbf{Z}$.
\end{remark}

\subsection{The first part of Theorem \ref{thm2.1}}

\begin{proof}[Proof of the first part of Theorem \ref{thm2.1}]
For any $\omega>\overline{\omega}$, i.e.,
$\omega>\min\limits_{p>0}\frac{\overline{H}(-p)}{p}$, there exist $p>0,n>0$ such that
$\overline{{\lambda}_{1}}(-p,n)<\omega p$. Hence for
$\delta\in(0,\omega p)$, there exists $\phi\in\mathcal{A}_{n}$ such that
$$(\mathcal{L}\phi^{(-p)})_{i}\leq(\omega p-\delta)\phi^{(-p)}_{i},\ \forall\ i>n,$$
where ${\phi}_{i}^{(p)}\in X_{n}$ with ${\phi}_{i}^{(p)}={e}^{pi}{\phi}_{i}$.
We may assume that ${\phi_{i}e}^{-pi}\geq u_{i}(0)$ for any $i\geq n$ and
${\phi_{n}e}^{-pn}>1$ through multiplying by a positive constant since
$\{i:u(0,i)\neq0\}$ is not empty and has finite elements.
Let $\psi(t)\in X_{n}$ with
$\psi_{i}(t)={\phi_{i}e}^{-pi+(\omega p-\delta)t}\geq{\phi_{i}e}^{-pi}$ for $i\geq n,\ t\geq0$. Then
$\psi_{n}(t)>1$ for $t\geq0$ and $\displaystyle{\lim_{i\rightarrow \infty}}\psi_{i}(t)=0$ locally uniformly
with respect to $t\in[0,\infty)$ since
$\displaystyle{\lim_{i\rightarrow \infty}}\frac{\ln{{\phi}_{i}}}{i}=0$. Moreover,
$$\overset{.}\psi_{i}(t)-(A\psi)_{i}-f_{s}^{\prime}(i,0)\psi_{i}\geq0.$$
Let
\begin{equation}\label{4.2.1}
{v}_{i}(t)=
\left\{
   \begin{aligned}
 1,\ & \ \ \forall i<n,\\
 \min\{1,\psi_{i}(t)\},\ & \ \ \forall i\geq n.\\
   \end{aligned}
   \right.
\end{equation}
Then $v_{i}(0)\geq u_{i}(0)$ for any $i\in\mathbf{Z}$,
${v}_{i}(t)\leq\psi_{i}(t)$ for all $i\geq n$ and $t\geq0$, and ${v_{i}(t)}$ is locally Lipschitz continuous in $t$ for $i\in\mathbf{Z}$.
Next, we will show that $v_{i}$ satisfies $\frac{d^{-}}{dt}v_{i}(t)-(Av)_{i}-f(i,v_{i})\geq0$ for any
$t>0, i\in\mathbf{Z}$.

For any $i\leq n$, $\overset{.}v_{i}(t)-(Av)_{i}-f(i,v_{i})=-(Av)_{i}=d^{\prime}_{i}(v_{i}-v_{i+1})=d^{\prime}_{i}(1-v_{i+1})\geq0$
since $v_{n}(t)=\min\{1,\psi_{n}(t)\}=1$.

For any $i>n$, we need to prove $\frac{d^{-}}{dt}v_{i}(t)-(Av)_{i}-f(i,v_{i})\geq0$. We divide the proof into two
cases.\\
Case 1: $\psi_{i}(t)\geq1$. In this case, $v_{i}(t)=1$. Hence $f(i,v_{i})=0$. Since $\psi_{i}(t)$ is
increasing in $t$, we have
$$\frac{d^{-}}{dt}{v}_{i}=\displaystyle{\lim_{\Delta t\rightarrow0^{-}}}
\frac{v_{i}(t+\Delta t)-v_{i}(t)}{\Delta t}=
\left\{
   \begin{aligned}
 \displaystyle{\lim_{\Delta t\rightarrow0^{-}}}
\frac{1-1}{\Delta t}=0,\ & \ \psi_{i}>1,\\
 \displaystyle{\lim_{\Delta t\rightarrow0^{-}}}
\frac{\psi_{i}(t+\Delta t)-1}{\Delta t}=\overset{.}\psi_{i}(t)>0,\ & \ \psi_{i}=1.\\
   \end{aligned}
   \right.$$
We also have
$$Av_{i}=d^{\prime}_{i}(v_{i+1}-v_{i})+d_{i}(v_{i-1}-v_{i})=d^{\prime}_{i}(v_{i+1}-1)+d_{i}(v_{i-1}-1)\leq0.$$
We have thus proved that $\frac{d^{-}}{dt}v_{i}(t)-(Av)_{i}-f(i,v_{i})\geq0$ in this case.\\
Case 2: $\psi_{i}(t)<1$. In this case, $v_{i}(s)=\psi_{i}(s)$ for $s\in[t-\delta,t]$, where $\delta>0$
is small enough. From this, we have $f(i,v_{i})=f(i,\psi_{i})$ and
$$\overset{.}v_{i}(t)=\displaystyle{\lim_{\Delta t\rightarrow0^{-}}}
\frac{\psi_{i}(t+\Delta t)-\psi_{i}(t)}{\Delta t}
=\overset{.}\psi_{i}(t)>0.$$
Moreover, $v_{i\pm1}(t)\leq\psi_{i\pm1}(t)$ by \eqref{4.2.1}, which yields
$$(Av)_{i}=d^{\prime}_{i}(v_{i+1}-v_{i})+d_{i}(v_{i-1}-v_{i})\leq d^{\prime}_{i}(\psi_{i+1}-v_{i})+d_{i}(\psi_{i-1}-v_{i})=(A\psi)_{i}.$$
Hence $\overset{.}v_{i}(t)-(Av)_{i}-f(i,v_{i})\geq\overset{.}\psi_{i}(t)-(A\psi)_{i}-f(i,\psi_{i})
\geq\overset{.}\psi_{i}(t)-(A\psi)_{i}-f_{s}^{\prime}(i,0)\psi_{i}\geq0.$

Therefore, $v_{i}$ satisfies $\frac{d^{-}}{dt}v_{i}(t)-(Av)_{i}-f(i,v_{i})\geq0$ for any
$t>0, i\in\mathbf{Z}$.
Now let us set $w_{i}(t)=v_{i}(t)-u_{i}(t)$. Then $w$ satisfies
$$\left\{
   \begin{aligned}
 \frac{d^{-}}{dt}w_{i}(t)-(Aw)_{i}-\hat{c}_{i}w_{i}\geq0,\ & \ \ \ (0,+\infty)\times\mathbf{Z},\\
                                     w_{i}(0)\geq0,\ & \ \ \ t=0,\\
   \end{aligned}
   \right.$$
where
$\hat{c}_{i}=
\left\{
   \begin{aligned}
 \frac{f(i,v_{i})-f(i,u_{i})}{v_{i}-u_{i}},\ & \ \ \ v_{i}\neq u_{i}\\
                                     0,\ & \ \ \ v_{i}=u_{i}\\
   \end{aligned}
    \right.$.
If $v_{i}\neq u_{i}$, then there exists some $\theta_{i}$ between $v_{i}$ and $u_{i}$ such that
$|\frac{f(i,v_{i})-f(i,u_{i})}{v_{i}-u_{i}}|=|f^{\prime}_{s}(i,\theta_{i})|\leq
\|f(i,\cdot)\|_{{C}^{1+\gamma}}\leq\sup\limits_{i\in\mathbf{Z}}\|f(i,0)\|_{{C}^{1+\gamma}}$. Hence
$\hat{c}_{i}\in\ell^{\infty}$ uniformly in $t\in\mathbf{R}$.
Thus we have $w\geq0$ by Remark \ref{re4.1}, i.e., $v_{i}(t)\geq u_{i}(t)$ for any
$(t,i)\in(0,+\infty)\times\mathbf{Z}$. Hence $\psi_{i}\geq v_{i}\geq u_{i}$
for $i\geq n$, which yields
$$\sup\limits_{i\geq\omega t}u_{i}(t)
=\sup\limits_{i\geq\omega t}{\phi_{i}e}^{-pi+(\omega p-\delta)t}
\leq\sup\limits_{i\geq\omega t}{\phi_{i}e}^{-pi+(\omega p-\delta)\frac{i}{\omega}}
=\sup\limits_{i\geq\omega t}{\phi_{i}e}^{\frac{\delta i}{\omega}}\to0$$
as $t\to+\infty$ since $\lim_{i\rightarrow +\infty}\frac{\ln{{\phi}_{i}}}{i}=0$.
\end{proof}
\subsection{Harnack inequality}
\begin{thm}\label{thm4.1}
Assume that $u$ is bounded on $[0,+\infty)\times\mathbf{Z}$
and solves equation \eqref{1.1}. Then for any $(t,i)\in[0,+\infty)\times\mathbf{Z},\ T>0$,
there exists a positive constant $\theta$ only depending on $T$ such that
\begin{equation}
u_{i}(t)\leq\theta(T)u_{j}(t+T),\ \ j\in\{i\pm1,i\}.
\end{equation}
\end{thm}
\begin{proof}
It is easy to see that the solution satisfying $0\leq u_{j}(s)\leq1$ for any
$(s,j)\in[0,\infty)\times\mathbf{Z}$ is unique by Lemma \ref{lem4.2}.\\

Fix $i\in\mathbf{Z}$.
Let $z^{(i)}:=z^{(i)}_{j}(s),\ j\in\mathbf{Z},\ s\in\mathbf{R},$ be a solution of
$$\left\{
\begin{aligned}
      \overset{.}z_{j}^{(i)}(s)-(Az^{(i)})_{j}=0,\ & \ \ \ (0,+\infty)\times\mathbf{Z},\\
       z_{j}^{(i)}(0)=
      \left\{
        \begin{aligned}
              1,\ & \ \ \ j=i,\\
              0,\ & \ \ \ j\neq i,\\
        \end{aligned}
      \right.
\end{aligned}
 \right.$$
with $0\leq z_{j}^{(i)}(s)\leq1$.
By Lemma \ref{lem4.2}, we have $u_{i}(t)z_{j}^{(i)}(s)\leq u_{j}(t+s)$ for any
$s,\ t\in[0,\infty)\ \text{and}\ i,j\in\mathbf{Z}$. In particular, $u_{i}(t)z_{j}^{(i)}(T)\leq u_{j}(t+T)$.
That is to say, we only need to show that $z_{j}^{(i)}(T)$ has a uniform lower bound with respect to $i\in\mathbf{Z},j\in\{i\pm1,i\}$. In fact,
\begin{equation}\label{t4.1.1}
\left\{
\begin{aligned}
 \overset{.}z_{i}^{(i)}(s)=d^{\prime}_{i}z_{i+1}^{(i)}(s)+d_{i}z_{i-1}^{(i)}(s)-(d^{\prime}_{i}+d_{i})z_{i}^{(i)}(s),\ \\
 \overset{.}z_{i+1}^{(i)}(s)=d^{\prime}_{i+1}z_{i+2}^{(i)}(s)+d_{i+1}z_{i}^{(i)}(s)-(d^{\prime}_{i+1}+d_{i+1})z_{i+1}^{(i)}(s),\
 \\
 \overset{.}z_{i-1}^{(i)}(s)=d^{\prime}_{i-1}z_{i}^{(i)}(s)+d_{i-1}z_{i-2}^{(i)}(s)-(d^{\prime}_{i-1}+d_{i-1})z_{i-1}^{(i)}(s).
\end{aligned}
\right.
\end{equation}
By the first equality of \eqref{t4.1.1} we have
$\overset{.}z_{i}^{(i)}(s)\geq-(d^{\prime}_{i}+d_{i})z_{i}^{(i)}(s)\geq-2Dz_{i}^{(i)}(s)$, and this gives
\begin{equation}\label{t4.1.2}
\frac{d}{ds}(z_{i}^{(i)}(s){e}^{2Ds})\geq0.
\end{equation}
Integrating \eqref{t4.1.2} from 0 to $s>0$,
we have $z_{i}^{(i)}(s)\geq{e}^{-2Ds}$ for any $s>0$ since $z_{i}^{(i)}(0)=1$. In particular, $z_{i}^{(i)}(T)\geq{e}^{-2DT}$. Hence
$$\overset{.}z_{i\pm1}^{(i)}(s)\geq \underline{D}z_{i}^{(i)}(s)-2Dz_{i\pm1}^{(i)}(s)
\geq \underline{D}{e}^{-2Ds}-2Dz_{i\pm1}^{(i)}(s)$$
by the last two equalities of \eqref{t4.1.1}.
These give
\begin{equation}\label{t4.1.3}
\frac{d}{ds}(z_{i\pm1}^{(i)}(s){e}^{2Ds})\geq \underline{D}.
\end{equation}
Integrating \eqref{t4.1.3} from 0 to $s>0$,
we have $z_{i\pm1}^{(i)}(s)\geq\underline{D}s{e}^{-2Ds}$ for any $s>0$ since $z_{i\pm1}^{(i)}(0)=0$. In particular,
$z_{i\pm1}^{(i)}(T)\geq \underline{D}T{e}^{-2DT}$.
Let $\theta(T)=\frac{1}{\min\{1,\underline{D}T\}}{e}^{2DT}$. Then we are done.
\end{proof}
\begin{cor}\label{cor4.2}
Assume that all the conditions of Lemma \ref{lem4.2} hold. Then $u_{i}(t)>0$ for $(t,i)\in(0,+\infty)\times\mathbf{Z}$.
\end{cor}
\begin{proof}
If not, then there exists $(s,j)\in(0,+\infty)\times\mathbf{Z}$
such that $u(s,j)=0$. Let $i\in\mathbf{Z}$ with $u_{i}(0)\neq 0$, and take integer $N\geq|j-i|$ and $T=\frac{s}{N}$.
By Theorem \ref{thm4.1}, one has $u(s-T,k_{1})=0$ for $k_{1}\in\{j\pm1,j\}$ and $u(s-2T,k_{2})=0$ for
$k_{2}\in\{j\pm2,j\pm1,j\}$, $\cdots$, $u(0,k_{N})=0$ for $k_{N}\in\{j\pm N,j\pm (N-1),\cdots,j\pm1,j\}$.
In particular, $u(0,i)=0$, a contradiction!
\end{proof}

\subsection{Homogenization techniques for the lattice system}

In order to show the second part of Theorem \ref{thm2.1}, we will first use homogenization techniques to consider
the behavior of $v_{\varepsilon}(t,x)=\widetilde{u}(\frac{t}{\varepsilon},\frac{x}{\varepsilon}):=
u(\frac{t}{\varepsilon},[\frac{x}{\varepsilon}])$ as $\varepsilon\to 0$. For this reason, we define
$$\widetilde{u}(t,x):=u(t,[x]),\ \ \widetilde{d^{\prime}}(x):=d^{\prime}_{[x]},\ \ \widetilde{d}(x):=d_{[x]},\ \ \widetilde{f}(x,u):=f([x],u),$$
$$v_{\varepsilon}(t,x)=\widetilde{u}(\frac{t}{\varepsilon},\frac{x}{\varepsilon}),\ \
z_{\varepsilon}(t,x):=\varepsilon\ln{v_{\varepsilon}(t,x)}$$
for any $(t,x)\in(0,+\infty)\times\mathbf{R}$. Note that $z_{\varepsilon}(t,x)$ is well defined by Corollary \ref{cor4.2}. Moreover, $\widetilde{u}$ satisfies
\begin{equation}
  \left\{
   \begin{aligned}
   \partial_{t}\widetilde{u}(t,x)=\widetilde{d^{\prime}}(x)(\widetilde{u}(t,x+1)-\widetilde{u}(t,x))
   +\widetilde{d}(x)(\widetilde{u}(t,x-1)-\widetilde{u}(t,x))+\widetilde{f}(x,\widetilde{u}),\\
   0\leq\widetilde{u}(0,x)\leq1\ \text{with compact support}.\\
   \end{aligned}
   \right.
\end{equation}

\begin{thm}\label{thm4.2}
For any compact set $Q\subseteq(0,+\infty)\times\mathbf{R}$, there exist constants $c>0$ and $0<\varepsilon_{0}\leq1$ depending on $Q$ such that $|z_{\varepsilon}(t,x)|\leq c$ for any $\varepsilon\in(0,\varepsilon_{0}),(t,x)\in Q$.
\end{thm}
\begin{proof}
The proof is similar to \cite[Lemma 4.1]{B2} by using Theorem \ref{thm4.1}.
\end{proof}
By Theorem \ref{thm4.2}, one can define
$$z_{\ast}(t,x)=\displaystyle{\liminf_{(s,y)\rightarrow(t,x),\varepsilon\to0}}z_{\varepsilon}(s,y)\leq0.$$
In the following content of this section, we will find a Hamilton-Jacobi equation which is related to
$z_{\ast}.$
\begin{lemma}\label{lem4.3}
$\displaystyle{\liminf_{\varepsilon\to0}}v_{\varepsilon}(t,x)>0$ for any
$(t,x)\in int\{z_{\ast}=0\}$.
\end{lemma}
\begin{proof}
For any $(t_{0},x_{0})\in int\{z_{\ast}=0\}$, there exists some $\delta>0$ such that
$z_{\ast}(t,x)=0$ for any $(t,x)\in B_{2\delta}(t_{0},x_{0}):=\{(t,x):|t-t_{0}|+|x-x_{0}|<2\delta\}$.
Hence $z_{\varepsilon}(t,x)\to0$ as $\varepsilon\to0$ uniformly in $B_{\delta}(t_{0},x_{0})$. In fact,
if there are sequences $\varepsilon_{n}$ and $(t_{n},x_{n})\in B_{\delta}(t_{0},x_{0})$ such that $|z_{\varepsilon_{n}}(t_{n,}x_{n})|>\eta>0$, i.e., $z_{\varepsilon_{n}}(t_{n,}x_{n})<-\eta$,
then there exists a subsequence still denoted by $(\varepsilon_{n},t_{n},x_{n})$ converging to
$(0,\overline{t},\overline{x})$ with $(\overline{t},\overline{x})\in\overline{B_{\delta}(t_{0}x_{0})}\subseteq B_{2\delta}(t_{0},x_{0})$,
which yields
$0=\displaystyle{\liminf_{(s,y)\rightarrow(t,x),\varepsilon\to0}}z_{\varepsilon}(s,y)
\leq\displaystyle{\liminf_{n\to0}}z_{\varepsilon_{n}}(t_{n,}x_{n})<-\eta$. This is a contradiction.
Now consider $\phi(t,x)=-({|t-t_{0}|}^{2}+{|x-x_{0}|}^{2})$. Then $z_{\varepsilon}(t,x)-\phi(t,x)$ reaches
its minimum at some point, say $(t_{\varepsilon},x_{\varepsilon})$, over $B_{\delta}(t_{0},x_{0})$ and $(t_{\varepsilon},x_{\varepsilon})\to(t_{0},x_{0})$ as $\varepsilon\to0$ since
$z_{\varepsilon}(t,x)\to0$ as $\varepsilon\to0$ uniformly in $B_{\delta}(t_{0},x_{0})$.
By our setting and \eqref{1.1}, $z_{\varepsilon}(t,x)$ satisfies:
$$\overset{.}z_{\varepsilon}(t,x)=
\widetilde{d^{\prime}_{\varepsilon}}(x)({e}^{{\partial}^{+}_{\varepsilon}z_{\varepsilon}}-1)+
\widetilde{d}_{\varepsilon}(x)({e}^{{\partial}^{-}_{\varepsilon}z_{\varepsilon}}-1)+
\frac{\widetilde{f}(\frac{x}{\varepsilon},v_{\varepsilon})}{v_{\varepsilon}},$$
where $\widetilde{d^{\prime}_{\varepsilon}}(x)=\widetilde{d^{\prime}}(\frac{x}{\varepsilon}),\
\widetilde{d}_{\varepsilon}(x)=\widetilde{d}(\frac{x}{\varepsilon})$,
${\partial}^{\pm}_{\varepsilon}z_{\varepsilon}
=\frac{z_{\varepsilon}(t,x\pm\varepsilon)-z_{\varepsilon}(t,x)}{\varepsilon}$. Also noting that
at $(t_{\varepsilon},x_{\varepsilon})$, we have
$$\overset{.}z_{\varepsilon}(t,x)-{\partial}_{t}\phi(t,x)=0,$$
$${{\partial}^{\pm}_{\varepsilon}z_{\varepsilon}}-{{\partial}^{\pm}_{\varepsilon}\phi}\geq0.$$
Hence at $(t_{\varepsilon},x_{\varepsilon})$
\begin{equation*}
\begin{split}
{\partial}_{t}\phi(t,x)=
\widetilde{d^{\prime}_{\varepsilon}}(x)({e}^{{\partial}^{+}_{\varepsilon}z_{\varepsilon}}-1)+
\widetilde{d}_{\varepsilon}(x)({e}^{{\partial}^{-}_{\varepsilon}z_{\varepsilon}}-1)+
\frac{\widetilde{f}(\frac{x}{\varepsilon},v_{\varepsilon})}{v_{\varepsilon}}\\
\geq\widetilde{d^{\prime}_{\varepsilon}}(x)({e}^{{\partial}^{+}_{\varepsilon}\phi_{\varepsilon}}-1)+
\widetilde{d}_{\varepsilon}(x)({e}^{{\partial}^{-}_{\varepsilon}\phi_{\varepsilon}}-1)+
\frac{\widetilde{f}(\frac{x}{\varepsilon},v_{\varepsilon})}{v_{\varepsilon}}.
\end{split}
\end{equation*}
Then at $(t_{\varepsilon},x_{\varepsilon})$ we obtain
$$0<\frac{\widetilde{f}(\frac{x}{\varepsilon},v_{\varepsilon})}{v_{\varepsilon}}
\leq{\partial}_{t}\phi(t,x)-
\widetilde{d^{\prime}_{\varepsilon}}(x)({e}^{{\partial}^{+}_{\varepsilon}{\phi}_{\varepsilon}}-1)-
\widetilde{d}_{\varepsilon}(x)({e}^{{\partial}^{-}_{\varepsilon}{\phi}_{\varepsilon}}-1)\to0$$
as $\varepsilon\to0$. Since
$f(i,\cdot)\in\mathcal{C}^{1+\gamma}([0,1])$ uniformly with respect to $i\in\mathbf{Z}$, there exists $C>0$ such that
$$\frac{\widetilde{f}(\frac{x_{\varepsilon}}{\varepsilon},v_{\varepsilon})}{v_{\varepsilon}}
=\frac{f([\frac{x_{\varepsilon}}{\varepsilon}],v_{\varepsilon})}{v_{\varepsilon}}
\geq f^{\prime}_{s}([\frac{x_{\varepsilon}}{\varepsilon}],0)-Cv^{\gamma}_{\varepsilon},$$
which yields
$$0<f^{\prime}_{s}([\frac{x_{\varepsilon}}{\varepsilon}],0)
\leq Cv^{\gamma}_{\varepsilon}+o(1)\ \text{as}\ \varepsilon\to 0^{+}.$$
Then we have
$${C}^{\frac{1}{\gamma}}\liminf\limits_{\varepsilon\to0}v_{\varepsilon}(t_{\varepsilon},x_{\varepsilon})
\geq{(\inf\limits_{i}f^{\prime}_{s}(i,0))}^{\frac{1}{\gamma}}>0,$$
and the last inequality follows from $0<\inf \limits_{i}f(i,s)\leq f(i,s)\leq f_{s}^{\prime}(i,0)s$ for any $s\in(0,1)$.
Furthermore, by the definition of $(t_{\varepsilon},x_{\varepsilon})$, we conclude
$$z_{\varepsilon}(t_{\varepsilon},x_{\varepsilon})\leq
z_{\varepsilon}(t_{\varepsilon},x_{\varepsilon})-\phi(t_{\varepsilon},x_{\varepsilon})\leq
z_{\varepsilon}(t_{0},x_{0})-\phi(t_{0},x_{0})=z_{\varepsilon}(t_{0},x_{0}),$$
thus
$\liminf\limits_{\varepsilon\to0}v_{\varepsilon}(t_{0},x_{0})\geq
\liminf\limits_{\varepsilon\to0}v_{\varepsilon}(t_{\varepsilon},x_{\varepsilon})>0$.
\end{proof}
\begin{lemma}\label{lem4.4}
The lower semi-continuous function $z_{\ast}$ is a viscosity supersolution of
$$\max\{\partial_{t}z_{\ast}-\underline{H}(\partial_{x}z_{\ast}),z_{\ast}\}\geq0\ \ \ \
(t,x)\in(0,+\infty)\times(0,+\infty).$$
\end{lemma}
\begin{proof}
Note that $z_{\ast}\leq0$. Hence we only need to show that
$$\partial_{t}z_{\ast}-\underline{H}(\partial_{x}z_{\ast})\geq0$$
in $\{(t,x)\in(0,+\infty)\times(0,+\infty)|z_{\ast}(t,x)<0\}$.
For a smooth function $\phi$ defined on $(0,+\infty)\times(0,+\infty)$,
assume that $z_{\ast}-\phi$ reaches its strict minimum at $(t_{0},x_{0})\in(0,+\infty)\times(0,+\infty)$
over $\overline{B_{r}(t_{0},x_{0})}$, with $z_{\ast}(t_{0},x_{0})<0$.
It is sufficient to show that
$$\partial_{t}\phi-\underline{H}(\partial_{x}\phi)\geq0 \ \ \ \ \text{at}\ \ \ \ (t_{0},x_{0}).$$
Denote $p:=\partial_{x}\phi(t_{0},x_{0})$ and fix some $N\in\mathbf{Z}$ large enough. For any $\mu>0$ small, there exists $\psi\in\mathcal{A}_{N}$ such that
\begin{equation}\label{l4.4.1}
L_{p}\psi\geq(\underline{{\lambda}_{1}}(p,N)-\mu)\psi\ \text{on}\ I_{N}.
\end{equation}
Denote $\psi_{\varepsilon}(x):=\widetilde{\psi}(\frac{x}{\varepsilon})=\psi([\frac{x}{\varepsilon}])$.
Then $\varepsilon\ln\psi_{\varepsilon}(\cdot)\to0$ as $\varepsilon\to0$ locally uniform in $(N,+\infty)$.

Using an argument similar to the proof of \cite[Propsition 4.3]{B2}, we can obtain a sequence $(\varepsilon_{n},t_{n},x_{n})\to(0,t_{0},x_{0})$  as $n\to+\infty$. Moreover,
$z_{\varepsilon_{n}}(t_{n},x_{n})\to z_{\ast}(t_{0},x_{0})$ as $n\to+\infty$, and $z_{\varepsilon_{n}}-\phi-\varepsilon_{n}\ln\psi_{\varepsilon_{n}}$ reaches its strict minimum at
$(t_{n},x_{n})$ over $\overline{B_{r}(t_{0},x_{0})}$ for $n\geq n_{0}$.

Noting that $\frac{x_{n}}{\varepsilon_{n}}\to+\infty$ as $n\to+\infty$ since $x_{0}>0$, we may assume that
$\frac{x_{n}}{\varepsilon_{n}}>N$ for $n\geq n_{0}$.
Hence by \eqref{l4.4.1} at $(\varepsilon_{n},t_{n},x_{n})$, we have
\begin{equation*}
\begin{split}
{e}^{p}\widetilde{d^{\prime}_{\varepsilon}}(x)\psi_{\varepsilon}(x+\varepsilon)+
{e}^{-p}\widetilde{d}_{\varepsilon}(x)\psi_{\varepsilon}(x-\varepsilon)-
(\widetilde{d^{\prime}_{\varepsilon}}(x)+\widetilde{d}_{\varepsilon}(x))\psi_{\varepsilon}(x)\\
\geq(-f^{\prime}_{s}([\frac{x}{\varepsilon}],0)+\underline{{\lambda}_{1}}(p,N)-\mu)\psi_{\varepsilon}(x),
\end{split}
\end{equation*}
that is,
\begin{equation}\label{l4.4.2}
\begin{split}
{\mathrm{e}}^{p}\widetilde{d^{\prime}_{\varepsilon}}(x)
\frac{\psi_{\varepsilon}(x+\varepsilon)}{\psi_{\varepsilon}(x)}+
{\mathrm{e}}^{-p}\widetilde{d}_{\varepsilon}(x)
\frac{\psi_{\varepsilon}(x-\varepsilon)}{\psi_{\varepsilon}(x)}-
(\widetilde{d^{\prime}_{\varepsilon}}(x)+\widetilde{d}_{\varepsilon}(x))\\
\geq(-f^{\prime}_{s}([\frac{x}{\varepsilon}],0)+\underline{{\lambda}_{1}}(p,N)-\mu).
\end{split}
\end{equation}
We also have
\begin{equation}\label{l4.4.3}
\overset{.}z_{\varepsilon}=
\widetilde{d^{\prime}_{\varepsilon}}(x)({\mathrm{e}}^{{\partial}^{+}_{\varepsilon}z_{\varepsilon}}-1)+
\widetilde{d}_{\varepsilon}(x)({\mathrm{e}}^{{\partial}^{-}_{\varepsilon}z_{\varepsilon}}-1)+
\frac{f([\frac{x}{\varepsilon}],v_{\varepsilon})}{v_{\varepsilon}}.
\end{equation}
Denote $\beta_{\varepsilon}(x):=\varepsilon\ln\psi_{\varepsilon}(x)$. Then $\frac{\psi_{\varepsilon}(x\pm\varepsilon)}{\psi_{\varepsilon}(x)}
=\mathrm{e}^{\partial^{\pm}_{\varepsilon}\beta_{\varepsilon}(x)}$
is bounded by Claim 2 in the proof of Proposition \ref{prop2.1}.
Since $z_{\varepsilon_{n}}-\phi-\beta_{\varepsilon_{n}}$ reaches its strict minimum at $(t_{n},x_{n})$,
$$\partial_{t}(z_{\varepsilon_{n}}-\phi-\beta_{\varepsilon_{n}})(t,x)=0,\
\text{i.e.},\ \overset{.}z_{\varepsilon_{n}}(t,x)=\partial_{t}\phi(t,x)\ \text{at}\ (t_{n},x_{n}),$$
$${\partial}^{\pm}_{\varepsilon_{n}}(z_{\varepsilon_{n}}-\phi-\beta_{\varepsilon_{n}})(t,x)\geq0,\
\text{i.e.},\ {\partial}^{\pm}_{\varepsilon_{n}}z_{\varepsilon_{n}}(t,x)
\geq{\partial}^{\pm}_{\varepsilon_{n}}(\phi+\beta_{\varepsilon_{n}})(t,x)\ \text{at}\ (t_{n},x_{n}).$$
\eqref{l4.4.2} and \eqref{l4.4.3} yield
\begin{equation}\label{l4.4.4}
\begin{split}
\partial_{t}\phi(t,x)-\underline{{\lambda}_{1}}(p,N)+\mu
&\geq\frac{f([\frac{x}{\varepsilon}],v_{\varepsilon})}{v_{\varepsilon}}
  +\widetilde{d^{\prime}_{\varepsilon}}(x)
  ({\mathrm{e}}^{{\partial}^{+}_{\varepsilon}z_{\varepsilon}}
  -{\mathrm{e}}^{{\partial}^{+}_{\varepsilon}\beta_{\varepsilon}+p})\\
  &\quad-f^{\prime}_{s}([\frac{x}{\varepsilon}],0)
  +\widetilde{d}_{\varepsilon}(x)
  ({\mathrm{e}}^{{\partial}^{-}_{\varepsilon}z_{\varepsilon}}
  -{\mathrm{e}}^{{\partial}^{-}_{\varepsilon}\beta_{\varepsilon}-p}).
\end{split}
\end{equation}
As $\lim\limits_{n\to+\infty}z_{\varepsilon_{n}}(t_{n},x_{n})=z_{\ast}(t_{0},x_{0})<0$, we have
$v_{\varepsilon_{n}}(t_{n},x_{n})=\exp{\frac{z_{\varepsilon_{n}}(t_{n},x_{n})}{\varepsilon_{n}}}\to0$ as
$n\to+\infty$. Then
$\frac{f([\frac{x}{\varepsilon}],v_{\varepsilon})}{v_{\varepsilon}}
-f^{\prime}_{s}([\frac{x}{\varepsilon}],0)\to0$
as $n\to+\infty$. It is known that $\widetilde{d}_{\varepsilon}(x)$ and
$\mathrm{e}^{\partial^{\pm}_{\varepsilon}\beta_{\varepsilon}(x)}$
are uniformly bounded for $\varepsilon>0, x\in\mathbf{R}$. Taking $n\to +\infty$ in \eqref{l4.4.4},
and noting the  definition of
$p=\partial_{x}\phi(t_{0},x_{0})=
\lim\limits_{n\to+\infty}\pm{\partial}^{\pm}_{\varepsilon_{n}}\phi$, we obtain that
\begin{equation*}
\begin{split}
\partial_{t}\phi(t_{0},x_{0})-\underline{{\lambda}_{1}}(p,N)+\mu
&\geq\limsup\limits_{n\to+\infty}\ [\widetilde{d^{\prime}}_{\varepsilon_{n}}(x_{n})
  ({\mathrm{e}}^{{\partial}^{+}_{\varepsilon_{n}}z_{\varepsilon_{n}}}
  -{\mathrm{e}}^{{\partial}^{+}_{\varepsilon_{n}}(\phi+\beta_{\varepsilon_{n}})})\\
  &\quad+\widetilde{d}_{\varepsilon_{n}}(x_{n})
  ({\mathrm{e}}^{{\partial}^{-}_{\varepsilon_{n}}z_{\varepsilon_{n}}}
  -{\mathrm{e}}^{{\partial}^{-}_{\varepsilon_{n}}(\phi+\beta_{\varepsilon_{n}})})]\\
    &\geq0.
\end{split}
\end{equation*}
Then taking $\mu\to0^{+}$, we have
$$\partial_{t}\phi(t_{0},x_{0})-\underline{{\lambda}_{1}}(p,N)\geq0.$$
Finally, taking $N\to+\infty$,
$$\partial_{t}\phi(t_{0},x_{0})-\underline{H}(p)\geq0.$$
The proof is complete .
\end{proof}

Next we need to consider the convex conjugate of $\underline{H}$ given by
$\underline{H}^{*}(q):=\sup\limits_{p\in\mathbf{R}}(pq-\underline{H}(p))\geq pq-\underline{H}(p)$
for any $p\in\mathbf{R}$. It is well defined by Proposition \ref{prop2.2}.
We have the following estimate for $z_{\ast}$.
\begin{lemma}\label{lem4.5}
$z_{\ast}(t,x)\geq\min\{-t\underline{H}^{*}(-\frac{x}{t}),0\}$ for any $(t,x)\in(0,+\infty)\times(0,+\infty).$
\end{lemma}
\begin{proof}
The proof is similar to \cite[Lemma 4.4]{B2}.
\end{proof}

\subsection{Complete the proof of Theorem \ref{thm2.1}}

\begin{proof}[Proof of the second part of Theorem \ref{thm2.1}] We will prove it in five steps.\\
Step 1: Show that $\liminf\limits_{t\to+\infty}\widetilde{u}(t,wt)>0$ $\forall\omega\in(0,\underline{\omega})$.

By the definition of $\underline{\omega}$, there exists $\varepsilon>0$ such that
$\underline{H}(-p)\geq p\omega(1+\varepsilon)$ for any $p>0$; also from Proposition \ref{prop2.2} one can find
that there exists $0<\eta\leq\varepsilon_{0}\leq\underline{H}(0)$ such that
$\underline{H}(-p)\geq p\omega+\eta$,
i.e., $-\eta\geq(-p)(-\omega)-\underline{H}(-p)$ for any
$p\in\mathbf{R}$. Then we obtain
$-\underline{H}^{*}(-\omega)\geq\eta>0$. Hence by the continuity of $\underline{H}^{*}$ and
Lemma \ref{lem4.5}, there exists a neighbourhood $B(1,\omega)$ of $(1,\omega)\in(0,+\infty)\times(0,+\infty)$ such that

$$z_{\ast}(t,x)\geq \min\{-t\underline{H}^{*}(-\frac{x}{t}),0\}=0,\ \ \forall(t,x)\in B(1,\omega).$$
That is to say, $(1,\omega)\in \text{int}\{z_{\ast}=0\}$. From Lemma \ref{lem4.3}, we have
$$\liminf\limits_{t\to+\infty}\widetilde{u}(t,t\omega)=
\liminf\limits_{\varepsilon\to0}\widetilde{u}(\frac{1}{\varepsilon},\frac{\omega}{\varepsilon})
=\liminf\limits_{\varepsilon\to0}v_{\varepsilon}(1,\omega)>0.$$
Step 2: If we set
$d^{\prime-}_{i}:=d^{\prime}_{-i}, d^{-}_{i}:=d_{-i},\ f^{-}(i,s):=f(-i,s),\ f^{\prime-}_{s}(i,0):=f^{\prime}_{s}(-i,0)\ \text{and}\ u^{-}(t,i):=u(t,-i)$,
then $u^{-}$ satisfies
$$\overset{.}u_{i}^{-}(t)=d^{-}_{i}(u^{-}_{i+1}(t)-u^{-}_{i}(t))+
d^{\prime-}_{i}(u^{-}_{i-1}(t)-u^{-}_{i}(t))+f^{-}(i,u^{-}_{i}).$$
As what we did before, one can still define $\underline{{\lambda}^{-}_{1}}(p,n),\overline{{\lambda}^{-}_{1}}(p,n),\
\underline{H}^{-}(p),\overline{H}^{-}(p),\underline{\omega}^{-} \text{and}\ \overline{\omega}^{-}$ associated
with $\mathcal{L}^{-}$, where $\mathcal{L}^{-}:X_{n}\rightarrow X_{n+1}$ is defined by
$(\mathcal{L}^{-}\phi)_{i}=d^{-}_{i}(\phi_{i+1}-\phi_{i})+d^{\prime-}_{i}(\phi_{i-1}-\phi_{i})+f^{\prime-}_{s}(i,0)\phi_{i}$.
Moreover, we have
$\underline{{\lambda}^{-}_{1}}(p,n)\leq\overline{{\lambda}^{-}_{1}}(p,n),\
0<\underline{H}^{-}(p)\leq\overline{H}^{-}(p)$ and
$0<\underline{\omega}^{-}\leq\overline{\omega}^{-}$. In particular, for any
$\omega^{-}\in(0,\underline{\omega}^{-})$, we have
$\liminf\limits_{t\to+\infty}\widetilde{u}^{-}(t,t\omega^{-})>0$, where
$\widetilde{u}^{-}(t,x):=u^{-}(t,[x])=u(t,-[x])$. Hence we obtain
$\liminf\limits_{t\to+\infty}u(t,-[t\omega^{-}])>0$.\\
Step 3: Show that $\liminf\limits_{t\to+\infty}\{\inf\limits_{i\in S_{t}}u(t,i)\}>0$, where $S_{t}=
\{i\in\mathbf{Z}:-\omega^{-}t\leq i\leq\omega t\}.$

To do this, we first introduce another generalized principal eigenvalue for $\mathcal{L}$
$$\lambda^{\prime}_{1}=\sup\{\lambda:\exists \phi\in\ell^{\infty}, \inf\limits_{i\in\mathbf{Z}}\phi_{i}>0, \mathcal{L}\phi\geq \lambda\phi,i\in\mathbf{Z}\}.$$
(Related notions were defined in \cite{B1,B12,B4} for continuous problems.)
Taking $\phi_{i}\equiv1$ as a test function, it is easy to see that
$0<\inf\limits_{i\in\mathbf{Z}}f^{\prime}(i,0)\leq\lambda^{\prime}_{1}\leq\underline{\lambda_{1}}(0,n)$
for any $n\in\{-\infty\}\cup\mathbf{Z}$ by the definition of $\lambda^{\prime}_{1}$. Hence there exist $0<\mu<\lambda^{\prime}_{1}$ and
$\phi\in\ell^{\infty}{(\mathbf{Z})}$ with $\inf\limits_{i\in\mathbf{Z}}\phi_{i}>0$ such that
$(A\phi)_{i}+f_{s}^{\prime}(i,0)\phi_{i}\geq\mu\phi_{i}$. Without loss of generality, we may assume
$\sup\limits_{i\in\mathbf{Z}}\phi_{i}=1$. Note that $f(i,\cdot)\in\mathcal{C}^{1+\gamma}([0,1])$
uniformly in $i\in\mathbf{Z}$. Then there exists $\delta_{0}$ with $0<\delta_{0}<1$ such that $f(i,s)\geq(f^{\prime}(i,0)-\mu)s$
for any$s\in[0,\delta_{0}]$.

Furthermore, by Step 1 and Step 2, there exist $t_{0}>0,\ \delta_{1}>0$ such that
$u(t,[\omega t])\geq\delta_{1}$, $u(t,[\omega^{-}t])\geq\delta_{1}$ for any $t\in[t_{0},+\infty)$, and
$u(t_{0},i)\geq\delta_{1}$ for any $i\in S_{t_{0}}$. Set $\delta_{2}=\min\{\delta_{0},\delta_{1}\}$ and
$z:=u-\delta_{2}\phi\geq0\ \text{on}\ \Gamma_{t_{0}}$. Thus $z$ satisfies
$$\left\{
   \begin{aligned}
 \overset{.}z_{i}(t)-(Az)_{i}-\hat{c}_{i}z_{i}\geq0,\ & \ \ (0,+\infty)\times\mathbf{Z},\\
                                z\geq0,\ & \ \ \Gamma_{t_{0}},\\
   \end{aligned}
   \right.$$
where
$\hat{c}_{i}=
\left\{
   \begin{aligned}
 \frac{f(i,u)-f(i,\delta\phi)}{v_{i}-u_{i}},\ & \ \ \ u_{i}\neq \delta\phi_{i}\\
                                     0,\ & \ \ \ u_{i}=\delta\phi_{i}\\
   \end{aligned}
    \right.$.
Then $\hat{c}_{i}$ is bounded since $\sup\limits_{i}\|f(i,\cdot)\|_{\mathcal{C}^{1+\gamma}}<+\infty$.
By Lemma \ref{lem4.1} we have
$z\geq0$ in $\Omega_{t_{0}}$, which means $u(t,i)\geq\delta_{2}\phi_{i}$ for any
$(t,i)\in\{t\geq t_{0},i\in S_{t}\}$. In particular,
$$\inf\limits_{t\geq t_{0},i\in S_{t}}u(t,x)\geq\delta_{2}\inf\limits_{t\geq t_{0},i\in S_{t}}\phi_{i}
\geq\delta_{2}\inf\limits_{i}\phi_{i}>0.$$
Hence $\liminf\limits_{t\to+\infty}\{\inf\limits_{i\in S_{t}}u(t,i)\}>0$.\\
Step 4: Consider $\limsup\limits_{t\to+\infty}\sup\limits_{i\in S_{\varepsilon}(t)}|u(t,i)-1|$, where
$S_{\varepsilon}(t)=\{i\in\mathbf{Z}:-(\omega^{-}-\varepsilon)t\leq i\leq(\omega-\varepsilon)t\}$ and
$\varepsilon>0$. One can assume that $\{t_{n}\}$ is increasing and $i_{n}\in S_{\varepsilon}(t_{n})$ with
$t_{n}\to+\infty$ as $n\to+\infty$ such that
$$\lim\limits_{n\to+\infty}|u(t_{n},i_{n})-1|=
\limsup\limits_{t\to+\infty}\sup\limits_{i\in S_{\varepsilon}(t)}|u(t,i)-1|.$$
Set $u^{(n)}(t,i):=u_{i+i_{n}}(t+t_{n})\in(0,1]$. Then $u^{(n)}$ satisfies
$$\overset{.}u^{(n)}(t,i) =
d^{\prime(n)}_{i}(u^{(n)}_{i+1}(t)-u^{(n)}_{i}(t))+d^{(n)}_{i}(u^{(n)}_{i-1}(t)-u^{(n)}_{i}(t))+f^{(n)}(i,u^{(n)})$$
for any $t\in(-t_{n},+\infty),\ i\in\mathbf{Z}$, where
$d^{\prime(n)}_{i}=d^{\prime}_{i+i_{n}}, d^{(n)}_{i}=d_{i+i_{n}}, \text{and}\ f^{(n)}(i,s)=f(i+i_{n},s)$.

Since $d^{\prime}$ and $d$ belong to $\ell^{\infty}(\mathbf{Z})$, and $f(i,\cdot)$ are uniformly bounded in
$\mathcal{C}^{1+\gamma}([0,1])$ with respect to $i\in\mathbf{Z}$, one can use a diagonal extraction method
to find a subsequence, still denoted by $i_{n}$, and $\overline{d^{\prime}}\in\ell^{\infty}(\mathbf{Z})$,
$\overline{d}\in\ell^{\infty}(\mathbf{Z})$, $\overline{f}(i,\cdot)\in\mathcal{C}([0,1])$ and
$\overline{u}_{i}(\cdot)\in\mathcal{C}^{1}_{loc}([0,1])$
such that
$$d^{\prime}_{i+i_{n}}\to\overline{d^{\prime}}_{i},$$
$$d_{i+i_{n}}\to\overline{d}_{i},$$
$$f^{(n)}(i,\cdot)\to\overline{f}(i,\cdot)\ \ \ \text{in}\ \ \mathcal{C}([0,1]),$$
$$u^{(n)}_{i}(\cdot)\to\overline{u}_{i}(\cdot)\ \ \ \text{in}\ \
  \mathcal{C}^{1}_{loc}(\mathbf{R})$$
as $n\to+\infty$.

Moreover, $\overline{u}$ satisfies
$$\frac{d}{dt}\overline{u}(t,i)=\overline{A}\overline{u}(t,i)+\overline{f}(i,\overline{u})\ \ \ \ \ \text{in}\ \ \mathbf{R}\times\mathbf{Z},$$
where $\overline{A}\phi(i)
:=\overline{d^{\prime}}_{i}\phi(i+1)+\overline{d}_{i}\phi(i-1)-(\overline{d}_{i}+\overline{d^{\prime}}_{i})\phi(i)$.
For any fixed $(\tau,j)\in\mathbf{R}\times\mathbf{Z},$ we have
$$-\omega^{-}<-(\omega^{-}-\varepsilon)\leq\liminf\limits_{n\to+\infty}\frac{j+i_{n}}{\tau+t_{n}}
\leq\limsup\limits_{n\to+\infty}\frac{j+i_{n}}{\tau+t_{n}}\leq\omega-\varepsilon<\omega.$$
Then $i+i_{n}\in S(t+t_{n})$ for $n$ large enough. Therefore, the conclusion in Step 3 yields
\begin{equation*}
\begin{split}
\overline{u}(\tau,j)
&=\lim\limits_{n\to+\infty}u(\tau+t_{n},j+i_{n})\\
&\geq\lim\limits_{n\to+\infty}\inf\limits_{i\in S(\tau+t_{n})}u(\tau+t_{n},i)\\
&\geq\liminf\limits_{t\to+\infty}\inf\limits_{i\in S(t)}u(t,i)\\
&>0.\\
\end{split}
\end{equation*}
That is to say,
$\inf\limits_{\mathbf{R}\times\mathbf{Z}}\overline{u}\geq
\liminf\limits_{t\to+\infty}\inf\limits_{i\in S(t)}u(t,i)>0.$\\
Step 5: Show that $\overline{u}(t,i)\equiv1$.

If this is true, then
$$\limsup\limits_{t\to+\infty}\sup\limits_{i\in S_{\varepsilon}(t)}|u(t,i)-1|
=\lim\limits_{n\to+\infty}|u(t_{n},i_{n})-1|=|\overline{u}(0,0)-1|=0.$$
Note that $\omega\in(0,\omega),\ -\omega\in(0,\omega^{-})$ are arbitrary and that $\varepsilon>0$ can be
arbitrary small. Then one exactly has
$\limsup\limits_{t\to+\infty}\sup\limits_{i\in S(t)}|u(t,i)-1|=0$. In particular, the conclusion of part 2
of Theorem \ref{thm2.1} is valid.

As $u(t,i)\in[0,1]$, it follows from the definition of $\overline{u}$ and the conclusion in Step 4 that
$0<m_{0}:=\inf\limits_{\mathbf{R}\times\mathbf{Z}}\overline{u}\leq\overline{u}\leq1$.
Now it is sufficient to show that $m_{0}=1$.
Assume by contradiction that $m_{0}<1$. Then there exists $\overline{u}(s_{n},j_{n})\to m_{0}$ as $n\to+\infty$.
Set $\overline{u}^{(n)}(t,i)=\overline{u}(t+s_{n},i+j_{n})$, and by the same argument as we did in
Step 4, there exist $\hat{d^{\prime}}(i), \hat{d}(i),\ \hat{f}(i,s)$ and $\hat{u}(t,i)$
with $\hat{u}(t,i)\geq\hat{u}(0,0)=m_{0}$ satisfying
$$\frac{d}{dt}\hat{u}(t,i)=\hat{A}\hat{u}(t,i)
+\hat{f}(i,\hat{u})\ \ \ \ \ \text{in}\ \ \mathbf{R}\times\mathbf{Z},$$
where $\hat{A}\phi(i)
:=\hat{d^{\prime}}(i)\phi(i+1)+\hat{d}(i)\phi(i-1)
-(\hat{d^{\prime}}(i)+\hat{d}(i))\phi(i)$.
It is easy to find that $\inf\limits_{i}\hat{f}(i,s)>0$ for any $s\in(0,1)$ since
$\inf\limits_{i}f(i,s)>0$ for any $s\in(0,1)$. Hence
$\hat{f}(0,m_{0})>0$. Since
$\hat{u}$ reaches its minimum $m_{0}$ at $(0,0)$, we deduce:
$$\frac{d}{dt}\hat{u}|_{(0,0)}=0,\ \ \
\hat{A}\hat{u}|_{(0,0)}\geq0,\ \ \
\hat{f}(i,\hat{u})|_{(0,0)}=\hat{f}(0,m_{0})>0,$$
which is a contradiction! Thus $m_{0}=1$.
\end{proof}

\section{Application: almost periodic coefficients}

In this section, we assume that the coefficients of $\mathcal L$ are almost periodic. Here a sequence
$e_{i}$ is said almost periodic with respect to $i$ if for any sequence $\{i_{n}\}_{i=1}^{\infty},$ there
exists a subsequence $\{i_{n_{k}}\}$ such that $e_{i+i_{n_{k}}}$
converges uniformly in $i\in\mathbf{Z}$ as $k\to+\infty$. We will prove that $\underline{\omega}=\overline{\omega}$.
\begin{thm}\label{thm5.1}
Assume that $d^{\prime}_{i},\ d_{i},\ f^{\prime}_{s}(i,0)$ are almost periodic with respect to $i$. Then
$$\underline{\omega}=\overline{\omega}.$$
\end{thm}
We only need to show $\underline{{\lambda}_{1}}(p,-\infty)=\overline{{\lambda}_{1}}(p,-\infty)$.
We start from a comparison argument. Consider $B:X_{-\infty}\rightarrow X_{-\infty}$ defined by
$B\phi(i):=d^{+}(i)\mathrm{e}^{\partial^{+}\phi(i)}+d^{-}(i)\mathrm{e}^{\partial^{-}\phi(i)}$, where
${\partial}^{\pm}\phi(i)=\phi(i\pm1)-\phi(i)$, and $d^{\pm}(i)>0$ are bounded sequences, and we denote
$\overline{D}=\max\{\sup\limits_{i}d^{+}(i),\sup\limits_{i}d^{-}(i)\}.$
\begin{lemma}\label{lem:cp}
Assume that $w$ and $v$ belong to $\ell^{\infty}(\mathbf{Z})$
satisfying
$$\left\{
   \begin{aligned}
    \varepsilon w(i)-Bw(i)-c(i)\leq0,\\
    \varepsilon v(i)-Bv(i)-c(i)\geq0,
   \end{aligned}
   \right.$$
  where $\varepsilon>0$ is a parameter, $\sup\limits_{i}|c(i)|<+\infty$. Then $w(i)\leq v(i)$ for any $i\in\mathbf{Z}$.
\end{lemma}
\begin{proof}
Set $$\Phi(i,j)=w(i)-v(j)-\alpha|i-j|-\mu(i^{2}+j^{2})$$
for $\alpha>0,\ \mu>0$. Then $\Phi$ reaches its maximum at some point, say $(k,l)$,
over $\mathbf{Z}\times\mathbf{Z}$. Obviously, $(k,l)$ depends on $\alpha,\ \mu$.
If $w(i)\leq v(i)$ is not true, then there must exist $i_{0}\in\mathbf{Z},\ \delta>0$ such that
$w(i_{0})-v(i_{0})\geq2\delta$. Now for sufficiently small $\mu$,
we have $\Phi(i_{0},i_{0})=w(i_{0})-v(i_{0})-2\mu i^{2}_{0}\geq\delta$, hence
$0<\delta\leq\Phi(k,l)\leq\sup\limits_{i}|w(i)|+\sup\limits_{i}|v(i)|$. From this we obtain
$$\delta+\alpha|k-l|+\mu(k^{2}+l^{2})\leq w(k)-v(l)\leq\sup\limits_{i}|w(i)|+\sup\limits_{i}|v(i)|,$$
which yields $\mu k,\ \mu l\to0$ as $\mu\to0$ uniformly with respect to $\alpha$, $\alpha|k-l|$ bounded as $\mu\to0$
and $|k-l|\to0$ as $\alpha\to+\infty$ uniformly with respect to $\mu$, i.e., there exists $\alpha_{0}>0$
such that $k=l$ if
$\alpha\geq\alpha_{0}$. Furthermore,
\begin{equation}\label{l5.1.1}
\begin{split}
\varepsilon\delta
&\leq\varepsilon w(k)-\varepsilon v(l)\\
&\leq Bw(k)+c(k)-Bv(l)-c(l)\\
&=d^{+}(k)\mathrm{e}^{\partial^{+}w(k)}+d^{-}(k)\mathrm{e}^{\partial^{-}w(k)}+c(k)-c(l)\\
  &\quad-d^{+}(l)\mathrm{e}^{\partial^{+}v(l)}-d^{-}(l)\mathrm{e}^{\partial^{-}v(l)}.\\
\end{split}
\end{equation}
Note that $\Phi(k,l)$ reaches its maximum at $(k,l)$. Then $\Phi(k\pm1,l\pm1)\leq\Phi(k,l)$, i.e.,
$$w(k\pm1)-v(l\pm1)-\alpha|k-l|-\mu((k\pm1)^{2}+(l\pm1)^{2})\leq w(k)-v(l)-\alpha|k-l|-\mu(k^{2}+l^{2}).$$
Thus $\partial^{\pm}w(k)\leq\partial^{\pm}v(l)+2\mu(1\pm k\pm l)$. Take $\alpha\geq\alpha_{0}$ and then
\eqref{l5.1.1} yields
\begin{equation*}
\begin{split}
\varepsilon\delta
&\leq d^{+}(k)\mathrm{e}^{\partial^{+}w(k)}+d^{-}(k)\mathrm{e}^{\partial^{-}w(k)}+c(k)-c(l)\\
  &\quad-d^{+}(l)\mathrm{e}^{\partial^{+}v(l)}-d^{-}(l)\mathrm{e}^{\partial^{-}v(l)}\\
&\leq d^{+}(k)\mathrm{e}^{\partial^{+}v(l)+2\mu(1+k+l)}-d^{+}(l)\mathrm{e}^{\partial^{+}v(l)}+c(k)-c(l)\\
  &\quad +d^{-}(k)\mathrm{e}^{\partial^{-}v(l)+2\mu(1-k-l)}-d^{-}(l)\mathrm{e}^{\partial^{-}v(l)}\\
&=d^{+}(l)\mathrm{e}^{\partial^{+}v(l)}({\mathrm{e}^{2\mu(1+2l)}}-1)
  +d^{-}(l)\mathrm{e}^{\partial^{-}v(l)}({\mathrm{e}^{2\mu(1-2l)}}-1)\\
&\leq \overline{D}\mathrm{e}^{2\sup\limits_{i}|v(i)|}|{\mathrm{e}^{2\mu(1+2l)}}-1|+
      \overline{D}\mathrm{e}^{2\sup\limits_{i}|v(i)|}|{\mathrm{e}^{2\mu(1-2l)}}-1|\to0
\end{split}
\end{equation*}
as $\mu\to0$, a contradiction! Hence $w(i)\leq v(i)$ for any $i\in\mathbf{Z}$.
\end{proof}

Consider the following equation
\begin{equation}\label{l5.1.2}
\varepsilon u^{\varepsilon}(i)-Bu^{\varepsilon}(i)-c(i)=0,\ \ \ \ i\in\mathbf{Z}.
\end{equation}
It follows from the Perron's method that there is a unique solution $u^{\varepsilon}\in X_{-\infty}$ of
equation \eqref{l5.1.2} such that $-\frac{\sup_{i}|c(i)|}{\varepsilon}\leq
u^{\varepsilon}(i)\leq\frac{2\overline{D}+\sup_{i}|c(i)|}{\varepsilon}$.
\begin{lemma}
Assume that $d^{\pm}(i)$ and $c(i)$ are almost periodic.
Let $u^{\varepsilon}\in\ell^{\infty}(\mathbf{Z})$ be the solution of \eqref{l5.1.2}.
Then $\varepsilon u^{\varepsilon}(i)$ converges to some constant
as $\varepsilon\to0$ uniformly with respect to $i\in\mathbf{Z}$.
\end{lemma}
\begin{proof}
First, note that $-\frac{\sup_{i}|c(i)|}{\varepsilon}\leq
u^{\varepsilon}(i)\leq\frac{2\overline{D}+\sup_{i}|c(i)|}{\varepsilon}$. Hence there exists $M>0$ such that
$\sup\limits_{\varepsilon}\|\varepsilon u^{\varepsilon}\|_{\ell^{\infty}}\leq M$. Moreover,
there exists $M_{1}>0$ such that
$\sup\limits_{\varepsilon}\|\partial^{\pm}u^{\varepsilon}\|_{\ell^{\infty}}\leq M_{1}$.
In fact, if there exist sequences $\{\varepsilon_{n}\}$ and $\{i_{n}\}$ such that $|u^{\varepsilon_{n}}(i_{n}+1)-u^{\varepsilon_{n}}(i_{n})|\geq n$, then there
must be subsequences still denoted by $\{\varepsilon_{n}\}$ and $\{i_{n}\}$ such that $u^{\varepsilon}(i_{n}+1)-u^{\varepsilon}(i_{n})\geq n$ or
$u^{\varepsilon}(i_{n}+1)-u^{\varepsilon}(i_{n})\leq-n$.

On the other hand, from $\varepsilon u^{\varepsilon}(i)-c(i)=Bu^{\varepsilon}(i)=
d^{+}(i)\mathrm{e}^{\partial^{+}u^{\varepsilon}(i)}
+d^{-}(i)\mathrm{e}^{\partial^{-}u^{\varepsilon}(i)}$, we have
$$\varepsilon u^{\varepsilon}(i_{n})-c(i_{n})>
 d^{+}(i_{n})\mathrm{e}^{\partial^{+}u^{\varepsilon}(i_{n})}\geq
 d^{+}(i_{n})\mathrm{e}^{n}$$
or
$$\varepsilon u^{\varepsilon}(i_{n}+1)-c(i_{n}+1)>
 d^{-}(i_{n}+1)\mathrm{e}^{\partial^{-}u^{\varepsilon}(i_{n}+1)}\geq
 d^{-}(i_{n}+1)\mathrm{e}^{n}.$$
By taking $n\to+\infty$, one gets a contradiction since $\sup\limits_{i}|c(i)|<+\infty$ and
$\sup\limits_{\varepsilon}\|\varepsilon u^{\varepsilon}\|_{\ell^{\infty}}\leq M$.

Let $\hat{u}^{\varepsilon}(i):=u^{\varepsilon}(i)-u^{\varepsilon}(0)$. Then $\hat{u}^{\varepsilon}(i)$
satisfies
$$\varepsilon\hat{u}^{\varepsilon}(i)-B\hat{u}^{\varepsilon}(i)-c(i)+\varepsilon u^{\varepsilon}(0)=0.$$
Claim. $\varepsilon\hat{u}^{\varepsilon}(i)\to0$ as $\varepsilon\to0$ uniformly with respect to $i$.\\
Proof of Claim: Assume by contradiction that there exist $\varepsilon_{n}\to0,\ i_{n},\ \theta>0$ such that
$|\varepsilon_{n}\hat{u}^{\varepsilon_{n}}(i_{n})|\geq2\theta$. Without loss of generality, we may assume
$d^{\pm}(i+i_{n}),\ c(i+i_{n})$ converge uniformly as $n\to+\infty$ since they are almost periodic. Set
$\tilde{u}(i):=\hat{u}^{\varepsilon}(i+i_{n}),\ v(i):=\hat{u}^{\varepsilon}(i+i_{m})$, and denote
$d^{\pm}_{k}(i)=d^{\pm}(i+i_{k}),\ c_{k}(i)=c(i+i_{k})$. Then $\tilde{u}(i),\ v(i)$ satisfy
$$\varepsilon\tilde{u}(i)-d^{+}_{n}(i)\mathrm{e}^{\partial^{+}\tilde{u}(i)}
-d^{-}_{n}(i)\mathrm{e}^{\partial^{-}\tilde{u}(i)}
-c_{n}(i)+\varepsilon u^{\varepsilon}(0)=0,$$
$$\varepsilon v(i)-d^{+}_{m}(i)\mathrm{e}^{\partial^{+}v(i)}
-d^{-}_{m}(i)\mathrm{e}^{\partial^{-}v(i)}
-c_{m}(i)+\varepsilon u^{\varepsilon}(0)=0.$$
Set $w(i):=\tilde{u}(i)-\frac{\eta_{m,n}}{\varepsilon}$, where
$$\eta_{m,n}:=3\mathrm{e}^{M_{1}}\max\{\sup\limits_{i}|d^{+}_{m}(i)-d^{+}_{n}(i)|,
   \sup\limits_{i}|d^{-}_{m}(i)-d^{-}_{n}(i)|,\sup\limits_{i}|c_{m}(i)-c_{n}(i)|\}.$$
Then
\begin{equation*}
\begin{split}
&\ \ \ \ \varepsilon w(i)-d^{+}_{m}(i)\mathrm{e}^{\partial^{+}w(i)}
  -d^{-}_{m}(i)\mathrm{e}^{\partial^{-}w(i)}
  -c_{m}(i)+\varepsilon u^{\varepsilon}(0)\\
&=\varepsilon\tilde{u}(i)-d^{+}_{m}(i)\mathrm{e}^{\partial^{+}\tilde{u}(i)}
  -d^{-}_{m}(i)\mathrm{e}^{\partial^{-}\tilde{u}(i)}
  -c_{m}(i)+\varepsilon u^{\varepsilon}(0)-\eta_{m,n}\\
&=(d^{+}_{n}(i)-d^{+}_{m}(i))\mathrm{e}^{\partial^{+}\tilde{u}(i)}
  +(d^{-}_{n}(i)-d^{-}_{m}(i))\mathrm{e}^{\partial^{-}\tilde{u}(i)}+c_{n}(i)-c_{m}(i)-\eta_{m,n}\\
&\leq0.
\end{split}
\end{equation*}
Hence by Lemma \ref{lem:cp}, we have $w(i)\leq v(i)$ for any $i\in\mathbf{Z}$, i.e.,
$$\varepsilon\hat{u}^{\varepsilon}(i+i_{n})\leq\varepsilon\hat{u}^{\varepsilon}(i+i_{m})+\eta_{m,n}.$$
Since $|\varepsilon_{n}\hat{u}^{\varepsilon_{n}}(i_{n})|\geq2\theta$, without loss of generality we may assume
that either $\varepsilon_{n}\hat{u}^{\varepsilon_{n}}(i_{n})\geq2\theta$ or
$\varepsilon_{n}\hat{u}^{\varepsilon_{n}}(i_{n})\leq-2\theta$ for any $n$. We will only prove the case where
$\varepsilon_{n}\hat{u}^{\varepsilon_{n}}(i_{n})\geq2\theta$, and the proof of the other case is similar.
Setting $\varepsilon=\varepsilon_{n},\ i=0$,  we have
\begin{equation}\label{l5.2.1}
2\theta\leq\varepsilon_{n}\hat{u}^{\varepsilon_{n}}(i_{n})\leq\varepsilon_{n}\hat{u}^{\varepsilon_{n}}(i_{m})+\eta_{m,n}
\leq\varepsilon_{n}(u^{\varepsilon_{n}}(i_{m})-u^{\varepsilon_{n}}(0))+\eta_{m,n}\leq
\varepsilon_{n}i_{m}M_{1}+\eta_{m,n}
\end{equation}
for any $n,m\in\mathbf{N}$.
On the other hand, there exists $n_{0}$ such that for any $n,m\geq n_{0}$, $\eta_{m,n}<\theta$ holds by the
definition of $\eta_{m,n}$ and the choice of $i_{n}$. In particular, $\eta_{n_{0},n}<\theta$ for any
$n\geq n_{0}$, then \eqref{l5.2.1} possesses a special case as follows
$$2\theta\leq\varepsilon_{n}i_{n_{0}}M_{1}+\eta_{n_{0},n}.$$
Letting $n\to+\infty$, we obtain a contradiction. Hence $\varepsilon\hat{u}^{\varepsilon}(i)\rightarrow0$,
i.e., $\varepsilon u^{\varepsilon}(i)-\varepsilon u^{\varepsilon}(0)\rightarrow0$ as $\varepsilon\to0$ uniformly with respect to $i\in\mathbf{Z}$.
The proof of the claim is complete.

The claim means that for any sequence $\{\varepsilon_{n}\}$ there exists a subsequence still denoted by $\{\varepsilon_{n}\}$ such that
$\varepsilon_{n}u^{\varepsilon_{n}}\rightarrow
\lim\limits_{n\to+\infty}\varepsilon_{n}u^{\varepsilon_{n}}(0)$ uniformly with respect to $i\in\mathbf{Z}$.
Then we still need to show that for any sequence $\varepsilon_{n}$ tending to $0$,
$\varepsilon_{n}u^{\varepsilon_{n}}$ converges to the same constant as $n\to+\infty$.
If not, then there exist $\{\varepsilon_{n}\}$ and $\{\varepsilon^{\prime}_{n}\}$ such that
$\varepsilon_{n}u^{\varepsilon_{n}}\rightarrow a,\ \text{and}\
 \varepsilon^{\prime}_{n}u^{\varepsilon^{\prime}_{n}}\rightarrow b$
as $n\to+\infty$ uniformly with respect to $i\in\mathbf{Z}$. Without loss of generality, we may assume that $a>b$.
Then we choose
$\varepsilon\in\{\varepsilon_{n}\},\ \varepsilon^{\prime}\in\{\varepsilon^{\prime}_{n}\}$ such that
$\|\varepsilon u^{\varepsilon}-a\|_{\ell^{\infty}}<\frac{a-b}{4},\
 \|\varepsilon^{\prime}u^{\varepsilon^{\prime}}-b\|_{\ell^{\infty}}<\frac{a-b}{4}$. Hence
$\varepsilon u^{\varepsilon}(i)-\varepsilon^{\prime}u^{\varepsilon^{\prime}}(j)>\frac{a-b}{2}$
for all $i,j\in\mathbf{Z}$. Let
$$\Phi(i,j)=u^{\varepsilon}(i)-u^{\varepsilon^{\prime}}(j)-\alpha|i-j|-\mu(i^{2}+j^{2}).$$
Then by the same argument as Lemma \ref{lem:cp}, $\Phi$ reaches its maximum at some point, say $(k,l)$,
and $\mu k,\mu l\to0$ as $\mu\to0$ uniformly with respect to $\alpha$, $|k-l|\to0$ as $\alpha\to+\infty$ uniformly
with respect to $\mu$, and there still exists $\alpha_{0}>0$ such that $k=l$ if
$\alpha\geq\alpha_{0}$. Moreover,
$\partial^{\pm}u^{\varepsilon}(k)\leq\partial^{\pm}u^{\varepsilon^{\prime}}(l)+2\mu(1\pm k\pm l)$.
Hence for $\alpha\geq\alpha_{0}$
$$\frac{a-b}{2}\leq\varepsilon u^{\varepsilon}(k)-\varepsilon^{\prime}u^{\varepsilon^{\prime}}(l)
=Bu^{\varepsilon}(k)+c(k)-Bu^{\varepsilon^{\prime}}(l)-c(l)\ \to0$$
as $\mu\to0$, which is a contradiction! Thus the proof is complete.
\end{proof}
Denote $\lambda_{0}:=\lim\limits_{\varepsilon\to0}\varepsilon u^{\varepsilon}$. Now we can prove our main
result of this section.

\begin{proof}[Proof of Theorem \ref{thm5.1}]
Let $d^{+}(i)=d^{\prime}_{i}\mathrm{e}^{p},\ d^{-}(i)=d_{i}\mathrm{e}^{-p},\ c(i)=f^{\prime}_{s}(i,0)-d^{\prime}_{i}-d_{i}$
and $\phi_{i}=\mathrm{e}^{u^{\varepsilon}(i)}$, where $u^{\varepsilon}$ is a solution of \eqref{l5.1.2}.
Then $\{\phi_{i}\}\in\mathcal{A}_{-\infty}$
satisfies $(L_{p}\phi)_{i}={e}^{-pi}\mathcal (Le^{p\cdot}\phi)_{i}=\varepsilon u^{\varepsilon}(i)\phi_{i}$ for any $\varepsilon>0$.
Moreover, for any $\kappa>0$ small, there always exists $\varepsilon$ such that
$\|\varepsilon u^{\varepsilon}-\lambda_{0}\|_{\ell^{\infty}}\leq\kappa$. Then from the definition and
the monotonicity of $\underline{\lambda}_{1},\overline{\lambda}_{1}$, one can take $\{\phi\}$ as a test
vector to obtain
$$\lambda_{0}-\kappa\leq\underline{{\lambda}_{1}}(p,-\infty)\leq
\underline{{\lambda}_{1}}(p,n)\leq\overline{{\lambda}_{1}}(p,n)
\leq\overline{{\lambda}_{1}}(p,-\infty)\leq\lambda_{0}+\kappa$$
for $\kappa>0,\ n\in\mathbf{N},\ p\in\mathbf{R}$. Then setting $\kappa\to0$,  we have
$$\lambda_{0}=\underline{{\lambda}_{1}}(p,-\infty)=\underline{{\lambda}_{1}}(p,n)
=\overline{{\lambda}_{1}}(p,n)=\overline{{\lambda}_{1}}(p,-\infty).$$
Hence $\underline{\omega}=\min \limits_{p>0}\frac{\underline{{\lambda}_{1}}(-p,-\infty)}{p}
=\min \limits_{p>0}\frac{\overline{{\lambda}_{1}}(-p,-\infty)}{p}=\overline{\omega}.$
\end{proof}

If, furthermore, $d^{\prime}_{i}=d_{i}\equiv d$, then we can prove that the speed in the positive direction
equals to the speed in the negative direction, i.e.,
$\underline{\omega}^{-}=\overline{\omega}^{-}=\underline{\omega}=\overline{\omega}$, where
$\underline{\omega}^{-},\ \overline{\omega}^{-}$ were given in Step 2 in the proof of the second part of
Theorem \ref{thm2.1}. In fact, we have the following theorem:
\begin{thm}\label{thm5.2}
Assume that $d^{\prime}_{i}=d_{i}\equiv d,$ and $f^{\prime}_{s}(i,0)$ is almost periodic with respect to $i$. Then
$$
\underline{{\lambda}_{1}}(p,-\infty)=\overline{{\lambda}_{1}}(p,-\infty)=
\underline{{\lambda}^{-}_{1}}(p,-\infty)=\overline{{\lambda}^{-}_{1}}(p,-\infty).
$$
\end{thm}
\begin{proof}
First by Theorem \ref{thm5.1}, there exist
$u^{\varepsilon}\in{\ell}^{\infty},\ v^{\varepsilon}\in{\ell}^{\infty}$ such that
$$L_{p}\phi=\varepsilon u^{\varepsilon}\phi,\ L^{-}_{p}\psi=\varepsilon v^{\varepsilon}\psi,$$
where $\phi=\mathrm{e}^{u^{\varepsilon}},\ \psi=\mathrm{e}^{v^{\varepsilon}}$. Moreover,
$$\underline{{\lambda}_{1}}(p,-\infty)=\overline{{\lambda}_{1}}(p,-\infty)=\lim\limits_{\varepsilon\to0}
\varepsilon u^{\varepsilon}$$
$$\underline{{\lambda}^{-}_{1}}(p,-\infty)=\overline{{\lambda}^{-}_{1}}(p,-\infty)=\lim\limits_{\varepsilon\to0}
\varepsilon v^{\varepsilon}.$$
We denote $\lambda_{0}:=\lim\limits_{\varepsilon\to0}\varepsilon u^{\varepsilon}$ and
$\lambda^{-}_{0}:=\lim\limits_{\varepsilon\to0}\varepsilon v^{\varepsilon}$.
Now it is sufficient to show
$\lambda_{0}=\lambda^{-}_{0}$. If not, we may assume by contradiction that
$\lambda_{0}<\lambda^{-}_{0}$ without loss of generality, then there exists $\varepsilon_{0}$ such that
$\varepsilon_{0} u^{\varepsilon_{0}}<\lambda_{0}+\frac{\lambda^{-}_{0}-\lambda_{0}}{4}$ and
$\varepsilon_{0} v^{\varepsilon_{0}}>\lambda^{-}_{0}-\frac{\lambda^{-}_{0}-\lambda_{0}}{4}$. Since
$$\left\{
   \begin{aligned}
    d\mathrm{e}^{p}\phi_{i+1}+d\mathrm{e}^{-p}\phi_{i-1}+(f^{\prime}_{s}(i,0)-2d)\phi_{i}
    =\varepsilon_{0} u^{\varepsilon_{0}}_{i}\phi_{i}
     \leq(\lambda_{0}+\frac{\lambda^{-}_{0}-\lambda_{0}}{4})\phi_{i},\\
    d\mathrm{e}^{p}\psi_{i+1}+d\mathrm{e}^{-p}\psi_{i-1}+(f^{\prime-}_{s}(i,0)-2d)\psi_{i}
    =\varepsilon_{0} v^{\varepsilon_{0}}_{i}\psi_{i}
     \geq(\lambda^{-}_{0}-\frac{\lambda^{-}_{0}-\lambda_{0}}{4})\psi_{i}.
   \end{aligned}
   \right.$$
Hence we have
\begin{equation}\label{t5.2.1}
\sum_{i=-n}^{n}(d\mathrm{e}^{p}\phi_{i+1}+d\mathrm{e}^{-p}\phi_{i-1}+(f^{\prime}_{s}(i,0)-2d)\phi_{i})\psi_{-i}
\leq\sum_{i=-n}^{n}(\lambda_{0}+\frac{\lambda^{-}_{0}-\lambda_{0}}{4})\phi_{i}\psi_{-i},
\end{equation}
and
\begin{equation}\label{t5.2.2}
\sum_{i=-n}^{n}(d\mathrm{e}^{p}\psi_{-i+1}+d\mathrm{e}^{-p}\psi_{-i-1}+(f^{\prime-}_{s}(-i,0)-2d)\psi_{-i})\phi_{i}
\geq\sum_{i=-n}^{n}(\lambda^{-}_{0}-\frac{\lambda^{-}_{0}-\lambda_{0}}{4})\psi_{-i}\phi_{i}
\end{equation}
for any $n\in\mathbf{N}$.
By \eqref{t5.2.2}-\eqref{t5.2.1}, we have
\begin{equation*}
\begin{split}
\frac{\lambda^{-}_{0}-\lambda_{0}}{2}\sum_{i=-n}^{n}\psi_{-i}\phi_{i}
&\leq\sum_{i=-n}^{n}(d\mathrm{e}^{p}\psi_{-i+1}\phi_{i}-d\mathrm{e}^{p}\phi_{i+1}\psi_{-i})\\
&\quad+\sum_{i=-n}^{n}(d\mathrm{e}^{-p}\psi_{-i-1}\phi_{i}-d\mathrm{e}^{-p}\phi_{i-1}\psi_{-i}),\\
\end{split}
\end{equation*}
i.e.,
$$n(\lambda^{-}_{0}-\lambda_{0})\inf \limits_{i}(\phi_{i}\psi_{i})\leq
d\mathrm{e}^{p}(\psi_{n+1}\phi_{-n}-\psi_{-n}\phi_{n+1})
+d\mathrm{e}^{-p}(\psi_{-n-1}\phi_{n}-\psi_{n}\phi_{-n-1}).$$
Take $n\to \infty$ in the above inequality. Then we have
$$\liminf \limits_{n\to\infty}(d\mathrm{e}^{p}(\psi_{n+1}\phi_{-n}-\psi_{-n}\phi_{n+1})
+d\mathrm{e}^{-p}(\psi_{-n-1}\phi_{n}-\psi_{n}\phi_{-n-1}))\geq+\infty,$$
which contradicts $\phi=\mathrm{e}^{u^{\varepsilon_{0}}}\in{\ell}^{\infty}$ and
$\psi=\mathrm{e}^{v^{\varepsilon_{0}}}\in{\ell}^{\infty}$.
\end{proof}
\begin{remark}
For the reaction diffusion equation \eqref{continuousequation}, which was investigated in \cite{B2},
one can also show that the speed in the positive direction equals to the speed in the
negative direction if $a(x)\equiv a\ \text{is a constant},\ q(x)\equiv0$ and $f^{\prime}_{s}(x,0)$ is almost periodic.
\end{remark}

\section{Application: random stationary ergodic coefficients}

In this section, we always assume that $d^{\prime}_{i}=d_{i+1}$.  We consider a probability space
$(\Omega,\mathbb{P},\mathcal F)$ and assume that the reaction rate
$f:(i,s;\omega)\in{\mathbf{Z}\times[0,1]\times\Omega}\rightarrow\mathbf{R}$ and
$d:(i,\omega)\in{\mathbf{Z}\times\Omega}\rightarrow(0,+\infty)$ in \eqref{1.1} are random variables.
Furthermore, we assume that there exists $\Omega_{0}\subset\Omega$
with $\mathbb{P}(\Omega_{0})=1$ such that the following conditions hold for any fixed $\omega\in\Omega_{0}$:
$0<\inf \limits_{i}d(i,\omega)\leq\sup \limits_{i}d(i,\omega)<+\infty$,
$f(i,\cdot;\omega)\in\mathcal{C}^{1+\gamma}([0,1])$ with $\sup\limits_{i}\|f(i,\cdot;\omega)\|_{\mathcal{C}^{1+\gamma}}<+\infty$,
$f(i,0;\omega)\equiv f(i,1;\omega)\equiv0$,\ $0<\inf\limits_{i}f(i,s;\omega)\leq f(i,s;\omega)\leq f_{s}^{\prime}(i,0;\omega)s$ for any $s\in (0,1)$,
and $\liminf \limits_{|i|\to \infty}(f_{s}^{\prime}(i,0;\omega)-({\sqrt {d(i+1,\omega)}}-{\sqrt {d(i,\omega)}})^2)>0$.
We denote $d(i,\omega)$ by $d_{i}(\omega)$ and $f_{s}^{\prime}(i,0;\omega)$ by $c_{i}(\omega)=c(i,\omega)$
for $(i,\omega)\in{\mathbf{Z}\times\Omega}$.
The functions $c(\cdot,\cdot)$ and $d(\cdot,\cdot)$
are assumed to be random stationary ergodic, which means that there exists a group
$(\pi_{i})_{i\in\mathbf{Z}}$ of measure-preserving transformations acting ergodically on
$\Omega$ such that $c(i+j,\omega)=c(i,\pi_{j}\omega)$ and
$d(i+j,\omega)=d(i,\pi_{j}\omega)$ for any $(i,j,\omega)\in\mathbf{Z}\times\mathbf{Z}\times\Omega$.
The linearized operator $\mathcal L$ now depends on the event $\omega$,
that is,
$$(\mathcal L^{\omega}\phi)_{i}=(A^{\omega}\phi)_{i}+c_{i}(\omega)\phi_{i}=
d_{i+1}(\omega)(\phi_{i+1}-\phi_{i})+d_{i}(\omega)(\phi_{i-1}-\phi_{i})+c_{i}(\omega)\phi_{i}\ \forall \omega\in\Omega.$$
We also denote $L_{p}^{\omega}\phi={\mathrm{e}}^{-p\cdot}\mathcal L^{\omega}({\mathrm{e}}^{p\cdot}\phi)$
for $p\in\mathbf{R}$.
We associate these operators with two generalized eigenvalues
$\underline{{\lambda}_{1}}(p,n;\omega),\overline{{\lambda}_{1}}(p,n;\omega)$ through Definition \ref{def2.1}, two
Hamiltonians $\underline{H}(p,\omega),\overline{H}(p,\omega)$ through \eqref{2.2} and two speeds
$\underline{\omega}(\omega)$ and $\overline{\omega}(\omega)$ through \eqref{2.3}.

Our main result in this section is
\begin{thm}\label{thm6.1}
Under the assumptions stated above, there is a measurable set $\Omega^{\prime}\subset\Omega$
with $\mathbb{P}(\Omega^{\prime})=1$, such that
$$
\overline{\omega}(\cdot)=\underline{\omega}(\cdot)
$$
is a constant function on $\Omega^{\prime}$.
\end{thm}

First, a simple observation is that $\sup \limits_{i}d_{i}(\omega),\ \inf \limits_{i}d_{i}(\omega),\
\sup \limits_{i}c_{i}(\omega),\ \inf \limits_{i}c_{i}(\omega)$ and
$\liminf \limits_{|i|\to \infty}(c_{i}(\omega)-({\sqrt {d_{i}(\omega)}}-{\sqrt {d_{i}(\omega)}})^2)$
are constants, almost surely. In fact, let $D(\omega)=\sup \limits_{i}d_{i}(\omega)$,  we have
$D(\pi_{j}\omega)=\sup \limits_{i}d_{i}(\pi_{j}\omega)=\sup \limits_{i}d_{i+j}(\omega)=D(\omega)$. Then
the assumption of ergodicity implies that $D(\omega)$ is a constant function a.e.. For the other four terms,
the reason is similar. We may assume that
$\sup \limits_{i}d_{i}(\omega)=D,\ \inf \limits_{i}d_{i}(\omega)=\underline{D},\
\sup \limits_{i}c_{i}(\omega)=C,\ \inf \limits_{i}c_{i}(\omega)=c$ on $\Omega_{0}$ by omitting a set with
probability zero.

Denote $\mathbf{I}_{l,k}=\{l+1,l+2,\cdots,l+k\}$ for $l\in\mathbf{Z},\ k\in\mathbf{N}$. Fix $\omega\in\Omega$ and
let $(\Gamma_{l,k}(\omega),\phi^{l,k}(\omega))$ be the principal eigenpair of the following eigenvalue
problem
\begin{equation}\label{t6.1.1}
\left\{
   \begin{aligned}
   (A^{\omega}\phi^{l,k})_{i}(\omega)
    +c_{i}(\omega)\phi_{i}^{l,k}(\omega)=\lambda\phi_{i}^{l,k}(\omega),&\ \ i\in\mathbf{I}_{l,k},\\
    \phi_{l}^{l,k}(\omega)=\phi_{l+k+1}^{l,k}(\omega)=0.
   \end{aligned}
   \right.
\end{equation}
It is easy to find that
$\Gamma_{i,1}(\omega)=c_{i+1}(\omega)-d_{i+1}(\omega)-d_{i+2}(\omega)$,  and that \eqref{t6.1.1} is equivalent
to the following eigenvalue problem of the matrix:
$$D_{l,k}x=\lambda x,$$
where
\begin{gather*}
D_{l,k}:=
\left(
  \begin{array}{ccccc}
    \Gamma_{l+1} &   d_{l+2}   &            &          &            \\
    d_{l+2}      & \Gamma_{l+2}& d_{l+3}    &          &            \\
                & d_{l+3}      &\Gamma_{l+3} &\ddots    &            \\
                &              &\ddots       &\ddots    &d_{l+k+1}   \\
                &              &             &d_{l+k+1} &\Gamma_{l+k}\\
  \end{array}
\right)
\end{gather*}
with $\Gamma_{i}(\omega)=c_{i}(\omega)-d_{i}(\omega)-d_{i+1}(\omega),\ d_{i}=d_{i}(\omega)$.
\begin{lemma}\label{lem6.1}
There exists a constant $\Gamma_{\infty}$ such that the following statements hold on a subset of $\Omega$
with probability one: $\Gamma_{l_{1},k_{1}}(\omega)\leq\Gamma_{l_{2},k_{2}}(\omega)$ if
$\mathbf{I}_{l_{1},k_{1}}\subset\mathbf{I}_{l_{2},k_{2}}$, and
$$\displaystyle{\lim_{k\rightarrow+\infty}}\Gamma_{-k,2k}(\omega)
=\displaystyle{\lim_{k\rightarrow+\infty}}\Gamma_{-k,k}(\omega)
=\displaystyle{\lim_{k\rightarrow+\infty}}\Gamma_{0,k}(\omega)=\Gamma_{\infty}.$$
\end{lemma}
\begin{proof}
If $\mathbf{I}_{l_{1},k_{1}}\subset\mathbf{I}_{l_{2},k_{2}}$ then the fact that
$\Gamma_{l_{1},k_{1}}\leq\Gamma_{l_{2},k_{2}}$ follows from the variation representation
$$\Gamma_{l,k}=\displaystyle{\sup_{x\neq0,x\in\mathbf{R}^{n}}}\frac{(D_{l,k}x,x)}{(x,x)}.$$
The remained proof can be finished by the almost same arguments as \cite[Lemma 2.1]{N1}.
\end{proof}
Without loss of generality, we may still denote the subset of $\Omega$
with probability one in Lemma \ref{lem6.1} by $\Omega_{0}$.
\begin{lemma}\label{lem6.2}
Let $\gamma>\Gamma_{\infty}$.
There are positive constants $\delta,K$ only depending on
$\gamma,\Gamma_{\infty},D,\underline{D},$ and $C$ but not depending on $\omega\in\Omega_{0}$
such that if $\{w(i,\omega)\}_{i=l}^{l+k+1}$ satisfies
$$(\mathcal{L}^{\omega}w)_{i}\geq\gamma w_{i}, \ i\in I_{l,k},$$
then
\begin{equation}\label{l6.2.0}
w(i,\omega)\leq\max\{0,w(l;\omega)\}K{\mathrm{e}}^{\delta(l-i)}+
\max\{0,w(l+k+1;\omega)\}K{\mathrm{e}}^{\delta(i-k-l-1)},\ i\in\mathbf{I}_{l,k}.
\end{equation}
\end{lemma}
\begin{proof}
We first prove the result by assuming $w(l,\omega)=1$ and $w(l+k+1)\leq0$.
Let $z_{i}={\mathrm{e}}^{-\delta(l-i)}w_{i}-{\mathrm{e}}^{(l-i)}$, where $\delta$ will be chosen later on. Then
$w_{i}=(z_{i}+{\mathrm{e}}^{(l-i)}){\mathrm{e}}^{\delta(l-i)}
=z_{i}{\mathrm{e}}^{\delta(l-i)}+{\mathrm{e}}^{(1+\delta)(l-i)}$
and $z_{i}$ satisfies
$$(A^{\omega}_{\delta}z)_{i}+(c_{i}-\gamma)z_{i}+d_{i+1}(e^{-\delta}-1)z_{i}
+d_{i}(e^{\delta}-1)z_{i}\geq(\gamma-c_{i}+\frac{(A^{\omega}e^{(1+\delta)(l-\cdot)})_{i}}{e^{(1+\delta)(l-i)}})e^{l-i},
$$
where $(A_{\delta}^{\omega}z)_{i}=d_{i+1}e^{-\delta}(z_{i+1}-z_{i})+d_{i}e^{\delta}(z_{i-1}-z_{i})$.
After multiplying this equation by $z^{+}_{i}=\max\{0,z_{i}\}$ and summing $i$ from $l+1$ to $l+k$, we obtain
\begin{equation}\label{l6.2.1}
\begin{split}
&\ \ \ \sum_{i=l+1}^{l+k}[(A_{\delta}^{\omega}z^{+})_{i}z_{i}^{+}+c_{i}(z_{i}^{+})^{2}+
(d_{i+1}(e^{-\delta}-1)+d_{i}(e^{\delta}-1)-\gamma)(z_{i}^{+})^{2}]\\
&\geq\sum_{i=l+1}^{l+k}[(A_{\delta}^{\omega}z)_{i}z_{i}^{+}+c_{i}z_{i}z_{i}^{+}+
(d_{i+1}(e^{-\delta}-1)+d_{i}(e^{\delta}-1)-\gamma)z_{i}z_{i}^{+}]\\
&\geq\sum_{i=l+1}^{l+k}
(\gamma-c_{i}+\frac{(A^{\omega}e^{(1+\delta)(l-\cdot)})_{i}}{e^{(1+\delta)(l-i)}})e^{l-i}z_{i}^{+}.\\
\end{split}
\end{equation}
The first inequality follows from the fact that
$(A_{\delta}^{\omega}z)_{i}z_{i}\leq(A_{\delta}^{\omega}z^{+})_{i}z^{+}_{i}$. Noting that
$z^{+}_{l}=(w_{i}-1)^{+}=0$ and $z_{l+k+1}^{+}=(e^{\delta(k+1)}w_{l+k+1}-e^{-(k+1)})^{+}=0$, we have
$$\sum_{i=l+1}^{l+k}(A_{\delta}^{\omega}z^{+})_{i}z_{i}^{+}+c_{i}(z_{i}^{+})^{2}=(D_{l,k}^{\delta}z^{+},z^{+}),$$
where
\begin{gather*}
D_{l,k}^{\delta}:=
\left(
  \begin{array}{ccccc}
    \Gamma_{l+1}^{\delta} &   d_{l+2}^{\delta}   &            &          &            \\
    d_{l+2}^{\delta}      & \Gamma_{l+2}^{\delta}& d_{l+3}^{\delta}    &          &            \\
                & d_{l+3}^{\delta}      &\Gamma_{l+3}^{\delta} &\ddots    &            \\
                &              &\ddots       &\ddots    &d_{l+k+1}^{\delta}   \\
                &              &             &d_{l+k+1}^{\delta} &\Gamma_{l+k}^{\delta}\\
  \end{array}
\right),
\end{gather*}
with $\Gamma_{i}^{\delta}=c_{i}-d_{i}e^{\delta}-d_{i+1}e^{-\delta}$,
$d_{i}^{\delta}=d_{i}\cosh{\delta}$. Hence
$\Gamma_{l,k}^{\delta}\geq\frac{(D_{l,k}^{\delta}z_{i}^{+},z_{i}^{+})}{(z_{i}^{+},z_{i}^{+})}$, where
$\Gamma_{l,k}^{\delta}$ is the principal eigenvalue of $D_{l,k}^{\delta}$.
Moreover, by \cite[Corollary 6.3.4]{H1}, $\Gamma_{l,k}^{\delta}\leq\Gamma_{l,k}+C_{\delta}$, where
$C_{\delta}=\sqrt{3}D((e^{\delta}-1)^{2}+(1-e^{-\delta})^{2}+(e^{\delta}-e^{-\delta})^{2})^{\frac{1}{2}}$. Then
$\Gamma_{l,k}^{\delta}\leq\Gamma_{\infty}+C_{\delta}$ for any $l,k\in\mathbf{Z}$.
Hence from \eqref{l6.2.1} we have
\begin{equation*}
\begin{split}
&\ \ \ \sum_{i=l+1}^{l+k}(\Gamma_{\infty}+C_{\delta}+D(e^{\delta}-1)+\underline{D}(e^{-\delta}-1)-\gamma)(z_{i}^{+})^{2}\\
&\geq\sum_{i=l+1}^{l+k}
(\gamma-c_{i}+\frac{(A^{\omega}e^{(1+\delta)(l-\cdot)})_{i}}{e^{(1+\delta)(l-i)}})e^{l-i}z_{i}^{+}.\\
\end{split}
\end{equation*}
We choose an appropriate $\delta>0$ (depending on $\gamma,\gamma_{\infty},D,\underline{D},C,$ but not on $\omega,l,k$)
and denote $\beta=\gamma-\Gamma_{\infty}-C_{\delta}-D(e^{\delta}-1)-\underline{D}(e^{-\delta}-1)>0$. Thus
\begin{equation*}
\begin{split}
\sum_{i=l+1}^{l+k}\beta(z_{i}^{+})^{2}&\leq\sum_{i=l+1}^{l+k}
(c_{i}-\gamma-\frac{(A^{\omega}e^{(1+\delta)(l-\cdot)})_{i}}{e^{(1+\delta)(l-i)}})e^{l-i}z_{i}^{+}\\
&\leq\frac{1}{2\beta}\sum_{i=l+1}^{l+k}
(c_{i}-\gamma-\frac{(A^{\omega}e^{(1+\delta)(l-\cdot)})_{i}}{e^{(1+\delta)(l-i)}})^{2}e^{2(l-i)}
+\frac{\beta}{2}\sum_{i=l+1}^{l+k}(z_{i}^{+})^{2}.\\
\end{split}
\end{equation*}
i.e.
\begin{equation*}
\begin{split}
\sum_{i=l+1}^{l+k}(z_{i}^{+})^{2}
&\leq\frac{1}{{\beta}^{2}}\sum_{i=l+1}^{l+k}
(c_{i}-\gamma-\frac{(A^{\omega}e^{(1+\delta)(l-\cdot)})_{i}}{e^{(1+\delta)(l-i)}})^{2}e^{2(l-i)}\\
&\leq\frac{(D(e^{\delta}+e^{-\delta}+2)+C+\gamma)^{2}}{{\beta}^{2}}\sum_{i=l+1}^{\infty}e^{2(l-i)}\\
&\leq K_{0}(\gamma,\gamma_{\infty},D,\underline{D},C).
\end{split}
\end{equation*}
Now for $i\in\mathbf{I}_{l,k}$, we have
$$z_{i}\leq z_{i}^{+}\leq\sum_{j=l+1}^{i}z_{j}^{+}
\leq\sqrt{\sum_{j=l+1}^{i}1\cdot\sum_{j=l+1}^{i}(z_{j}^{+})^{2}}\leq K\sqrt{i-l}.$$
Hence
\begin{equation*}
\begin{split}
w_{i}&=(z_{i}+e^{(l-i)})e^{\delta(l-i)}\\
&\leq(K\sqrt{i-l}e^{\varepsilon(l-i)}+e^{(1+\varepsilon)(l-i)})e^{(\delta-\varepsilon)(l-i)}\\
&\leq K(\gamma,\gamma_{\infty},D,\underline{D},C,\varepsilon)e^{(\delta-\varepsilon)(l-i)}.
\end{split}
\end{equation*}
Then we are done for this assumption when taking $\varepsilon=\min\{\frac{\delta}{2},\frac{1}{2}\}$ and also rewriting $\delta=\delta-\varepsilon$.

For the assumption $w_{l}\leq0, w_{l+k+1}=1$, one can similarly find $w_{i}\leq Ke^{\delta(i-l-k-1)}$
by setting $w_{i}=z_{i}-e^{\delta(i-l-k-1)}$. For the general assumption, \eqref{l6.2.0} still holds because $\mathcal{L}^{\omega}$ is a linear operator .
\end{proof}
\begin{remark}\label{re6.1}
$\delta$ can be large if $\gamma$ is large enough. In fact, if we choose $\delta>0$ satisfying
$$C_{\delta}+D(e^{\delta}-1)=\frac{1}{2}(\gamma-\Gamma_{\infty})$$ whenever $\gamma$ is large, then
$\beta>\gamma-\Gamma_{\infty}-C_{\delta}-D(e^{\delta}-1)=\frac{1}{2}(\gamma-\Gamma_{\infty})>0$ and
thus the proof of Lemma \ref{lem6.1} is still valid.
\end{remark}
\begin{cor}\label{cor6.1}
Let $\omega\in\Omega_{0}$ and $\gamma>\Gamma_{\infty}$.
Consider $w=\{w_{i}\}_{i=l}^{+\infty}$ with $w_{l}\leq0$ satisfying $(A^{\omega}w)_{i}+c_{i}(\omega)w
\geq\gamma w_{i}$ for $i=l+1,l+2,\cdots$.
Then $w_{i}\leq0$ provided $\displaystyle{\liminf_{i\rightarrow+\infty}}w_{i}<+\infty$.
In particular, if $(A^{\omega}w)_{i}+c_{i}(\omega)w=\gamma w_{i}$ and $w_{l}=0$, then $w_{i}\equiv0$.
\end{cor}
\begin{proof}
There are a constant $M$ and a sequence $\{k_{n}\}_{n=1}^{+\infty}$ with
$k_{n}\rightarrow\infty$ as $n\rightarrow\infty$ such that
$w_{l+k_{n}+1}\leq M$ since $\displaystyle{\liminf_{i\rightarrow+\infty}}w_{i}<+\infty$.
Hence we can prove this result by using Lemma \ref{lem6.2} on $\mathbf{I}_{l,k_{n}}$
and also taking $n\rightarrow+\infty$.
\end{proof}

Consider $w^{(k)}=\{w^{(k)}_{i}\}_{i=0}^{k+1}$ be a solution of
$$\left\{
   \begin{aligned}
    (A^{\omega}w)_{i}+(c_{i}(\omega)-\gamma)w_{i}=0&,\ i\in\mathbf{I}_{0,k},\\
    w_{0}=1, w_{k+1}=0.
   \end{aligned}
   \right.$$
One can easily find $w_{i}^{(k)}\geq0$ on $i\in\mathbf{I}_{0,k}$ by Lemma \ref{lem6.2}
with $-w_{i}^{(k)}$ instead of $w$.
Moreover, $w_{i}^{(k)}\leq w_{i}^{(k+1)}$ for $i\in\mathbf{I}_{0,k}$ by Lemma \ref{lem6.2}
with $w_{i}^{(k)}-w_{i}^{(k+1)}$ instead of $w$. Therefore, $w^{k}_{i}$ is increasing in $k$,
and from \eqref{l6.2.0} we have $w_{i}^{(k)}\leq Ke^{-\delta i}$ for any $k\geq i$.
Let $u_{i}=\displaystyle{\lim_{k\rightarrow+\infty}}w_{i}^{(k)}$. Then one can easily verify that
$\{u_{i}\}_{i=0}^{+\infty}$ satisfies $(A^{\omega}u)_{i}+(c_{i}(\omega)-\gamma)u_{i}=0$ and
\begin{equation}\label{c6.1.1}
0\leq u_{i}\leq Ke^{-\delta i}\ \text{for}\ i\in\{1,2,\cdots\}.
\end{equation}
One can use the equality $(A^{\omega}u)_{i}+(c_{i}(\omega)-\gamma)u_{i}=0$ to extend this solution
on $\{-1,-2,\cdots\}$ by induction.
Then $u\in X_{-\infty}$ be the unique solution of
\begin{equation}\label{c6.1.2}
\left\{
   \begin{aligned}
    (A^{\omega}u)_{i}+(c_{i}(\omega)-\gamma)u_{i}=0,\ \ &i\in\mathbf{Z},\\
    u_{0}=1, \displaystyle{\lim_{i\rightarrow+\infty}}u_{i}=0.
   \end{aligned}
   \right.
\end{equation}
The uniqueness follows from Corollary \ref{cor6.1}. We sometimes denote it by
$u_{i}(\gamma,\omega)$ to emphasize that $u_{i}$ depends on
$\omega, \gamma$.
\begin{lemma}\label{lem6.3}
Let $\omega\in\Omega_{0}$, $\gamma>\Gamma_{\infty}$, and $u=\{u_{i}\}_{i\in\mathbf{Z}}$ be the unique solution of \eqref{c6.1.2}.
Then
\begin{equation}\label{l6.3.0}
u_{i}
\left\{
   \begin{aligned}
   \geq(\frac{\underline{D}}{\gamma-c+2D})^{i},&\ \ i\geq0,\\
   \leq(\frac{\underline{D}}{\gamma-c+2D})^{i},&\ \ i<0,
   \end{aligned}
   \right.
\end{equation}
where $c=\inf\limits_{i} c_{i}$. Moreover, $u_{i}>0$ for
$i\in\mathbf{Z}$ and $\displaystyle{\lim_{i\rightarrow-\infty}}u_{i}=+\infty.$
\end{lemma}
\begin{proof}
If there exists $i_{0}\in\mathbf{Z}$ such that $u_{i_{0}}\leq0$ then
$u_{i}\leq0$ for all $i>i_{0}$ by Corollary \ref{cor6.1} since $\displaystyle{\lim_{i\rightarrow+\infty}}u_{i}=0$.
From this and \eqref{c6.1.1}  we have $u_{i}=0$ for $i\geq\max\{1,i_{0}\}$, which yields that $u\equiv0$.
That is impossible since $u_{0}=1$. Note that
$d_{i+1}(\omega)(u_{i+1}-u_{i})+d_{i}(\omega)(u_{i-1}-u_{i})+c_{i}(\omega)u_{i}=\gamma u_{i}$, i.e.,
$d_{i+1}(\omega)\frac{u_{i+1}}{u_{i}}+d_{i}(\omega)\frac{u_{i-1}}{u_{i}}=(\gamma-\Gamma_{i}(\omega))$.
Hence $\frac{u_{i\pm1}}{u_{i}}\leq\frac{\gamma-c+2D}{\underline{D}}$ for $i\in\mathbf{Z}$.
This yields \eqref{l6.3.0} immediately.
Next, we prove that $\displaystyle{\lim_{i\rightarrow-\infty}}u_{i}=+\infty$.
If not, then there exist some constant $M$ and sequence $\{l_{n}\}_{n=1}^{\infty}$ with
$\displaystyle{\lim_{n\rightarrow+\infty}}l_{n}=+\infty$ such that $u_{-l_{n}}\leq M$.
By using Lemma \ref{lem6.2} on $\mathbf{I}_{-l_{n},2l_{n}}$, we obtain
$$u_{i}\leq u_{-l_{n}}K{\mathrm{e}}^{\delta(-l_{n}-i)}+
u_{l_{n}+1}K{\mathrm{e}}^{\delta(i-l_{n}-1)},\ i\in\mathbf{I}_{-l_{n},2l_{n}}.$$
Letting $n\rightarrow+\infty$, we can easily find that $u_{i}\leq0$ for any $i\in\mathbf{Z}$,
but that's impossible since $u_{0}=1$.
\end{proof}

A fact is that
$\bigcap\limits_{i\in\mathbf{Z}}(\pi_{i}\Omega_{0})$, a subset of $\Omega_{0}$, satisfies
$\mathbb{P}(\bigcap\limits_{i\in\mathbf{Z}}(\pi_{i}\Omega_{0}))=1$ since
$\pi_{i}$ is measure-preserving for any $i\in\mathbf{Z}$.
We still denote $\bigcap\limits_{i\in\mathbf{Z}}(\pi_{i}\Omega_{0})$ by $\Omega_{0}$.
Now from \eqref{c6.1.1} and \eqref{l6.3.0} we have
\begin{equation}\label{l6.3.1}
(\frac{\underline{D}}{\gamma-c+2D})^{i}\leq u_{i}(\gamma,\omega)\leq Ke^{-\delta i}
\end{equation}
for $i\geq0$, $\omega\in\Omega_{0}$ and $\gamma>\Gamma_{\infty}$. Moreover, for any fixed $i\in\mathbf{Z}$,
both $j\longmapsto u_{i+j}(\gamma,\omega)$ and
$j\longmapsto u_{i}(\gamma,\omega)u_{j}(\gamma,\pi_{i}\omega)$ satisfy
\begin{equation*}
\left\{
   \begin{aligned}
    (A^{\pi_{i}\omega}v)_{j}+(c_{j}(\pi_{i}\omega)-\gamma)v_{j}=0,\\
    v_{0}=u_{i}(\gamma,\omega), \displaystyle{\lim_{j\rightarrow+\infty}}v_{j}=0,
   \end{aligned}
   \right.
\end{equation*}
since $c_{i+j}(\omega)=c_{j}(\pi_{i}\omega), d_{i+j}(\omega)=d_{j}(\pi_{i}\omega)$. Therefore,
by Corollary \ref{cor6.1}, we have
\begin{lemma}
Let $\omega\in\Omega_{0}$, $\gamma>\Gamma_{\infty}$, and $u=\{u_{i}\}_{i\in\mathbf{Z}}$
be the unique solution of \eqref{c6.1.2}. Then
\begin{equation}\label{l6.4.1}
u_{i+j}(\gamma,\omega)=u_{i}(\gamma,\omega)u_{j}(\gamma,\pi_{i}\omega)
\end{equation}
for any $i,j\in\mathbf{Z}$.
\end{lemma}

Now suppose that $\gamma<\Gamma_{\infty}$. Then for any $\omega\in\Omega_{0}$,  we have
$\gamma<\Gamma_{l,k}(\omega)\leq\Gamma_{\infty}$ for some $l,k$ depending on $\omega$ by Lemma \ref{lem6.1}. Let
$z_{i}(t):=e^{(\Gamma_{l,k}(\omega)-\gamma)t}{\phi}_{i}^{l,k}(\omega)$, where $(\Gamma_{l,k}(\omega),\phi^{l,k}(\omega))$ is the principal eigenpair of \eqref{t6.1.1}.
Then we have
$\overset{.}z_{i}(t)-A^{\omega}z_{i}-c_{i}(\omega)z_{i}+\gamma z_{i}=0$ and $z_{i}\rightarrow+\infty$ as $t\rightarrow+\infty$.
If there exists $w\in X_{-\infty}$ with $w_{i}>0$ satisfying
$(A^{\omega}w)_{i}+(c_{i}(\omega)-\gamma)w_{i}=0$ for $i\in\mathbf{Z}$,
then the maximum principle would imply that for some suitable constant $\kappa>0$, $w_{i}\geq\kappa z_{i}(t)$
must hold for any $t\geq0$ and $i\in\mathbf{I}_{l,k}$.
That's impossible since $z_{i}\rightarrow+\infty$ as $t\rightarrow+\infty$.
Therefore, for $\gamma<\Gamma_{\infty}$, \eqref{c6.1.2} doesn't possess a positive solution.
\begin{thm}\label{thm6.2}
There exists a measurable set $\Omega^{\prime}\subset\Omega_{0}$ with $\mathbb{P}(\Omega^{\prime})=1$
such that
$$\mu(\gamma):=\displaystyle{\lim_{i\rightarrow+\infty}}\frac{-\ln(u_{i}(\gamma,\omega))}{i}
=\displaystyle{\lim_{i\rightarrow-\infty}}\frac{-\ln(u_{i}(\gamma,\omega))}{i}=:\nu(\gamma),$$
where $\omega\in\Omega^{\prime},\ \gamma>\Gamma_{\infty}$, and $u=u(\gamma,\omega)$
is the unique solution of \eqref{c6.1.2}.
Moreover, $\mu(\gamma)$, which does not depend on $\omega$, is strictly increasing,
concave and converges to $+\infty$ as $\gamma$ tends to $+\infty$.
\end{thm}
\begin{proof}
For each integer $i\geq1$, we iterate \eqref{l6.4.1} $i-1$ times to obtain
$$\ln(u_{i}(\gamma,\omega))=\ln(\prod_{k=0}^{i-1}u_{1}(\gamma,\pi_{k}\omega))=
\sum_{k=0}^{i-1}\ln(u_{1}(\gamma,\pi_{k}\omega)).$$
One can easily check that $\omega\in\Omega\longmapsto\ln(u_{1}(\gamma,\pi_{k}\omega))$ is a measurable
function for $k\in\{0,1,\cdots,i-1\}$, and $\ln(u_{1}(\gamma,\omega))\in\mathcal{L}^{1}(\Omega)$ by \eqref{l6.3.1}.
Hence the ergodic theorem implies that
\begin{equation*}
\begin{split}
\mu(\gamma,\omega):&=\displaystyle{\lim_{i\rightarrow+\infty}}\frac{-\ln(u_{i}(\gamma,\omega))}{i}\\
&=-\displaystyle{\lim_{i\rightarrow+\infty}}\frac{1}{i}\sum_{k=0}^{i-1}\ln(u_{1}(\gamma,\pi_{k}\omega))\\
&=-\mathbb{E}(\ln(u_{1}(\gamma,\omega)))\\
&=-\mathbb{E}(\ln(u_{1}(\gamma,\pi_{j}\omega)))\\
&=\mu(\gamma,\pi_{j}\omega)
\end{split}
\end{equation*}
for any $j\in\mathbf{Z}$. The ergodicity assumption implies that $\mu(\gamma,\omega)$
should be a constant almost surely and we write it by $\mu(\gamma)$. Similarly, $\nu(\gamma,\omega)$
should be a constant almost surely and we write it by $\nu(\gamma)$. Without loss of generality, we may assume
$\mu(\gamma,\omega)=\mu(\gamma)$ and $\nu(\gamma,\omega)=\nu(\gamma)$ for $\omega\in\Omega_{\gamma}$, where
$\Omega_{\gamma}\subset\Omega_{0}$ with $\mathbb{P}(\Omega_{\gamma})=1$.

Now consider $\Gamma_{\infty}<\gamma_{0}<\gamma_{1}=\frac{\gamma_{0}+\gamma_{2}}{2}<\gamma_{2},\ \omega\in\Omega_{0}$ and
$u^{(k)}:=u(\gamma_{k},\omega), k=0,1,2$. $u^{(k)}$ satisfies
\begin{equation*}
\left\{
   \begin{aligned}
    (A^{\omega}u^{(k)})_{i}+(c_{i}(\omega)-\gamma_{k})u_{i}^{(k)}=0,\\
    u_{0}^{(k)}=1, \displaystyle{\lim_{i\rightarrow+\infty}}u_{i}^{(k)}=0.
   \end{aligned}
   \right.
\end{equation*}
One can compute that
$$\frac{(A^{\omega}\phi\psi)_{i}}{\phi_{i}\psi_{i}}=\frac{(A^{\omega}\phi)_{i}}{\phi_{i}}
+\frac{(A^{\omega}\psi)_{i}}{\psi_{i}}
+d_{i+1}\frac{\partial_{i}^{+}\phi}{\phi_{i}}\frac{\partial_{i}^{+}\psi}{\psi_{i}}
+d_{i}\frac{\partial_{i}^{-}\phi}{\phi_{i}}\frac{\partial_{i}^{-}\psi}{\psi_{i}},$$
where $\partial_{i}^{\pm}\phi=\phi_{i\pm1}-\phi_{i}$. Then
$$\frac{(A^{\omega}u^{(k)})_{i}}{u_{i}^{(k)}}=\frac{(A^{\omega}\sqrt{u^{(k)}}\sqrt{u^{(k)}})_{i}}{u_{i}^{(k)}}
=2\frac{(A^{\omega}\sqrt{u^{(k)}})_{i}}{\sqrt{u^{(k)}_{i}}}
+d_{i+1}{\Big(\frac{\partial_{i}^{+}\sqrt{u^{(k)}}}{\sqrt{u^{(k)}_{i}}}\Big)}^{2}
+d_{i}{\Big(\frac{\partial_{i}^{-}\sqrt{u^{(k)}}}{\sqrt{u^{(k)}_{i}}}\Big)}^{2},$$
that is,
$$\frac{(A^{\omega}\sqrt{u^{(k)}})_{i}}{\sqrt{u^{(k)}_{i}}}=
\frac{1}{2}\Big[\frac{(A^{\omega}u^{(k)})_{i}}{u_{i}^{(k)}}-
d_{i+1}{\Big(\frac{\partial_{i}^{+}\sqrt{u^{(k)}}}{\sqrt{u^{(k)}_{i}}}\Big)}^{2}
-d_{i}{\Big(\frac{\partial_{i}^{-}\sqrt{u^{(k)}}}{\sqrt{u^{(k)}_{i}}}\Big)}^{2}\Big].$$
Using this, we have
\begin{equation*}
\begin{split}
\frac{(A^{\omega}\sqrt{u^{(0)}u^{(2)}})_{i}}{\sqrt{u_{i}^{(0)}u_{i}^{(2)}}}
&=\frac{1}{2}\Big[\frac{(A^{\omega}u^{(0)})_{i}}{u_{i}^{(0)}}-
d_{i+1}{\Big(\frac{\partial_{i}^{+}\sqrt{u^{(0)}}}{\sqrt{u^{(0)}_{i}}}\Big)}^{2}
-d_{i}{\Big(\frac{\partial_{i}^{-}\sqrt{u^{(0)}}}{\sqrt{u^{(0)}_{i}}}\Big)}^{2}\Big]\\
&\ \ +\frac{1}{2}\Big[\frac{(A^{\omega}u^{(2)})_{i}}{u_{i}^{(2)}}-
d_{i+1}{\Big(\frac{\partial_{i}^{+}\sqrt{u^{(2)}}}{\sqrt{u^{(2)}_{i}}}\Big)}^{2}
-d_{i}{\Big(\frac{\partial_{i}^{-}\sqrt{u^{(2)}}}{\sqrt{u^{(2)}_{i}}}\Big)}^{2}\Big]\\
&\ \ +d_{i+1}\frac{\partial_{i}^{+}\sqrt{u^{(0)}}}{\sqrt{u^{(0)}_{i}}}
\frac{\partial_{i}^{+}\sqrt{u^{(2)}}}{\sqrt{u^{(2)}_{i}}}
+d_{i}\frac{\partial_{i}^{-}\sqrt{u^{(0)}}}{\sqrt{u^{(0)}_{i}}}
\frac{\partial_{i}^{-}\sqrt{u^{(2)}}}{\sqrt{u^{(2)}_{i}}}\\
&=\gamma_{1}-c_{i}-
\frac{d_{i+1}}{2}{\Big(\frac{\partial_{i}^{+}\sqrt{u^{(0)}}}{\sqrt{u^{(0)}_{i}}}+
\frac{\partial_{i}^{+}\sqrt{u^{(2)}}}{\sqrt{u^{(2)}_{i}}}\Big)}^{2}\\
&\ \ -\frac{d_{i}}{2}{\Big(\frac{\partial_{i}^{-}\sqrt{u^{(0)}}}{\sqrt{u^{(0)}_{i}}}+
\frac{\partial_{i}^{-}\sqrt{u^{(2)}}}{\sqrt{u^{(2)}_{i}}}\Big)}^{2}\\
&\leq\gamma_{1}-c_{i},
\end{split}
\end{equation*}
i.e., $(A^{\omega}\sqrt{u^{(0)}u^{(2)}})_{i}+c_{i}\sqrt{u_{i}^{(0)}u_{i}^{(2)}}
\leq\gamma_{i}\sqrt{u_{i}^{(0)}u_{i}^{(2)}}$. Using Corollary \ref{cor6.1} to $u^{(1)}-\sqrt{u^{(0)}u^{(2)}}$,
one can find that $u^{(1)}\leq\sqrt{u^{(0)}u^{(2)}}$, i.e., ${(u^{(1)})}^{2}\leq u^{(0)}u^{(2)}$.
Thus
$$2\ln u_{i}^{(1)}(\frac{\gamma_{0}+\gamma_{2}}{2},\omega)=2\ln u_{i}^{(1)}(\gamma_{1},\omega)
\leq\ln u_{i}^{(0)}(\gamma_{0},\omega)+\ln u_{i}^{(2)}(\gamma_{2},\omega),$$
which yields that
$2\mu(\frac{\gamma_{0}+\gamma_{2}}{2},\omega)\geq\mu(\gamma_{0},\omega)+\mu(\gamma_{2},\omega)$,
i.e., $\mu(\cdot,\omega)$ is concave.

Next, letting ${\gamma}^{\prime}>\gamma$ and $\varepsilon>0$ small enough, we have
\begin{equation*}
\begin{split}
\frac{(A^{\omega}e^{-\varepsilon\cdot}u)_{i}}{e^{-\varepsilon i}u_{i}}
&=\frac{(A^{\omega}u)_{i}}{u_{i}}+\frac{(A^{\omega}e^{-\varepsilon\cdot})_{i}}{e^{-\varepsilon i}}
+d_{i+1}\frac{\partial_{i}^{+}e^{-\varepsilon\cdot}}{e^{-\varepsilon i}}\frac{\partial_{i}^{+}u}{u_{i}}
+d_{i}\frac{\partial_{i}^{-}e^{-\varepsilon\cdot}}{e^{-\varepsilon i}}\frac{\partial_{i}^{-}u}{u_{i}}\\
&=\gamma^{\prime}-c_{i}-(\gamma^{\prime}-\gamma)
+d_{i+1}\frac{u_{i+1}}{u_{i}}(e^{-\varepsilon}-1)+d_{i}\frac{u_{i-1}}{u_{i}}(e^{\varepsilon}-1)\\
&\leq\gamma^{\prime}-c_{i},
\end{split}
\end{equation*}
i.e.,
$(A^{\omega}e^{-\varepsilon\cdot}u)_{i}\leq(\gamma^{\prime}-c_{i}){e^{-\varepsilon i}u_{i}}$,
which yields that $u_{i}(\gamma^{\prime},\omega)\leq e^{-\varepsilon i}u_{i}(\gamma,\omega)$, hence
$\mu(\gamma^{\prime},\omega)\geq\mu(\gamma,\omega)+\varepsilon>\mu(\gamma,\omega)$.
Moreover, by \eqref{l6.3.1} and Remark \ref{re6.1}, one can easily find that
$\displaystyle{\lim_{\gamma\rightarrow+\infty}}\mu(\gamma,\omega)=+\infty$.

The concavity and monotonicity yield that $\mu(\gamma,\omega)$ is continuous in $\gamma$
for any fixed $\omega\in\Omega_{0}$.
Let $\{\gamma_{i}\}_{i=1}^{+\infty}=(\Gamma_{\infty},+\infty)\cap\mathbf{Q}$, and
$\Omega^{\prime}:={\bigcap}_{i=1}^{+\infty}\Omega_{\gamma_{i}}$. Then $\mathbb{P}(\Omega^{\prime})=1$
and $\mu(\gamma_{i},\omega)=\mu(\gamma_{i})$ for any $\omega\in\Omega^{\prime}$ and $i\in\mathbf{Z}_{+}$. Hence for
any $\gamma\in(\Gamma_{\infty},+\infty)$, there exists a sequence $\{\gamma_{i_{n}}\}_{n=1}^{+\infty}$
such that $\gamma_{i_{n}}\rightarrow\gamma$ as $n\rightarrow+\infty$, and
$\mu(\gamma,\omega)=\displaystyle{\lim_{n\rightarrow+\infty}}u(\gamma_{i_{n}},\omega)
=\displaystyle{\lim_{n\rightarrow+\infty}}\mu(\gamma_{i_{n}})$ for any $\omega\in\Omega^{\prime}$.
That is to say, $\mu(\gamma,\omega)$ is a constant with respect to $\omega\in\Omega^{\prime}$,
so does $\nu(\gamma,\omega)$.

Finally, we will prove $\mu(\gamma)=\nu(\gamma)$. If not, then there exists $S\subset\Omega^{\prime}$ such that $\mathbb{P}(S)>\frac{1}{2}$ and $N<+\infty$ such that
$$\Big|\frac{\ln u_{N}(\gamma,\omega)}{N}+\mu(\gamma)\Big|<\frac{|\mu(\gamma)-\nu(\gamma)|}{2},$$
$$\Big|\frac{\ln u_{-N}(\gamma,\omega)}{-N}+\nu(\gamma)\Big|<\frac{|\mu(\gamma)-\nu(\gamma)|}{2}$$
for any $\omega\in S$. Then
\begin{equation*}
\begin{split}
\Big|\frac{\ln u_{N}(\gamma,\pi_{-2N}\omega)}{N}+\mu(\gamma)\Big|
&=\Big|\frac{\ln u_{-N}(\gamma,\omega)}{-N}-\mu(\gamma)\Big|\\
&=\Big|\frac{\ln u_{-N}(\gamma,\omega)}{-N}+\nu(\gamma)-(\nu(\gamma)+\mu(\gamma))\Big|\\
&\geq|\nu(\gamma)+\mu(\gamma)|-\Big|\frac{\ln u_{-N}(\gamma,\omega)}{-N}+\nu(\gamma)\Big|\\
&\geq\frac{|\mu(\gamma)-\nu(\gamma)|}{2},
\end{split}
\end{equation*}
and the last inequality holds since both $\mu(\gamma)$ and $\nu(\gamma)$ are positive by Lemma \ref{lem6.3} and \eqref{l6.3.1}.
Thus $S\cap\pi_{-2N}S=\emptyset$, a contradiction with $\mathbb{P}(S)=\mathbb{P}(\pi_{-2N}S)>\frac{1}{2}.$
\end{proof}
\begin{thm}\label{thm6.3}
There is a measurable set $\Omega^{\prime\prime}\subset\Omega$
with $\mathbb{P}(\Omega^{\prime\prime})=1$ such that $\underline{{\lambda}_{1}}(p,-\infty;\omega)\geq\Gamma_{\infty}$ for any $\omega\in\Omega^{\prime\prime}$ and for any $p\in\mathbf{R}$.
\end{thm}
We will prove this theorem later on, and we will first use it to prove Theorem \ref{thm6.1}.

\begin{proof}[Proof of Theorem \ref{thm6.1}]
First we still denote $\Omega^{\prime\prime}\cap\Omega^{\prime}$, a set
with probability one, by $\Omega^{\prime}$. As $\mu(\gamma)$ is strictly increasing and nonnegative on
$(\Gamma_{\infty},+\infty)$, we can define
$\mu(\Gamma_{\infty}):=\displaystyle{\lim_{\gamma\rightarrow\Gamma_{\infty}^{+}}}\mu(\gamma)=p_{r}\geq0$.
The function $\mu$ admits an inverse $k:[p_{r},+\infty)\rightarrow[\Gamma_{\infty},+\infty)$.
For any $p>p_{r}$ and $\omega\in\Omega^{\prime}$, let $\phi_{i}(k(p),\omega)=e^{pi}u_{i}(k(p),\omega)>0$,
where $u_{i}(k(p),\omega)$ is the solution of \eqref{c6.1.2}. Then $\phi$ satisfies
$$(L_{-p}^{\omega}\phi)_{i}={\mathrm{e}}^{pi}(\mathcal L^{\omega}{\mathrm{e}}^{-p\cdot}\phi)_{i}
={\mathrm{e}}^{pi}(\mathcal L^{\omega}u)_{i}={\mathrm{e}}^{pi}k(p)u_{i}=k(p)\phi_{i},\ i\in\mathbf{Z}.$$

Moreover, ${\Big\{\frac{{\phi}_{i\pm1}-{\phi}_{i}}{{\phi}_{i}}\Big\}}_{i=-\infty}^{\infty}\in{\ell}^{\infty}$
since $\frac{u_{i\pm1}}{u_{i}}\leq\frac{\gamma-c+2D}{\underline{D}}$, and
$$\displaystyle{\lim_{i\rightarrow\pm\infty}}\frac{\ln\phi_{i}(k(p),\omega)}{i}=-\mu(k(p))+p=0$$
for any $\omega\in\Omega^{\prime}$. Hence $\phi\in\mathcal{A}_{-\infty}$. Combining this with Corollary \ref{cor2.1},
we can prove that
\begin{equation}\label{t6.3.1}
\underline{{\lambda}_{1}}(p,-\infty;\omega)=\overline{{\lambda}_{1}}(p,-\infty;\omega)=k(p)
\end{equation}
for any $p>p_{r}$ and $\omega\in\Omega^{\prime}$. The continuity of $k(p)$, $\overline{{\lambda}_{1}}(p,-\infty;\omega)$ and $\underline{{\lambda}_{1}}(p,-\infty;\omega)$ yields that \eqref{t6.3.1} holds for $p=p_{r}$ and all $\omega\in\Omega^{\prime}$.

Similarly, one can prove the existence of $p_{l}\leq0$ such that for any $p<p_{l}$ and
$\omega\in\Omega^{\prime}$, and there exists a solution
$\phi\in\mathcal{A}_{-\infty}$ of $L_{-p}^{\omega}\phi=\tilde{k}(p)\phi$,
where $\tilde{k}:(-\infty,p_{l}]\rightarrow[\Gamma_{\infty},+\infty)$
is strictly decreasing. Also,
$$\underline{{\lambda}_{1}}(p,-\infty;\omega)=\overline{{\lambda}_{1}}(p,-\infty;\omega)=\tilde{k}(p)$$
for any $p\leq p_{r}$ and $\omega\in\Omega^{\prime}$.

Finally, the convexity of $\overline{{\lambda}_{1}}(p,-\infty;\omega)$ from Lemma \ref{lem3.2} yields that
$$\underline{{\lambda}_{1}}(p,-\infty;\omega)\leq\overline{{\lambda}_{1}}(p,-\infty;\omega)\leq\Gamma_{\infty}$$
for any $p\in[p_{l},p_{r}]$ and $\omega\in\Omega^{\prime}$. Then by Theorem \ref{thm6.3}
$$\underline{{\lambda}_{1}}(p,-\infty;\omega)=\overline{{\lambda}_{1}}(p,-\infty;\omega)=\Gamma_{\infty}$$
for any $p\in[p_{l},p_{r}]$ and $\omega\in\Omega^{\prime}$.
\end{proof}
We are now in the position to prove Theorem \ref{thm6.3}.

\begin{proof}[Proof of Theorem \ref{thm6.3}] Let $\lambda\in\mathbf{R}$ be a constant. Denote
$(\mathcal{L}^{\omega}_{\lambda}\phi)_{i}=(\mathcal{L}^{\omega}\phi)_{i}+\lambda\phi_{i}$ and
$(L^{\omega}_{p,\lambda}\phi)_{i}={\mathrm{e}}^{-pi}(\mathcal{L}^{\omega}_{\lambda}{\mathrm{e}}^{p\cdot}\phi)_{i}$
for $i\in\mathbf{Z}$, where $\phi\in X_{-\infty}$.
Then for any fixed $\lambda$, we can find $\Gamma_{\infty}(\lambda)$ by Lemma \ref{lem6.1} and
define $\underline{{\lambda}_{1}}(p,-\infty;\omega,\lambda)$ by Definition \ref{def2.1}. Both of them are related to $\mathcal{L}^{\omega}_{\lambda}$.
We write down $\lambda$ here
to emphasize that $\Gamma_{\infty}$ and $\underline{{\lambda}_{1}}(p,-\infty;\omega)$ depend on $\lambda$.
Moreover, it is easy to see that $\Gamma_{\infty}(\lambda)=\Gamma_{\infty}(0)+\lambda$ and
$\underline{{\lambda}_{1}}(p,-\infty;\omega,\lambda)=\underline{{\lambda}_{1}}(p,-\infty;\omega,0)+\lambda$.
Hence we only need to show that $\underline{{\lambda}_{1}}(p,-\infty;\omega,\lambda)\geq0$ provided that
$\Gamma_{\infty}(\lambda)>0$. Without loss of generality, we assume that $\Gamma_{\infty}=\Gamma_{\infty}(0)>0,$
and then prove that
$\underline{{\lambda}_{1}}(p,-\infty;\omega)=\underline{{\lambda}_{1}}(p,-\infty;\omega,0)\geq0$.
We do this in five steps.\\
Step 1: For any $k\in\mathbf{Z}_{+}, (j,\omega)\in\mathbf{Z}\times\Omega$, let
$B(j,k)=\mathbf{I}_{j-k-1,2k+1}=\{j-k,j-k+1,\cdots,j,\cdots,j+k\}$ and
$$(\chi^{j,k}(\omega),\Lambda_{j,k}(\omega))=(\phi^{j-k-1,2k+1}(\omega),\Gamma_{j-k-1,2k+1}(\omega))$$
for convenience, where $(\phi^{j-k-1,2k+1}(\omega),\Gamma_{j-k-1,2k+1}(\omega))$
is the principal eigenpair of \eqref{t6.1.1}. We reduce that $\max\limits_{i\in B(j,k)}\chi_{i}^{j,k}(\omega)=1$.
Then
\begin{equation}\label{t6.3.2}
\left\{
   \begin{aligned}
   {(\mathcal{L}^{\omega}\chi^{j,k})}_{i}=(A^{\omega}\chi^{j,k})_{i}
    +c_{i}(\omega)\chi_{i}^{j,k}=\Lambda_{j,k}(\omega)\chi_{i}^{j,k},\ \ i\in B(j,k),\\
    \chi_{j-k-1}^{j,k}=\chi_{j+k+1}^{j,k}=0,\ \chi_{i}^{j,k}>0,\ \ i\in B(j,k),\\
    \max\limits_{i\in B(j,k)}\chi_{i}^{j,k}(\omega)=1.
   \end{aligned}
   \right.
   \end{equation}
It is easy to check that $\omega\longmapsto(\chi_{i}^{j,k}(\omega),\Lambda_{j,k}(\omega))$ is a measurable function
for any $j\in\mathbf{Z},\ k\in\mathbf{Z_{+}}$, and $i\in B(j,k)$.

Now for $(j,n,k,\omega)\in\mathbf{Z}^{2}\times\mathbf{Z}_{+}\times\Omega$, we define
$\{\psi_{i}\}_{i=j-k-1}^{j+k+1}$ with $\psi_{i}=\chi_{i+n}^{j+n,k}(\omega)$. Then
\begin{equation*}
\begin{split}
{(\mathcal{L}^{\pi_{n}\omega}\psi)}_{i}
&=d_{i+1}(\pi_{n}\omega)(\psi_{i+1}-\psi_{i})+d_{i}(\pi_{n}\omega)(\psi_{i-1}-\psi_{i})
  +c_{i}(\pi_{n}\omega)\psi_{i}\\
&={(\mathcal{L}^{\omega}\chi^{j+n,k})}_{i+n}=\Lambda_{j+n,k}(\omega)\chi_{i+n}^{j+n,k}(\omega)\\
&=\Lambda_{j+n,k}(\omega)\psi_{i}
\end{split}
\end{equation*}
for $i\in B(j,k)$, $\psi_{j-k-1}=\psi_{j+k+1}=0,\ \psi_{i}>0$ for all $i\in B(j,k)$, and
$\max\limits_{i\in B(j,k)}\psi_{i}=1$. Noting that the solution of \eqref{t6.3.2} is unique, we have
\begin{equation}\label{t6.3.3}
\psi_{i}=\chi_{i+n}^{j+n,k}(\omega)=\chi_{i}^{j,k}(\pi_{n}\omega)\ \text{for}\ i\in B(j,k)
\ \text{and}\ \Lambda_{j+n,k}(\omega)=\Lambda_{j,k}(\pi_{n}\omega).
\end{equation}
This means that the eigenelements are random stationary ergodic in $(j,\omega)$.
Moreover, for any given
$0<\overline{\gamma}<\Gamma_{\infty}$ one can define
$k(j,\omega)=\min\{k|\ \Lambda_{j,k}(\omega)\geq\overline{\gamma}\}$ for
$(j,\omega)\in\mathbf{Z}\times\Omega_{0}$ since
$\Lambda_{j,k}(\omega)$ is increasing in $k$ and
$\displaystyle{\lim_{k\rightarrow+\infty}}\Lambda_{j,k}(\omega)=\Gamma_{\infty}$.
It follows from \eqref{t6.3.3} that $k(j+n,\omega)=k(j,\pi_{n}\omega)$ for any
$(j,n,\omega)\in\mathbf{Z}^{2}\times\Omega_{0}$.

Step 2: Consider the following equation
\begin{equation}\label{t6.3.4}
(L_{p}^{\omega}\phi)_{i}={\phi}^{2}_{i},\ i\in\mathbf{Z}
\end{equation}
for $\omega\in\Omega_{0}$. Note that
$(L_{p}^{\omega}\phi)_{i}={(A_{-p}^{\omega}\phi)}_{i}+\tilde{c}_{i}(\omega)\phi_{i}$,
where
$${(A_{-p}^{\omega}\phi)}_{i}=d_{i+1}(\omega)e^{p}(\phi_{i+1}-\phi_{i})+d_{i}(\omega)e^{-p}(\phi_{i-1}-\phi_{i})$$
and $\tilde{c}_{i}(\omega)=d_{i+1}(\omega)(e^{p}-1)+d_{i}(\omega)(e^{-p}-1)+c_{i}(\omega)$. Then
$\overline{\phi}\equiv\sup \limits_{i}\tilde{c}_{i}>0$ is a supersolution of \eqref{t6.3.4}, i.e.,
${(A_{-p}^{\omega}\overline{\phi})}_{i}+\tilde{c}_{i}(\omega)\overline{\phi_{i}}\leq{\overline{\phi}_{i}}^{2}$.
Moreover, it is easy to check that
\begin{equation*}
\underline{\phi}_{i}^{j}(\omega):=
\left\{
   \begin{aligned}
    \Lambda_{j,k(j,\omega)}(\omega)\chi_{i}^{j,k(j,\omega)}(\omega)e^{-p(i-j+\text{sgn}(p)k(j,\omega))},\ \
    i\in B(j,k(j,\omega)),\\
    0,\ \ i\notin B(j,k(j,\omega)),
   \end{aligned}
   \right.
\end{equation*}
satisfies
${(A_{-p}^{\omega}\underline{\phi}^{j})}_{i}+\tilde{c}_{i}(\omega)\underline{\phi}^{j}_{i}
\geq({\underline{\phi}^{j}_{i}})^{2}$ since $\max\limits_{i\in B(j,k)}\chi_{i}^{j,k}(\omega)=1$.
In other words, $\underline{\phi}^{j}$ is a subsolution. We also have
$0\leq\underline{\phi}_{i}^{j}(\omega)\leq\Lambda_{j,k(j,\omega)}(\omega)\leq\Gamma_{\infty}\leq\overline{\phi}.$

Step 3: Let $M>0$. We will prove two claims.\\
Claim 1. Assume that $\phi^{(k)}=\{\phi^{(k)}_{i}\}_{i=-k-1}^{k+1}$ satisfies
\begin{equation}\label{t6.3.5}
\left\{
   \begin{aligned}
    ((MI-A_{-p}^{\omega})\phi^{(k)})_{i}\geq0,&\ \ i\in B(0,k),\\
    \phi^{(k)}_{\pm(k+1)}\geq0,
   \end{aligned}
   \right.
\end{equation}
where $I$ is an identity matrix. Then $\phi^{(k)}_{i}\geq0$ for $i\in B(0,k)$.
Moreover, $\phi^{(k)}_{i}>0$ for $i\in B(0,k)$ provided $((MI-A_{-p}^{\omega})\phi^{(k)})_{i}>0$
for some $i\in B(0,k)$.\\
Proof of Claim 1: Assume that $\phi^{(k)}_{i}$ reaches its minimum at $i_{0}\in B(0,k)$, i.e.,
$\phi^{(k)}_{i_{0}}=\min\limits_{i\in B(0,k)}{\phi^{(k)}_{i}}$. If $\phi^{(k)}_{i_{0}}<0$, then
one can conclude that
$$((MI-A_{-p}^{\omega})\phi^{(k)})_{i_{0}}<0,$$
which contradicts \eqref{t6.3.5}. Hence $\phi^{(k)}_{i_{0}}\geq0$. Furthermore,
suppose that $((MI-A_{-p}^{\omega})\phi^{(k)})_{i}>0$ for some $i\in B(0,k)$.
Then $\min\limits_{i\in B(0,k)}{\phi^{(k)}_{i}}>0$. If not, then there must exist
$i\in B(0,k)$ such that $\phi^{(k)}_{i}=\min\limits_{i\in B(0,k)}{\phi^{(k)}_{i}}=0$ and $\phi^{(k)}_{i-1}+\phi^{(k)}_{i+1}>0$. Hence at $i$ we have
$$((MI-A_{-p}^{\omega})\phi^{(k)})_{i}=
-d_{i+1}(\omega)e^{p}\phi^{(k)}_{i+1}-d_{i}(\omega)e^{-p}\phi^{(k)}_{i-1}<0,$$
which contradicts \eqref{t6.3.5}.\\
Claim 2: Assume that $\phi\in X_{-\infty}$ with $\sup\limits_{i}\phi_{i}<+\infty$ satisfying
$((MI-A_{-p}^{\omega})\phi)_{i}\leq0$ for $i\in\mathbf{Z}$. Then $\phi_{i}\leq0$ for $i\in\mathbf{Z}$.\\
Proof of Claim 2: Assume by contradiction that $\phi_{i_{0}}>0$ for some $i_{0}\in B(0,k)$. Then we have either
$\phi_{i_{0}-1}>\phi_{i_{0}}$ or $\phi_{i_{0}+1}>\phi_{i_{0}}$. In fact, $\phi_{i_{0}\pm1}\leq\phi_{i_{0}}$
would imply $((MI-A_{-p}^{\omega})\phi)_{i_{0}}>0$. That is impossible. If $\phi_{i_{0}-1}>\phi_{i_{0}}$
holds, then one can easily find that $\phi_{i_{0}-(i+1)}>\phi_{i_{0}-i}$ for $i=1,2,\cdots$ by induction.
Hence $\{\phi_{i_{0}-i}\}_{i=1}^{\infty}$ is strictly increasing. Moreover,
$\lim\limits_{i\to\infty}\phi_{i_{0}-i}$ exists since $\sup\limits_{i}\phi_{i}<+\infty$. Then we have
$$0\geq\lim\limits_{i\to\infty}((MI-A_{-p}^{\omega})\phi)_{i_{0}-i}=M\lim\limits_{i\to\infty}\phi_{i_{0}-i}>0,$$
which is a contradiction. We have thus proved that $\phi_{i_{0}-1}>\phi_{i_{0}}$ fails. Therefore,
$\phi_{i_{0}+1}>\phi_{i_{0}}$, which yields that $\{\phi_{i_{0}+i}\}_{i=1}^{\infty}$ is strictly increasing.
Then one can obtain a contradiction by the same argument as we just did.
Hence $\phi_{i}\leq0$ for $i\in\mathbf{Z}$.

Step 4: Now choose $M$ large enough such that
$$(M+\tilde{c}_{i}-t)t\geq(M+\tilde{c}_{i}-s)s$$
for any $i\in\mathbf{Z},\ t,s\in[0,\overline{\phi}]$ with $t\geq s$. Fix $j\in\mathbf{Z}$. For any given $\varphi\in X_{-\infty}$
with $\underline{\phi}_{i}^{j}(\omega)\leq\varphi_{i}\leq\overline{\phi}$, $i\in\mathbf{Z}$,
consider the equation
\begin{equation}\label{t6.3.6}
\left\{
   \begin{aligned}
    ((MI-A_{-p}^{\omega})\psi^{(k)})_{i}=(M+\tilde{c}_{i}-\varphi_{i})\varphi_{i},&\ \ i\in B(0,k)\\
    \psi^{(k)}_{\pm(k+1)}=0.
   \end{aligned}
   \right.
\end{equation}
\eqref{t6.3.6} possesses a unique solution $\psi^{(k)}=\{\psi^{(k)}_{i}\}_{i=-k-1}^{k+1}$
since $M$ can be chosen large enough.
Moreover, $0\leq\psi^{(k)}_{i}\leq\overline{\phi}$ for $i\in B(0,k)$.
The last inequality is obtained by using Claim 1 with $\overline{\phi}-\psi^{(k)}$ instead of $\phi^{(k)}$.
By using the diagonal extraction method one can find a solution $\psi\in X_{-\infty}$ of
$$((MI-A_{-p}^{\omega})\psi)_{i}=(M+\tilde{c}_{i}-\varphi_{i})\varphi_{i},\ \ i\in\mathbf{Z},$$
with $0\leq\psi_{i}\leq\overline{\phi}$ for $i\in\mathbf{Z}$. Using Claim 2 with
$\underline{\phi}^{j}(\omega)-\psi$ instead of $\phi$, one can obtain that
$\underline{\phi}_{i}^{j}(\omega)\leq\psi_{i}$ for $i\in\mathbf{Z}$.
Let $S^{j}=\{\phi\in X_{-\infty}|\ \underline{\phi}_{i}^{j}(\omega)\leq\varphi_{i}\leq\overline{\phi}\}$.
Define $T:S^{j}\rightarrow X_{-\infty}$ by $T\varphi=\psi$. Then $TS^{j}\subset S^{j}$.
Consider $\varphi^{1},\ \varphi^{2}\in S^{j}$ with $\varphi^{1}_{i}\leq\varphi^{2}_{i}$ for $i\in\mathbf{Z}$.
Then $T\varphi^{1}_{i}-T\varphi^{2}_{i}$ is a bounded solution of
$$((MI-A_{-p}^{\omega})\psi)_{i}=(M+\tilde{c}_{i}-\varphi^{1}_{i})\varphi^{1}_{i}-(M+\tilde{c}_{i}-\varphi^{2}_{i})\varphi^{2}_{i}.$$
Hence we have $T\varphi^{1}_{i}\leq T\varphi^{2}_{i}$ by the choice of $M$ and Claim 2.
Then by the super-subsolution method one can find a minimal solution $u=u_{i}(\omega)>0$ of $u=Tu$,
i.e., $L_{p}^{\omega}u_{i}=u^{2}_{i}$ for $i\in\mathbf{Z}$, in the class of all the solutions satisfying
$\underline{\phi}^{0}_{i}(\omega)\leq u_{i}(\omega)\leq\overline{\phi}$ for any
$(i,\omega)\in\mathbf{Z}\times\Omega_{0}.$

Take $j\in\mathbf{Z}$ and $v_{i}(\omega):=u_{i+j}(\pi_{-j}\omega)$. Then
$$\underline{\phi}_{i}^{0}(\omega)=\underline{\phi}_{i+j}^{0}(\pi_{-j}\omega)\leq
u_{i+j}(\pi_{-j}\omega)=v_{i}(\omega)\leq\overline{\phi}.$$
The stationarity of the coefficients yields that $v_{i}(\omega)$ satisfies $L_{p}^{\omega}v_{i}=v^{2}_{i}$ for
$i\in\mathbf{Z}$. Hence
$u_{i}(\omega)\leq v_{i}(\omega)=u_{i+j}(\pi_{-j}\omega)$ for any
$(i,j,\omega)\in{\mathbf{Z}}^{2}\times\Omega_{0}$ from the
minimality of $u_{i}(\omega)$.
Similarly, $u_{i+j}(\pi_{-j}\omega)\leq u_{i+j-j}(\pi_{j}\pi_{-j}\omega)=u_{i}(\pi_{j}\pi_{-j}\omega)$,
where $u_{i}(\pi_{j}\pi_{-j}\omega)$ is a minimal solution that satisfies
$\underline{\phi}_{i}^{0}(\pi_{j}\pi_{-j}\omega)\leq u_{i}(\pi_{j}\pi_{-j}\omega)\leq\overline{\phi}$.
On the other hand, $\underline{\phi}_{i}^{0}(\pi_{j}\pi_{-j}\omega)=\underline{\phi}_{i+j-j}^{0}(\omega)=
\underline{\phi}_{i}^{0}(\omega)\leq u_{i}(\omega)\leq\overline{\phi}$.
It yields that $u_{i}(\omega)=u_{i+j}(\pi_{-j}\omega)$ immediately.

Step 5: Since $L_{p}^{\omega}u=u^{2}$, one can find that $u>0$ and
${\Big\{\frac{{u}_{i\pm1}-{u}_{i}}{{u}_{i}}\Big\}}_{i=-\infty}^{\infty}\in{\ell}^{\infty}$ by Claims 1 and 2 in Step 3. The ergodic
theorem yields that there exists a measurable set $\Omega_{p}\subset\Omega_{0}$ with $\mathbb{P}(\Omega_{p})=1$
such that
$$\displaystyle{\lim_{i\rightarrow\pm\infty}}\frac{\ln u_{i}(\omega)}{i}$$
exist for any $\omega\in\Omega_{p}$. A similar argument to the proof of Theorem \ref{thm6.2} yields that these two limits are equal.
Note that $u$ is random stationary ergodic with $0<u_{i}(\omega)\leq\overline{\phi}$ for any $(i,\omega)$.
Thus these limits must be zero. Hence
$u(\omega)\in\mathcal{A}_{-\infty}$ for any $\omega\in\Omega_{p}$. Taking $u(\omega)$ as a test function
in the definition of $\underline{{\lambda}_{1}}(p,-\infty;\omega)$ one can find that
$\underline{{\lambda}_{1}}(p,-\infty;\omega)\geq0$ for any $\omega\in\Omega_{p}$. Now we have already
proved that for any fixed $p\in\mathbf{R}$, there exists $\Omega_{p}$ with $\mathbb{P}(\Omega_{p})=1$
such that $\underline{{\lambda}_{1}}(p,-\infty;\omega)\geq\Gamma_{\infty}$ for all $\omega\in\Omega_{p}$.
Let $\{p_{i}\}_{i=1}^{+\infty}=\mathbf{Q},\ \Omega^{\prime\prime}:={\bigcap}_{i=1}^{+\infty}\Omega_{p_{i}}$.
Then $\Omega^{\prime\prime}$ with $\mathbb{P}(\Omega^{\prime\prime})=1$.
Noting that $\underline{{\lambda}_{1}}(p,-\infty;\omega)$ is continuous in $p$,
we have $\underline{{\lambda}_{1}}(p,-\infty;\omega)\geq\Gamma_{\infty}$ for any $(p,\omega)\in\mathbf{R}\times\Omega^{\prime\prime}$.
\end{proof}


\begin{thebibliography}{99}
{\small
\bibitem {A1} Aronson D G, Weinberger H F. Multidimensional nonlinear diffusion arising in population genetics[J]. Advances in Mathematics, 1978, 30(1): 33-76.
\bibitem {BH}Berestycki H, Hamel F. Front propagation in periodic excitable media,
Communications on Pure and Applied Mathematics, 2002, 55: 0949-1032.
\bibitem {BHR} Berestycki H, Hamel F, Roques L.  Analysis of the periodically fragmented
environment model:II-Biological invasions and pulsating traveling fronts, J.
Math. Pures Appl., 2005, 84: 1101-1146.
\bibitem {BHN} Berestycki H, Hamel F, Nadirashvili N. The speed of propagation for KPP type problems. I: Periodic framework[J]. Journal of The European Mathematical Society, 2005, 7(2): 173-213.
\bibitem {BHN2}Berestycki H, Hamel F,  Nadirashvili N.The speed of propagation for KPP type problems. II - General domains. J. Amer. Math. Soc., 2010: 23:1-34.
\bibitem {B1} Berestycki H, Hamel F, Nadin G. Asymptotic spreading in heterogeneous diffusive excitable media[J]. Journal of Functional Analysis, 2008, 255(9): 2146-2189.
\bibitem {B12} Berestycki H, Hamel F, Rossi L. Liouville-type results for semilinear elliptic equations in unbounded domains[J]. Annali Di Matematica Pura Ed Applicata, 2007, 186(3): 469-507.
\bibitem {B2} Berestycki H, Nadin G. Spreading speeds for one-dimensional monostable reaction-diffusion equations[J]. Journal of Mathematical Physics, 2012, 53(11): 115619.
\bibitem {B3} Berestycki H, Nadin G. Asymptotic spreading for general heterogeneous Fisher-KPP type equations[J]. 2015, preprint.
\bibitem {B4} Berestycki H, Rossi L. On the principal eigenvalue of elliptic operators in $\ R^ N $ and applications[J]. Journal of the European Mathematical Society, 2006, 8(2): 195-215.
\bibitem{CLW} Cheng C P, Li W T, Wang Z C. Spreading speeds and travelling waves in a delayed population model with stage structure on a 2D spatial lattice[J]. IMA journal of applied mathematics, 2008, 73(4): 592-618.
\bibitem{CCV} Cahn J W, Chow S N, Van Vleck E S. Spatially discrete nonlinear diffusion equations[J]. Rocky Mount. J. Math., to appear, 1995.
\bibitem{CFG}Chen X, Fu S C, Guo J S. Uniqueness and asymptotics of traveling waves of monostable dynamics on lattices[J]. SIAM journal on mathematical analysis, 2006, 38(1): 233-258.

\bibitem {C1} Crandall M G, Ishii H, Lions P L. User's guide to viscosity solutions of second order partial differential equations[J]. Bulletin of the American Mathematical Society, 1992, 27(1): 1-67.

\bibitem{DL} Ding W, Liang X. Principal eigenvalues of generalized convolution operators on the circle and spreading speeds of noncompact evolution systems in periodic media. SIAM Journal on Mathematical Analysis, 2015, 47,855-896.
\bibitem {F1} Fisher R A. The wave of advance of advantageous genes[J]. Annals of eugenics, 1937, 7(4): 355-369.
\bibitem {GF} Gartner J, Freidlin M I , On the propagation of concentration waves in
periodic and random media, Sov. Math. Dokl., 1979, 20: 1282-1286.
\bibitem {G1} Guo J S, Hamel F. Front propagation for discrete periodic monostable equations[J]. Mathematische Annalen, 2006, 335(3): 489-525.
\bibitem {H1} Horn R A, Johnson C R. Matrix analysis[M]. Cambridge university press, 2012.
\bibitem {HZ} Hudson W, Zinner B. Existence of traveling waves for a generalized discrete Fisher's equation,
Comm. Appl. Nonlinear Anal. 1994,1: 23-46.
\bibitem{JZ}
Jin Y, Zhao X. Spatial dynamics of a discrete-time population model in a periodic lattice habitat.
 J. Dyn. Diff. Equat., 2009, 21,  501--525.
\bibitem {K1} Kolmogorov A N, Petrovsky I G, Piskunov N S.  Etude de l equation
de la diffusion avec croissance de la quantite de matiere et son application a
un probleme biologique. Moscow Univ. Math. Bull, 1937, 1: 1-25.
\bibitem{LZ}  Liang X, Zhao X. Spreading speeds and traveling waves for abstract
monostable evolution systems, J. Funct. Anal., 2010, 259: 857-903.
\bibitem {L1} Lions P L, Souganidis P E. Homogenization of degenerate second-order PDE in periodic and almost periodic environments and applications. Annales de l'IHP Analyse non linear. 2005, 22(5): 667-677.
\bibitem{MWZ} Ma S, Weng P, Zou X. Asymptotic speed of propagation and traveling wavefronts in a non-local delayed lattice differential equation[J]. Nonlinear Analysis: Theory, Methods and Applications, 2006, 65(10): 1858-1890.
\bibitem{MZ} Ma S, Zou X. Existence, uniqueness and stability of travelling waves in a discrete reaction-diffusion monostable equation with delay[J]. Journal of Differential Equations, 2005, 217(1): 54-87.
\bibitem{Mallet} Mallet-Paret J. The global structure of traveling waves in spatially discrete dynamical systems[J]. Journal of Dynamics and Differential Equations, 1999, 11(1): 49-127.
\bibitem {Na} Nadin G. The effect of the Schwarz rearrangement on the periodic principal eigenvalue of a nonsymmetric operator[J]. SIAM Journal on Mathematical Analysis, 2010, 41(6): 2388-2406.

\bibitem {N1} Nolen J. A central limit theorem for pulled fronts in a random medium[J]. NHM, 2011, 6(2): 167-194.
\bibitem {s1} Shen W. Variational principle for spreading speeds and generalized propagating
speeds in time almost periodic and space periodic KPP models, Trans.
Amer. Math. Soc., 2010, 362, 5125-5168.
\bibitem {s2} Shen W. Existence, uniqueness, and stability of generalized traveling waves in
time dependent monostable equations, J. Dynam. Differential Equations., 2011, 23: 1-44.
\bibitem {s3} Shen W. Existence of generalized traveling waves in time recurrent and space
periodic monostable equations, J. Appl. Anal. Comput.,  2011, 1: 69-93.
\bibitem{SKT} Shigesada N, Kawasaki K, Teramoto E. Traveling periodic waves in heterogeneous
environments. Theor. Popul. Biol.,1986, 30: 143-160.
\bibitem{WLW} Wang Z C, Li W T, Wu J. Entire solutions in delayed lattice differential equations with monostable nonlinearity[J]. SIAM Journal on Mathematical Analysis, 2009, 40(6): 2392-2420.
\bibitem {W}  Weinberger H. On spreading speeds and traveling waves for growth and
migration models in a periodic habitat, J. Math. Biol., 2002, 45: 511-548.
\bibitem {X}  Xin J. Existence of planar flame fronts in convective-diffusive periodic media,
Archive for rational mechanics and analysis, 1992, 121: 205-233.
\bibitem {Z} Zlatos A. Transition fronts in inhomogeneous Fisher-KPP reaction-diffusion equations[J]. Journal de mathšŠmatiques pures et appliqušŠes, 2012, 98(1): 89-102.
}
\end{thebibliography}
\end{document}